\documentclass[12pt]{amsart}
\usepackage{amsmath, amsthm, amsfonts, amssymb, bbm}
\usepackage{graphics}
\usepackage{xypic}
\usepackage[all]{xy}
\usepackage[english]{babel}

\usepackage{hyperref}
\usepackage[normalem]{ulem}

\usepackage{multirow}	
\usepackage[mathscr]{eucal}

\usepackage{tikz}

\newcounter{braid}
\newcounter{strands}
\pgfkeyssetvalue{/tikz/braid height}{1cm}
\pgfkeyssetvalue{/tikz/braid width}{1cm}
\pgfkeyssetvalue{/tikz/braid start}{(0,0)}
\pgfkeyssetvalue{/tikz/braid colour}{black}
\pgfkeys{/tikz/strands/.code={\setcounter{strands}{#1}}}

\makeatletter
\def\cross{%
  \@ifnextchar^{\message{Got sup}\cross@sup}{\cross@sub}}

\def\cross@sup^#1_#2{\render@cross{#2}{#1}}

\def\cross@sub_#1{\@ifnextchar^{\cross@@sub{#1}}{\render@cross{#1}{1}}}

\def\cross@@sub#1^#2{\render@cross{#1}{#2}}

\def\render@cross#1#2{
  \def\strand{#1}
  \def\crossing{#2}
  \pgfmathsetmacro{\cross@y}{-\value{braid}*\braid@h}
  \pgfmathtruncatemacro{\nextstrand}{#1+1}
  \foreach \thread in {1,...,\value{strands}}
  {
    \pgfmathsetmacro{\strand@x}{\thread * \braid@w}
    \ifnum\thread=\strand
    \pgfmathsetmacro{\over@x}{\strand * \braid@w + .5*(1 - \crossing) * \braid@w}
    \pgfmathsetmacro{\under@x}{\strand * \braid@w + .5*(1 + \crossing) * \braid@w}
    \draw[braid] \pgfkeysvalueof{/tikz/braid start} +(\under@x pt,\cross@y pt) to[out=-90,in=90] +(\over@x pt,\cross@y pt -\braid@h);
    \draw[braid] \pgfkeysvalueof{/tikz/braid start} +(\over@x pt,\cross@y pt) to[out=-90,in=90] +(\under@x pt,\cross@y pt -\braid@h);
    \else
    \ifnum\thread=\nextstrand
    \else
     \draw[braid] \pgfkeysvalueof{/tikz/braid start} ++(\strand@x pt,\cross@y pt) -- ++(0,-\braid@h);
    \fi
   \fi
  }
  \stepcounter{braid}
}

\tikzset{braid/.style={double=\pgfkeysvalueof{/tikz/braid colour},double distance=1pt,line width=2pt,white}}

\newcommand{\braid}[2][]{%
  \begingroup
  \pgfkeys{/tikz/strands=2}
  \tikzset{#1}
  \pgfkeysgetvalue{/tikz/braid width}{\braid@w}
  \pgfkeysgetvalue{/tikz/braid height}{\braid@h}
  \setcounter{braid}{0}
  \let\g=\cross
  #2
  \endgroup
}
\makeatother

\newcommand{\brw}{17}

\usepackage{color}

\makeatletter
\@addtoreset{equation}{section}
\makeatother

\def\begeqar{\begin{eqnarray}}
\def\endeqar{\end{eqnarray}}
\def\begeq{\begin{equation}}
\def\endeq{\end{equation}}

\def\wgta#1#2#3#4{\hbox{\rlap{\lower.35cm\hbox{$#1$}}
\hskip.2cm\rlap{\raise.25cm\hbox{$#2$}}
\rlap{\vrule width1.3cm height.4pt}
\hskip.55cm\rlap{\lower.6cm\hbox{\vrule width.4pt height1.2cm}}
\hskip.15cm
\rlap{\raise.25cm\hbox{$#3$}}\hskip.25cm\lower.35cm\hbox{$#4$}\hskip.6cm}}

\def\wgtb#1#2#3#4{\hbox{\rlap{\raise.25cm\hbox{$#2$}}
\hskip.2cm\rlap{\lower.35cm\hbox{$#1$}}
\rlap{\vrule width1.3cm height.4pt}
\hskip.55cm\rlap{\lower.6cm\hbox{\vrule width.4pt height1.2cm}}
\hskip.15cm
\rlap{\lower.35cm\hbox{$#4$}}\hskip.25cm\raise.25cm\hbox{$#3$}\hskip.6cm}}

\newcommand{\vv}{\mathsf{v}}
\newcommand{\ww}{\mathsf{w}}
\newcommand{\uu}{\mathsf{u}}

\newcommand{\one}{\boldsymbol{1}}

\newcommand{\ATL}[1]{\mathsf{T}^a_{#1}}
\newcommand{\fH}[1]{\mathsf{H}_{#1}}
\newcommand{\AH}[1]{\widehat{\mathsf{H}}_{#1}}

\newcommand{\TL}[1]{\mathsf{TL}_{#1}}

\newcommand{\rmi}{\mathrm{i}}
\newcommand{\q}{\mathfrak{q}}
\newcommand{\ffrac}[2]{\mbox{\footnotesize$\displaystyle\frac{#1}{#2}$}}
\newcommand{\half}{%
  \mathchoice{\ffrac{1}{2}}{\frac{1}{2}}{\frac{1}{2}}{\frac{1}{2}}}

\newcommand{\tensor}{\otimes}

\newcommand{\fus}{\times_{\!f}}
\newcommand{\afus}{\,\widehat{\times}_{\!f}\,}
\newcommand{\ahfus}{\,\widehat{\times}_{\!f}^{\mathsf{H}}\,}
\newcommand{\afusm}{\,\widehat{\times}^{-}_{\!f}\,}

\newcommand{\Hom}{\mathrm{Hom}}

\newcommand{\ass}{\alpha}
\newcommand{\aass}{\alpha}
\newcommand{\liso}{\ell}
\newcommand{\riso}{r}
\newcommand{\id}{\mathrm{id}}

\newcommand{\fuslim}{\otimes_{\catTLinf}}
\newcommand{\afuslim}{\otimes_{\catATLinf}}

\newcommand{\asslim}{\ass^{\catTLinf}}
\newcommand{\aasslim}{\ass^{\catATLinf}}

\newcommand{\brlim}{\br^{\catTLinf}}

\newcommand{\bAStTL}[2]{\overline{\mathcal{W}}_{#1,#2}}

\newcommand{\IrTL}[1]{\mathcal{X}_{#1}}
\newcommand{\PrTL}[1]{\mathcal{P}_{#1}}
\newcommand{\StTL}[1]{\mathcal{W}_{#1}}

\newcommand{\StJTL}[2]{\mathcal{W}_{#1,#2}}
\newcommand{\StATL}[1]{\mathcal{W}_{#1}}

\newcommand{\catTL}{\mathsf{C}}
\newcommand{\catTLinf}{\catTL_{\infty}}
\newcommand{\catTLev}{\catTLinf^{\mathrm{ev}}}
\newcommand{\catTLodd}{\catTLinf^{\mathrm{odd}}}

\newcommand{\catATL}{\widehat{\catTL}}
\newcommand{\catATLinf}{\widehat{\catTL}_{\infty}}
\newcommand{\catATLev}{\widehat{\catTL}_{\infty}^{\mathrm{ev}}}
\newcommand{\catATLodd}{\widehat{\catTL}_{\infty}^{\mathrm{odd}}}

\newcommand{\oN}{\mathbb{N}}
\newcommand{\oC}{\mathbb{C}}
\newcommand{\oZ}{\mathbb{Z}}

\newcommand{\modd}{\,\mathrm{mod}\,}
\newcommand{\g}{g}

\newcommand{\e}{\mathsf{\,e_{(N)}}}
\newcommand{\es}{\mathsf{e}}

\newcommand{\funLoc}{\mathcal{L}}
\newcommand{\funGl}{\mathcal{G}}
\newcommand{\funLocl}{\funLoc^l}
\newcommand{\funGll}{\funGl^{l}}

\newcommand{\afunLoc}{\widehat{\mathcal{L}}}
\newcommand{\afunGl}{\widehat{\mathcal{G}}}
\newcommand{\afunGll}{\afunGl^{l}}

\newcommand{\I}{\mathsf{I}}
\newcommand{\ue}{\omega}
\newcommand{\Ind}{\mathrm{Ind}}

\newcommand{\drawu}{
 	\draw[thick, dotted] (-0.05,0.5) arc (0:10:0 and -7.5);
 	\draw[thick, dotted] (-0.05,0.55) -- (3.25,0.55);
 	\draw[thick, dotted] (3.25,0.5) arc (0:10:0 and -7.5);
	\draw[thick, dotted] (-0.05,-0.85) -- (3.25,-0.85);
	\draw[thick] (0.3,0.5) arc (0:10:20 and -3.75);
	\draw[thick] (2.7,-0.81) arc (0:10.6:-30 and 3.75);

	\draw[thick] (0.9,0.5) arc (0:10:40 and -7.6);
	\draw[thick] (1.5,0.5) arc (0:10:40 and -7.6);
	\draw[thick] (2.1,0.5) arc (0:10:40 and -7.6);
	\draw[thick] (2.7,0.5) arc (0:10:40 and -7.6);
}

\newcommand{\stemb}{\iota}

\newcommand{\fk}{\mathbb{K}}
\newcommand{\VF}{\mathscr{FF}}

\newcommand{\VK}{\mathscr{K}}
\newcommand{\VP}{\mathscr{P}}
\newcommand{\Verma}{\mathscr{V}}

\newtheorem{Thm}[subsection]{Theorem}

\newtheorem{Lemma}[subsection]{Lemma}
\newtheorem{lemma}[subsubsection]{Lemma}
\newtheorem{Prop}[subsection]{Proposition}
\newtheorem{prop}[subsubsection]{Proposition}

\newtheorem{cor}[subsubsection]{Corollary}
\newtheorem{Conj}[subsection]{Conjecture}
\newtheorem{conj}[subsubsection]{Conjecture}

\newtheorem{Dfn}[subsection]{Definition}
\newtheorem{dfn}[subsubsection]{Definition}
\newtheorem{rem}[subsubsection]{Remark}

\begin{document}
\begin{flushright}
{\footnotesize
ZMP-HH/16-12\\
Hamburger Beitr\"age zur Mathematik 596
}
\end{flushright}

\title{}
\maketitle
\begin{center}
{\Large Fusion and braiding in finite and  affine Temperley--Lieb \\
\vskip0.2cm
categories}

\vskip 1cm

{\large A.M. Gainutdinov $^{a,b}$ and H. Saleur $^{c,d}$}

\vspace{1.0cm}

{\sl\small $^a$
DESY, Theory Group, Notkestrasse 85, Bldg. 2a, 22603 Hamburg, Germany\\}
\vskip0.1cm
{\sl\small $^b$  Laboratoire de Math\'ematiques et Physique Th\'eorique CNRS,\\
Universit\'e de Tours,
Parc de Grammont, 37200 Tours, 
France\\}
\vskip0.1cm
{\sl\small $^c$  Institut de Physique Th\'eorique, CEA Saclay,
Gif Sur Yvette, 91191, France\\}
\vskip0.1cm
{\sl\small $^d$ Department of Physics and Astronomy,
University of Southern California,\\
Los Angeles, CA 90089, USA\\}
\end{center}

\vskip1cm

\begin{abstract} 

Finite Temperley--Lieb algebras are  ``diagram algebra'' quotients of (the group algebra of) the famous Artin's braid group $\mathbb{B}_N$, while the affine Temperley--Lieb algebras arise as  diagram algebras from a generalized version of the braid group. 
We study  asymptotic ``$N\to\infty$'' representation theory of these quotients (parametrized by $\q\in\oC^{\times}$) from a perspective of braided monoidal categories.

Using certain idempotent subalgebras in the finite and affine algebras, we construct infinite ``arc'' towers of the diagram algebras and
the corresponding direct system of representation categories, with terms labeled by $N\in\oN$.

The corresponding direct-limit category is our main object of studies.

For the  case of the finite 
Temperley--Lieb algebras, we prove that the direct-limit category is abelian and highest-weight at any $\q\in\oC^{\times}$ and  endowed with braided monoidal structure.
The most interesting result is when $\q$ is a root of unity  where the representation theory is non-semisimple. The resulting braided monoidal categories we obtain at different roots of unity are new and interestingly they are not rigid. We observe then a fundamental relation of these categories to a certain representation category of the Virasoro algebra and give a conjecture on the existence of a braided monoidal equivalence between the categories.
 This should have powerful   applications to the study of the ``continuum'' 
  limit of critical statistical mechanics systems  
based on the Temperley--Lieb algebra. 

We also introduce a novel class of embeddings for the affine Temperley--Lieb algebras 
(algebras of diagrams drawn without intersections on the surface of a cylinder) 
and related new concept of fusion or bilinear $\oN$-graded tensor product of modules for these algebras. We prove that the fusion rules are stable with the index $N$ of the tower and prove that the corresponding  direct-limit category is endowed with an associative tensor product. We also study 
the braiding properties of this affine  Temperley--Lieb fusion. 
Potential relationship with representations of the product of two Virasoro algebras are left for future work.
\end{abstract}

\bigskip

\section{Introduction}

Let $\fk$ be a field and $\q$ is an invertible element in $\fk$. Graham and Lehrer introduced in~\cite{GL}  a sequence $\ATL{N}$ (with $N = 1, 2, 3, \dots$) of infinite-dimensional
algebras over $\fk$ as the sets of endomorphisms of the objects $N$ in a certain category of diagrams that depends on $\q$.
These algebras are called  affine  Temperley--Lieb and they are extended versions of the Temperley--Lieb quotients of the affine
Hecke algebras of type $\widehat{A}_{N-1}$. They have bases consisting of diagrams drawn without
intersections on the surface of a cylinder. A slightly different version of these algebras (the one without the translation generator) was  introduced much earlier by Martin and Saleur~\cite{MartinSaleur1} in the context of 2d statistical mechanics models such as loop models on a cylinder. 
One can cut the cylinder lengthwise and obtain then a much smaller algebra of diagrams on the strip which is the famous  (finite)    Temperley--Lieb algebra~\cite{TL}.


The categorical and statistical physics interpretations of the finite and affine  Temperley--Lieb (TL) algebras
made possible many interesting connections between different parts of mathematics (representation theory, knot theory, low-dimensional topology) and quantum physics. 
Recent proposals for ``fault tolerant'' topological quantum computers for instance --  which  may be realized experimentally using fractional quantum Hall devices -- are based on models of non-abelian anyons, whose Hilbert space, interactions and computational properties (e.g. qbit gates) are all expressed in terms of certain representations of the   Temperley-Lieb algebra (see e.g. \cite{Nayak, Freedman}, and~\cite{Wang} for a recent review.)

From a different, more physical point of view, it is a well established experimental fact that the properties of  two-dimensional statistical mechanics lattice models at their critical point~\cite{Baxter} 
 (e.g., those based on representations of the TL algebras  with $\q\in\oC^\times$ and when $|\q|=1$) can be described using conformal field theories -- that is, quantum  field theories which are invariant under conformal transformations~\cite{Schott, DfMS}.  This means, for instance, that expectation values  of  products of local physical observables such as the spin or the energy density can be identified with the Green functions  of the conformal field theory -- the latter being calculated using abstract algebraic manipulations  
based on the Virasoro algebra representation theory~\cite{DfMS,Gaw}. 

It is important to understand that  the identification between properties of the lattice model and of the quantum field theory can only hold in the so called {\sl scaling limit} or {\sl the continuum limit}.  This 
requires in particular that all scales (such as the size $N$ of the system) be much larger than the lattice spacing limit. In terms of the finite or affine Temperley-Lieb algebras, this means therefore that one needs to study $N\to\infty$. 

Understanding mathematically  the continuum limit of the two-dimensional statistical mechanics lattice models at their critical point  remains a challenge. Many works in the mathematical literature have concentrated on the existence of the  limit, and investigations of conformal invariance and what it may mean for models defined initially on a discrete lattice. Significant  progress  was obtained in some cases (such as the Ising model or percolation) spurred in part by developments around the Schramm--Loewner Evolution (see e.g. \cite{LSW,Smirnov}).
For a recent attempt at constructing the conformal stress-energy tensor or the Virasoro algebra on the lattice, see \cite{CGS,HJVK}.

A different, algebraic  route was pioneered by physicists~\cite{KooSaleur}, who investigated the potential relations between the Lie brackets  of the Virasoro algebra generators  and of certain elements in the associative algebras -- such as the Temperley--Lieb  algebra -- underlying the lattice models. This approach saw considerable renewed interest in the last few years  as it shows great promise in helping to understand logarithmic conformal field theories. For recent reviews and references about these theories --  from both mathematics and physics point of view  see the special volume \cite{special}. 
For works (in mathematical physics) on  relationships between  the Temperley--Lieb lattice models and logarithmic conformal field theories in the context of the continuum limit, see \cite{PRZ, ReadSaleur07-1, ReadSaleur07-2, GV, GRS1, GJSV, PRV,  GST, MDRR, Belletete, GRS3}.

\medskip

From the perspective of the continuum limit, we have thus an interesting problem of understanding the ``asymptotic'' representation theory of the finite and affine TL algebras when $N$ tends to infinity. There are at least two very distinct parts of the problem: (i) to define appropriately a tower of the algebras and understand its inductive limit as a certain infinite-dimensional operator algebra and (ii) to understand representation-theoretic aspects of the limits, i.e. to define and study towers of the representation categories of the algebras. The first problem (i) seems to be very hard because it requires constructing
inductive systems of algebras (and their representations) where the homomorphisms respect also the  $\oC$-grading by the Hamiltonian of the lattice model, which is a particular element of the algebra (e.g. TL). Such systems 
are designed to control the growth of the algebras as the number $N$ of sites tends to infinity and to keep only those states and operators that are eigenstates of the Hamiltonian  with  {\sl finite}  eigenvalue in the limit\footnote{The idea here comes from physics: one should keep at $N\to\infty$ only relevant ``physical''  excitations -- finite-energy states -- above the ground states (that is, the Hamiltonian's eigenstates of  minimum eigenvalue). Note that the Hamiltonian in this discussion has to be properly normalized.}. By the definition, the inductive limit of such an inductive system gives what physicists call \textit{the continuum limit} of lattice models and we use this term as well.  
Unfortunately, our understanding of the spectrum and eigenstates of the Temperley--Lieb Hamiltonians is very limited, so only rather simple cases within this program were appropriately rigorously treated~\cite{GRS3} (see also~\cite{GST})
and have indeed shown an explicit connection with the conformal field theory on the operator level. 

The main goal of our paper is to study the second point (ii) or the ``categorical'' part of the asymptotic representation theory of the finite and affine TL algebras. 
Using certain idempotent subalgebras in the finite and affine cases, we construct infinite \textit{arc}-towers of the diagram algebras and
the corresponding direct system of representation categories, with terms labeled by $N\in\oN$.
In more details for the finite TL case, we introduce the ``big'' representation category $\bigoplus_N \TL{N}\!-\!\mathrm{mod}$  encapsulating all representations of the TL algebras for all $N$, and a fixed $\q\in\oC^{\times}$. We then  define  naturally the tensor product (bi-functor) in the category  using the induction:
it  is endowed with an \textit{associator} and \textit{braiding} isomorphisms that we prove do satisfy the coherence conditions (the pentagon and hexagon identities).  

Inside the  $\mathbb{N}$-graded category of finite TL representations,
we  construct the ``shift'' functors $\mathcal{F}_N: \TL{N}\!-\!\mathrm{mod}\longrightarrow \TL{N+2}\!-\!\mathrm{mod}$ associated with  the arc-towers of  the  TL algebras. These functors have a nice property: they map the standard modules to the standard ones of the same weight, projective covers to projective covers, etc.
The inductive or direct system of the TL-representation categories thus defined has also certain nice and important properties: the associators are mapped to associators and braiding isomorphism to braiding isomorphism. We then prove that all these structures endow the corresponding direct  limit  with a braided monoidal category structure, which we denote as $\catTLinf$. This construction can be considered as a machinery to prove that the TL fusion rules are stable with $N$ for any $\q\in\oC^\times$. When $\q$ is generic, the category $\catTLinf$ is semi-simple and have left and right duals. Interestingly, the construction is well defined even for $\q$ a root of unity, which is the most interesting case for the applications, in particular 
to low-dimensional topology and  logarithmic conformal field theory. In the root of unity case, we prove that the direct-limit categories we obtain are abelian and  highest-weight. They have  a {\sl non-simple} tensor unit and in particular not all objects have duals, so the categories are not rigid when $\q$ is a root of unity.

We also give a rather explicit description of objects in the category $\catTLinf$ (subquotient structure of standard and projective objects) and describe the fusion rules. It turns out that this data matches  what we know about certain category of the Virasoro algebra representations at critical central charges, 
 so we have in particular an equivalence of  abelian categories, described in Prop.~\ref{prop:equiv-1}.  We formulate (for the first time) an explicit conjecture (in Conj.~\ref{conj:equiv-2}) that the two categories are equivalent as braided monoidal categories and give few supporting arguments.

This result provides a mathematical framework to understand 
the observations in \cite{ReadSaleur07-2,GV}  that the representation theory of the Virasoro algebra relevant to  the description of the scaling limit of statistical models  and the representation theory of the Temperley-Lieb algebra occurring in the lattice formulation of these models ``are very similar'' (this observation has spurred many developments in the analysis of Virasoro and Temperley--Lieb-like algebras modules -- see e.g.~\cite{GJSV, RSA, BRSA}). 

\medskip
Note that, since we are dealing with the ordinary Temperley--Lieb algebra, the statistical models here have ``open boundary conditions'', which is known to correspond, in the scaling limit, to boundary conformal field theory, involving a single Virasoro algebra \cite{Cardybdr}.
Meanwhile, the physics of critical statistical lattice models away from their boundaries (the so called ``bulk'' case)  is described by two copies of the Virasoro algebra, corresponding to the chiral and anti-chiral dependencies of the correlation functions. The ideas in \cite{ReadSaleur07-1} can then be extended, and involve now, on the lattice side, models with periodic boundary conditions, for which the correct algebra is now  the corresponding affine Temperley--Lieb algebra.

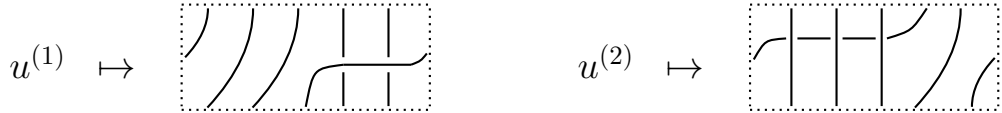
\begin{figure}
\begin{equation*}
 \begin{tikzpicture}
 \node[font=\large]  at (-1.5,-0.2) {\mbox{} $u^{(1)}\;\;\mapsto$ \mbox{}\qquad};
 	\draw[thick, dotted] (-0.05,0.5) arc (0:10:0 and -7.5);
 	\draw[thick, dotted] (-0.05,0.55) -- (3.25,0.55);
 	\draw[thick, dotted] (3.25,0.5) arc (0:10:0 and -7.5);
	\draw[thick, dotted] (-0.05,-0.85) -- (3.25,-0.85);

	\draw[thick] (0.3,0.5) arc (0:10:20 and -3.75);
	\draw[thick] (0.9,0.5) arc (0:10:40 and -7.6);
	\draw[thick] (1.5,0.5) arc (0:10:40 and -7.6);
	

	\draw[thick] (1.6,-0.81) .. controls (1.7,-0.3)   .. (2.1,-0.25);
		\draw[thick] (2.1,-0.25) -- (3,-0.25);
	\draw[thick] (2.1,0.5) -- (2.1,-0.15);
	\draw[thick] (2.1,-0.35) -- (2.1,-0.8);
	\draw[thick] (2.7,0.5) -- (2.7,-0.15);
	\draw[thick] (2.7,-0.35) -- (2.7,-0.8);
		\draw[thick] (3,-0.25) .. controls (3.13,-0.2)   .. (3.21,-0.1);
\end{tikzpicture}
\qquad\qquad
 \begin{tikzpicture}
 \node[font=\large]  at (-1.5,-0.2) {\mbox{} $u^{(2)}\;\;\mapsto$ \mbox{}\qquad};
 
 	\draw[thick, dotted] (-0.05,0.5) arc (0:10:0 and -7.5);
 	\draw[thick, dotted] (-0.05,0.55) -- (3.25,0.55);
 	\draw[thick, dotted] (3.25,0.5) arc (0:10:0 and -7.5);
	\draw[thick, dotted] (-0.05,-0.85) -- (3.25,-0.85);
	\draw[thick] (0.0,-0.18) .. controls (0.15,0.08)   .. (0.43,0.1);
	\draw[thick] (0.58,0.1) -- (1.02,0.1);
	\draw[thick] (1.18,0.1) -- (1.62,0.1);
	\draw[thick] (1.77,0.1) .. controls (2.1,0.2)   .. (2.3,0.5);

	\draw[thick] (2.9,-0.81) arc (0:10:-20 and 3.75);

	\draw[thick] (0.5,0.5) -- (0.5,-0.8);
	\draw[thick] (1.1,0.5) --  (1.1,-0.8);
	\draw[thick] (1.7,0.5) -- (1.7,-0.8);
	\draw[thick] (2.75,0.5) arc (0:10:40 and -7.6);

\end{tikzpicture}
\end{equation*}
\caption{The map of the two translation generators $u^{(1)}\in \ATL{N_1}$ and $u^{(2)}\in\ATL{N_2}$, with $N_1=3$ and $N_2=2$, into $\ATL{N_1+N_2}$ in terms of affine TL diagrams where each crossing/braiding at $i$-th site has to be replaced by a linear combination of identity and $e_i$.}
\label{intro:fig1}
\end{figure}

Our second result is concentrated on the limit of affine TL representation categories and on a new tensor-product or fusion in the enveloping category (where we call it \textit{the affine TL fusion}) and eventually in the direct-limit category.  The affine TL fusion is based on a non-trivial embedding of two affine or periodic TL algebras for $N_1$ and $N_2$ into the ``big'' algebra $\ATL{N_1+N_2}$ on $N_1+N_2$ sites.
 Note that there is an obvious embedding in the finite TL case while in the periodic or affine case it is not a priori clear how to embed the product of two periodic TL algebras into another periodic TL algebra. We introduce a novel diagrammatic way for this in Sec.~\ref{AffTLembed} and  describe here shortly only  the embedding of the two translation generators $u^{(1)}$ and $u^{(2)}$, as in Fig.~\ref{intro:fig1}. (We have thus constructed a novel tower of the affine TL algebras.) The affne TL fusion for the affine TL modules is then defined as the induction from such subalgebra $\ATL{N_1}\otimes\ATL{N_2}$ to  $\ATL{N_1+N_2}$.
We then use the machinery elaborated in the finite TL case (in Sec.~\ref{limTLcat}: note  it is actually a very general construction that can be applied to a relatively  general class of diagram algebras) and prove that the affine TL fusion rules are stable with $N$ and the tensor product is associative.

Like in  the finite TL case,  inside the  $\mathbb{N}$-graded category of affine TL representations,
we  construct the ``shift'' functors $\mathcal{F}_N: \ATL{N}\!-\!\mathrm{mod}\longrightarrow \ATL{N+2}\!-\!\mathrm{mod}$
corresponding to what we call  the arc-towers of the  affine TL algebras and the corresponding direct limit that is denoted by $\catATLinf$. We have proven that this limit has a tensor product with an associator but there is no braiding (we have computed some of the fusion rules and they are apparently non-commutative). However, we have shown that there is a semi-braiding connecting `chiral' and `anti-chiral' tensor products, the one defined above and the other defined similarly but interchanging over-crossings with under-crossings.

This periodic or  ``bulk'' case is much harder  than the ``open-boundary'' case, and less explored, despite some beautiful results on some simple cases \cite{GRS1,GRS2,GRS3,DPR,DSA,GRSV1}. A key ingredient that was missing so far was how to define for lattice models a fusion that may correspond, in the continuum limit, to the fusion of non chiral fields. The results of the present paper provide a potential way to do this, which will be explored more in our subsequent work \cite{GJS}.

\medskip

The paper is organized as follows. We first recall in section  \ref{FusTLcat} the definition and well-known facts about the finite Temperley--Lieb (TL) algebra, including the fusion of TL modules, then  introduce a tower of finite TL algebras that leads to the concept of an ``enveloping'' TL representation category endowed with $\oN$-graded bilinear tensor product  -- an object which, we  show later, is crucial for comparison with results from physics. We introduce then a novel  tower for the affine TL algebras in section \ref{AffTLembed} where we also introduce our diagrammatic fusion for affine TL modules. This fusion is further considered in section \ref{FusAffTL} where a connection to  the affine Hecke algebra is discussed. We start in section \ref{limTLcat} our discussion of limits of TL representation categories and introduce a braided monoidal structure on the direct limit of the TL categories, with the main result in Thm.~\ref{thm:catTLinf}. This section is rather technical and long but provides a general approach for studying the limits associated with towers of algebras. This approach is then used  in section \ref{semiGaffTL} in the affine TL case where our main result is formulated in Thm.~\ref{thm:catATLinf}. We conclude in section \ref{sec:outlook} by a discussion of the relation of our work with the Virasoro algebra representation theory, Prop.~\ref{prop:equiv-1}, and write down a fundamental conjecture~\ref{conj:equiv-2} about an equivalence of braided tensor categories.  In Appendix~A, we give several  examples of the diagrammatic calculation of affine TL fusion rules.

\medskip

\noindent{\bf Conventions:} 
Throughout the paper we fix the field $\fk=\oC$ for  convenience, though most of our results hold for any field $\fk$ (we need the characteristic zero only when we discuss the subquotient structure of the TL representations). 
We also denote by $\oN$ the additive semi-group of positive integers $\{1,2,3, \ldots\}$. 

\bigskip

\noindent
{\bf Acknowledgements:}
We are grateful to Jesper L. Jacobsen for fruitful discussions and general interest in this work. This work was supported in part by the ERC Advanced Grant NuQFT. 
The work of AMG was  also supported by a Humboldt fellowship,  DESY and CNRS. AMG  thanks IPHT/Saclay for its kind hospitality during last few years where a part of this work was done.
 AMG also thanks the organizers of the workshop \textit{Conformal Field Theories and Tensor categories}, International Center for Mathematical Research, Beijing, 2011 where some of our results (those on the direct limit of finite TL categories and the relation to the Virasoro algebra) were presented.

\section{Fusion in TL-mod categories }\label{FusTLcat}

\textit{The (finite) Temperley--Lieb (TL) algebra} $\TL{N}(m)$ is an associative algebra over $\oC$ generated by unit $\one$ and $e_j$, with $1\leq j\leq N-1$, satisfying  the defining relations
\begin{eqnarray}\label{TL-rel}
e_j^2&=&me_j,\nonumber\\
e_je_{j\pm 1}e_j&=&e_j,\\
e_je_k&=&e_ke_j\qquad(j\neq k,~k\pm 1).\nonumber
\end{eqnarray}
This algebra has a well-known faithful diagrammatical representation in terms of non-crossing pairings on a rectangle with $N$ points on each of the opposite sides. Multiplication is performed by placing two rectangles on top of each other, and replacing any closed loops by a factor $m$. While the identity corresponds to the diagram in which each point is directly connected to the point above it, the generator $e_i$ is represented by the diagram, see  Fig.~\ref{fig:TLe}, where the points $i$  on both sides of the rectangle are connected to the point $i+1$ on the same side, all other points being connected like in the identity diagram. The defining relations are easily checked by using isotopy ambient on the boundary of the rectangle, see Fig.~\ref{fig:TLrel}. 

\begin{figure}
\begin{equation*}
\begin{tikzpicture}
\node at (-9.,-0.2) {$e_i\quad =$};
\draw[thick] (-5,-0.8) -- (-5,0.5);
\draw[thick] (-7,-0.8) -- (-7,0.5);
\node at (-6,-0.2) {$\ldots$};
\node at (-7.,-1.2) {$1$};
\node at (-4.2,-1.2) {$i$};
\node at (-3.2,-1.2) {$i+1$};
	\draw[thick] (-4.2,0.5) arc (-180:0:0.5 and 0.45);


	\draw[thick] (-3.2,-0.8) arc (0:180:0.5 and 0.45);

\node at (-2,-0.2) {$\ldots$};
\draw[thick] (0,-0.8) -- (0,0.5);
\node at (0.,-1.2) {$N$};
	\end{tikzpicture}
\end{equation*}
\caption{The diagrammatic representation of $e_i$.}
\label{fig:TLe}
\end{figure}
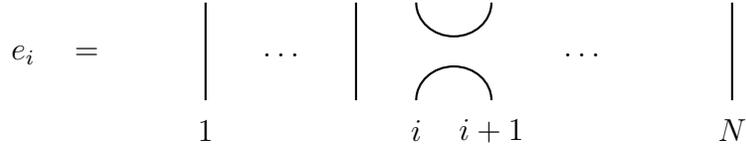
\begin{figure}
\begin{equation*}
\begin{tikzpicture}
	\draw[thick] (0.8,0.5) arc (-180:0:0.5 and 0.45);
	\draw[thick] (1.8,-0.8) arc (0:180:0.5 and 0.45);
	\draw[thick] (1.8,1.8) arc (-180:0:0.5 and 0.45);
	\draw[thick] (2.8,0.5) arc (0:180:0.5 and 0.45);
	
	\draw[thick] (0.8,0.5) -- (0.8,1.8);
\draw[thick] (0.8,3.1) arc (-180:0:0.5 and 0.45);
	\draw[thick] (1.8,1.8) arc (0:180:0.5 and 0.45);
	\draw[thick] (2.8,-0.9) -- (2.8,0.5);
	\draw[thick] (2.8,1.8) -- (2.8,3.1);
	\node at (4.,1.2) {$ =$};
	\draw[thick] (5.2,1.8) arc (-180:0:0.5 and 0.45);


	\draw[thick] (6.2,0.5) arc (0:180:0.5 and 0.45);

		\end{tikzpicture}
\end{equation*}
\caption{The diagrammatic version of the relation $e_ie_{i+1}e_i=e_i$.
}
\label{fig:TLrel}
\end{figure}
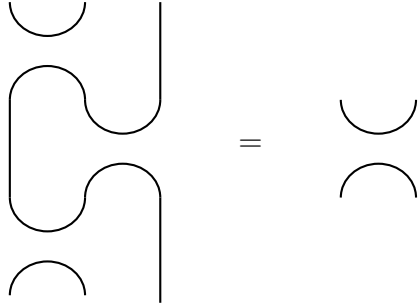

We will often omit mentioning the parameter $m$ and write simply $\TL{N}$ as the replacement for $\TL{N}(m)$.

\subsection{Towers of the TL algebras}\label{sec:TL-towers}
The important ingredient of our constructions below are towers of the TL algebras. In terms of the diagrams, we can naturally construct two kinds of towers.
\begin{itemize}
\item
 The first one is standard, it uses
 \textit{the standard embeddings} of the algebras:
\begin{equation}\label{st-emb-0}
\TL{N} \xrightarrow{\quad \stemb \quad} \TL{N+1}
\end{equation}
such that
\begin{equation}\label{st-emb}
\iota(e_j) = e_j,\qquad 1\leq j \leq N-1,
\end{equation}
or in terms of the TL diagrams one adds one vertical string on the right of the diagram considered as an element in $\TL{N}$ -- this gives an element in $\TL{N+1}$.
\item
The second tower uses what we call the arc-embeddings: a diagram on $N$ sites is enlarged up to the diagram on $N+2$ sites by adding arcs, see precise definition below in Sec.~\ref{sec:arc-tower}. 
\end{itemize}
The first type of TL towers 
is used in the definition of the TL fusion, while the second is used in constructing direct limits of the TL representation  categories.

\subsection{TL fusion}
In this section, we recall  \textit{fusion} for modules over the (finite) TL algebras initially introduced in the physics literature~\cite{ReadSaleur07-1} and further studied on more formal grounds in~\cite{GV}, which we follow in terms of conventions and notations. The fusion's definition is based on the standard embeddings~\eqref{st-emb}. 
We then  use this fusion construction to define $\oN$-graded tensor-product structure (with an associator) on the direct sum of the categories of TL representations.

Let $\catTL_N$ denotes the category of finite-dimensional $\TL{N}$-modules (we will usually drop the  parameter $m$ for brevity.) 

\begin{dfn}[\cite{ReadSaleur07-1,GV}]\label{defRSfusion}
Let $M_1$ and $M_2$ be two modules over $\TL{N_1}$ and $\TL{N_2}$ respectively.
Then, the tensor product $M_1\tensor M_2$ is a module over
 the product $\TL{N_1}\tensor\TL{N_2}$ of the two algebras.
Using the standard  embedding, we consider this
 product of algebras as a subalgebra in $\TL{N}$, for $N=N_1+N_2$.
\textit{The fusion (bi-)functor}
\begin{equation}\label{fusfunc-TL-def-0}
\fus: \quad \catTL_{N_1}\times \catTL_{N_2}\to \catTL_{N_1+N_2}
\end{equation}
 on two modules $M_1$ and $M_2$  is then defined as the  module induced from this subalgebra, {\it i.e.,}
\begin{equation}\label{fusfunc-TL-def}
M_1\fus M_2 = \TL{N}\tensor_{\bigl(\TL{N_1}\tensor\TL{N_2}\bigr)} M_1\tensor M_2 \ ,
\end{equation}
where we used the balanced tensor product over $\TL{N_1}\tensor\TL{N_2}$.
\end{dfn}

An explicit calculation of the TL fusion for a large class of indecomposable representations at any non-zero $\q$ (i.e. including the root of unity cases) is given in~\cite{GV}, 
see also recent works~\cite{RSA,Belletete,BRSA}.
 
 \medskip 

We emphasize that the TL fusion is associative, i.e. the functor $\fus$ in~\eqref{fusfunc-TL-def-0} is equipped with a family of natural  isomorphisms 
$\bigl(M_1\fus M_2\bigr) \fus M_3 \cong
M_1\fus \bigl( M_2 \fus M_3\bigr)$, for each triple of $\TL{N_i}$ modules $M_i$, $i=1,2,3$, and each triple of the natural numbers  $(N_1,N_2,N_3)$.
We leave the proof of the associativity till Sec.~\ref{sec:ass-TL} and give the associator explicitly in Prop.~\ref{prop:ass}
and prove the pentagon identity in Prop.~\ref{prop:ass-pentagon}.
On top of it, we also have the braiding, i.e. the tensor product $\fus$ is commutative, see Sec.~\ref{sec:braiding-TL} below.

It is then natural to introduce a large category embracing these structures. We call it the ``enveloping'' TL representation category (that formally contains TL representations at any~$N$)
\begin{equation}
\catTL = \bigoplus_{N\geq 1} \catTL_{N}\ .
\end{equation}
The direct sum here means that $\catTL$ contains $\catTL_{N}$ as a full subcategory and there are no morphisms between the full subcategories for different $N$. The category $\catTL$ is thus graded by $\oN$. We will label  an object $M$ from $\catTL_N$ as $M[N]$ to emphasize its grade. We can thus consider $\fus$ defined in~\eqref{fusfunc-TL-def} as an $\oN$-graded  tensor-product functor on $\catTL$:

\begin{prop}\label{prop:N-graded}
Let $\fus$  denote the $\oN$-graded bilinear tensor product on $\catTL$ as defined for each pair $(N_1,N_2)\in\oN\times\oN$ in~\eqref{fusfunc-TL-def-0} and~\eqref{fusfunc-TL-def}. It is equipped with an associator satisfying the pentagon identity. 
\end{prop}

Note that we do not have the tensor unit in~$\catTL$, as we do not have the zero grade.
Let us call such a category (a monoidal category weaken by removing the tensor unit axioms) as a \textit{semi-group category}.

\medskip

For the later use, we will need the structure of  a highest-weight category on $\catTL_{N}$ and we recall below few standard facts from the representation theory of $\TL{N}(\q+\q^{-1})$.

\subsection{Standard and projective TL modules}\label{sec:CN}
We also recall~\cite{TL} \textit{the standard} $\TL{N}(m)$ modules $\StTL{j}[N]$ of  weight $x \leq j\leq N/2$, where $x =\half (N\,\mathrm{mod}\, 2)$. First, we need to introduce ``half-diagrams'' (usually called link states) obtained from Temperley-Lieb diagrams (i.e.,  non-crossing pairings on a rectangle with $N$ points on each of the opposite sides) by cutting these diagrams horizontally in the middle. Each half has $N$ points: some of them are connected by arcs, and some others are not connected to anything. The latter are often called \textit{through-lines} (or defects). The algebra acts in the obvious diagrammatic way by concatenating Temperley-Lieb diagrams with link states, eliminating  all  loops at the  price of multiplying the diagram by $m^{n}$, where $n$ is the number of loops, and keeping track of the connectivities using isotopy. It is clear that the number of through-lines cannot increase under the action of the algebra. Standard modules $\StTL{j}[N]$ are obtained by letting the algebra act as usual  when the number of through-lines  -- denoted by $2j$ -- is conserved, and setting this action to zero when the number of through-lines decreases. It is well known that these modules are irreducible for $\q$ generic~\cite{TL}, while their dimension is given by differences of binomial coefficients 
\begin{equation}\label{dj}
d_j[N]=\left(\begin{array}{c} N\\{N\over 2}+j\end{array}\right)-\left(\begin{array}{c} N\\{N\over 2}+j+1\end{array}\right).
\end{equation}

It is well-known also~\cite{Westbury} that the category of finite-dimensional $\TL{N}$ modules is a highest-weight category, i.e., it has a special class of objects -- the standard modules $\StTL{j}[N]$ -- with morphisms only in one direction (there is a homomorphism from $\StTL{k}[N]$ to $\StTL{j}[N]$ only if $k\geq j$) and the projective and invective modules have a nice filtration with sections given by the standard and costandards modules, respectively,
see more details in~\cite{DlRin} 
. When $\q$ is generic (not a root of unity) then the category is  semi-simple and the structure of a highest-weight category is trivial. 

Let $\q=e^{\rmi\pi/p}$ for integer $p\geq 2$ and set 
$s\equiv s(j)=(2j+1)\;\mathrm{mod}\; p$, we then recall the subquotient structure of $\StTL{j}[N]$,  
which was studied in many works from different perspectives~\cite{RSA, GV,Westbury,M1}  (though we use slightly different conventions here). 
The standard modules with $s(j)=0$ are simple. For non-zero $s(j)$
there is a non-trivial homomorphism  from $\StTL{k}[N]$ to $\StTL{j}[N]$ only if $k=j$ or $k=j+p-s$, and in the second case:
$$
\phi_j:\quad \StTL{j+p-s}[N]\; \to \; \StTL{j}[N]
$$
with $\ker \phi_j$ given by the socle of $\StTL{j+p-s}[N]$ and $\mathrm{im} \, \phi_j$ is the socle of $\StTL{j}[N]$ and is isomorphic to the head of $\StTL{j+p-s}[N]$.
Note that $\phi_j$ exists only if $2(j+p-s)\leq N$. We thus have the subquotient diagram (for non-zero $s$) 
$$
\StTL{j}[N] \; = \; \IrTL{j}[N] \longrightarrow \IrTL{j+p-s}[N]
$$
where we introduce the notation $\IrTL{j}[N]$ for the irreducible quotient of $\StTL{j}[N]$ and set here $\IrTL{k}[N]\equiv0$ if $2k>N$.

Let $\PrTL{j}[N]$ denotes the projective cover  of the simple module $\IrTL{j}[N]$.
The subquotient structure of the projective covers can be easily deduced due to a reciprocity relation
for a highest-weight category. Let $[\StTL{j}:\IrTL{j'}]$ and  $[\PrTL{j}:\StTL{j'}]$ denote the number of appearances of $\IrTL{j'}[N]$ in the subquotient diagram for $\StTL{j}[N]$ and the number of appearances of $\StTL{j'}[N]$ in the standard filtration for the projective cover $\PrTL{j}[N]$, respectively. Then, the reciprocity relation reads
\begin{equation}
[\PrTL{j}:\StTL{j'}] = [\StTL{j'}:\IrTL{j}],
\end{equation}
or, in words, the projective
modules $\PrTL{j}$ 
are composed of the standard modules that have the irreducible module
$\IrTL{j}$ as a subquotient.
The projective covers $\PrTL{j}[N]$ are then simple if $s(j)=0$, they are equal to $\StTL{j}[N]$ for $0\leq j\leq \half(p-2)$ and otherwise have the following structure
\begin{align}\label{prTL-pic}
   \xymatrix@C=5pt@R=25pt@M=2pt{%
    &&\\
    &\PrTL{j}[N]\quad = &\\
    &&
 }
&  \xymatrix@C=20pt@R=20pt{%
    &&{\IrTL{j}[N]}\ar[dl]\ar[dr]&\\
    &\IrTL{j-s}[N]\ar[dr]&&\IrTL{j+p-s}[N]\ar[dl]\\
    &&\IrTL{j}[N]&
 } 
\end{align}
where the nodes are irreducible subquotients and the arrows correspond to $\TL{N}$ action, i.e., $\IrTL{j}[N]$ in the bottom of the diagram is the irreducible submodule (or socle), while the socle of the quotient $\PrTL{j}[N]/\IrTL{j}[N]$ is the direct sum $\IrTL{j-s}[N]\oplus\IrTL{j+p-s}[N]$ in the middle of the diagram, etc. Also note that the diagram~\eqref{prTL-pic} has three nodes instead of four if $2(j+p-s)>N$.

\section{Affine Temperley--Lieb embedding}\label{AffTLembed}

We have seen in the previous section  embeddings
of the finite Temperley--Lieb algebras  $\TL{N_1}\to\TL{N}$, for $N_1<N$, that are naturally defined in terms of TL diagrams by adding the vertical strings,
or by the use of the standard embeddings $\TL{N}\to\TL{N+1}$ repeatedly. This standard embedding is the basic step in the definition of the TL fusion.
Constructing embeddings of periodic or affine TL algebras is a non-trivial problem that we solve in this section. We first recall the  definition of the affine TL algebras (also parametrized by $N\in\oN$) and then propose a novel diagrammatical way of 
defining a tower of these algebras.

\subsection{The affine Temperley--Lieb algebras}
\label{sec:TL-alg-def}

We recall here two equivalent definitions of the affine Temperley--Lieb algebra - independently introduced  and studied in many works~\cite{MartinSaleur,Jones,MartinSaleur1,GL,Green}.
We follow mainly conventions and notations from the work of Graham and Lehrer \cite{GL,GL1}
 whenever
possible.

\subsubsection{Definition I: generators and relations}
\textit{The affine Temperley--Lieb (aTL) algebra} $\ATL{N}(m)$ is an associative algebra over $\oC$ generated by $u$, $u^{-1}$ , and $e_j$, with $j\in \oZ/N\oZ$, satisfying  the defining relations
\begin{eqnarray}
e_j^2&=&me_j,\nonumber\\
e_je_{j\pm 1}e_j&=&e_j,\label{TL}\\
e_je_k&=&e_ke_j\qquad(j\neq k,~k\pm 1),\nonumber
\end{eqnarray}
which are the standard TL relations  but defined now  for  indices modulo $N$, and
\begin{eqnarray}
ue_ju^{-1}&=&e_{j+1},\nonumber\\
u^2e_{N-1}&=&e_1\ldots e_{N-1},\label{TLpdef-u2}
\end{eqnarray}
where  the indices $j=1, \ldots, N$ are  again interpreted modulo $N$.

\subsubsection{Definition II: diagrammatic}

 \textit{The affine Temperley--Lieb (TL) algebra} $\ATL{N}(m)$ is an associative algebra over $\oC$ spanned by particular diagrams
on an annulus
with  $N$ sites on the inner and $N$ on
the outer boundary.
 The sites are connected in pairs,
and only configurations that can be represented using lines inside the
annulus that do not cross are allowed. Diagrams related by an isotopy
leaving the labeled sites fixed are considered equivalent. We call such
(equivalence classes of) diagrams  \textit{affine} diagrams.   Examples of affine diagrams are shown in Fig.~\ref{fig:aff-diag} for $N=4$, where we draw them in a slightly different geometry: 
we cut the annulus and transform it to a rectangle which we call \textit{framing} so that the sites labeled by `$1$' are closest to the left and sites labeled by `$N$' are to the right sides of the rectangle. 
 Multiplication $a\cdot b$ of two affine diagrams $a$ and $b$ is defined in a natural
way, by joining an inner boundary of $a$ to an outer boundary of the annulus of $b$, and
removing the interior sites.
Whenever a closed contractible loop is
produced when diagrams are multiplied together, this loop must be
replaced by a numerical factor~$m$ that we often parametrize by $\q$ as $m=\q+\q^{-1}$. 

\begin{figure}
\begin{equation*}
 \begin{tikzpicture}
 	\draw[thick, dotted] (-0.05,0.5) arc (0:10:0 and -7.5);
 	\draw[thick, dotted] (-0.05,0.55) -- (2.65,0.55);
 	\draw[thick, dotted] (2.65,0.5) arc (0:10:0 and -7.5);
	\draw[thick, dotted] (-0.05,-0.85) -- (2.65,-0.85);
	\draw[thick] (0.3,0.5) arc (0:10:20 and -3.75);
	\draw[thick] (2.3,-0.81) arc (0:10:-20 and 3.75);

	\draw[thick] (0.9,0.5) arc (0:10:40 and -7.6);
	\draw[thick] (1.55,0.5) arc (0:10:40 and -7.6);
	\draw[thick] (2.2,0.5) arc (0:10:40 and -7.6);

	\end{tikzpicture}\;\;,
	\qquad
 \begin{tikzpicture}
 	\draw[thick, dotted] (-0.05,0.5) arc (0:10:0 and -7.5);
 	\draw[thick, dotted] (-0.05,0.55) -- (2.65,0.55);
 	\draw[thick, dotted] (2.65,0.5) arc (0:10:0 and -7.5);
	\draw[thick, dotted] (-0.05,-0.85) -- (2.65,-0.85);
	\draw[thick] (0,0) arc (-90:0:0.5 and 0.5);
	\draw[thick] (0.9,0.5) arc (0:10:0 and -7.6);
	\draw[thick] (1.65,0.5) arc (0:10:0 and -7.6);
	\draw[thick] (2.6,0) arc (-90:0:-0.5 and 0.5);

	\draw[thick] (0.5,-0.8) arc (0:90:0.5 and 0.5);
	\draw[thick] (2.1,-0.8) arc (0:90:-0.5 and 0.5);
	\end{tikzpicture}\;\;,
	\qquad
 \begin{tikzpicture}
 	\draw[thick, dotted] (-0.05,0.5) arc (0:10:0 and -7.5);
 	\draw[thick, dotted] (-0.05,0.55) -- (2.65,0.55);
 	\draw[thick, dotted] (2.65,0.5) arc (0:10:0 and -7.5);
	\draw[thick, dotted] (-0.05,-0.85) -- (2.65,-0.85);
	\draw[thick] (0,0.1) arc (-90:0:0.5 and 0.4);
	\draw[thick] (0,-0.1) arc (-90:0:0.9 and 0.6);
	\draw[thick] (2.6,-0.1) arc (-90:0:-0.9 and 0.6);
	\draw[thick] (2.6,0.1) arc (-90:0:-0.5 and 0.4);
	
	\draw[thick] (0.5,-0.8) arc (0:90:0.5 and 0.5);
	\draw[thick] (1.8,-0.8) arc (0:180:0.5 and 0.5);
	\draw[thick] (2.1,-0.8) arc (0:90:-0.5 and 0.5);
	\end{tikzpicture}\;\;,
	\qquad
 \begin{tikzpicture}
 	\draw[thick, dotted] (-0.05,0.5) arc (0:10:0 and -7.5);
 	\draw[thick, dotted] (-0.05,0.55) -- (2.65,0.55);
 	\draw[thick, dotted] (2.65,0.5) arc (0:10:0 and -7.5);
	\draw[thick, dotted] (-0.05,-0.85) -- (2.65,-0.85);
	\draw[thick] (0,0.05) arc (-90:0:0.5 and 0.45);
	\draw[thick] (0.8,0.5) arc (-180:0:0.5 and 0.45);
	\draw[thick] (2.6,0.05) arc (-90:0:-0.5 and 0.45);

	\draw[thick] (0.0,-0.15) arc (-180:0:1.3 and 0.0);

	\draw[thick] (0.5,-0.8) arc (0:90:0.5 and 0.45);
	\draw[thick] (1.8,-0.8) arc (0:180:0.5 and 0.45);
	\draw[thick] (2.1,-0.8) arc (0:90:-0.5 and 0.45);

	\end{tikzpicture}
\end{equation*}
\caption{Examples of affine diagrams for $N=4$, with the left and right sides of the framing rectangle identified. The first diagram represents the translation generator $u$ while the second diagram is for the generator $e_4\in\ATL{4}(m)$. The third and fourth ones are examples of $j=0$ diagrams.
}
\label{fig:aff-diag}
\end{figure}
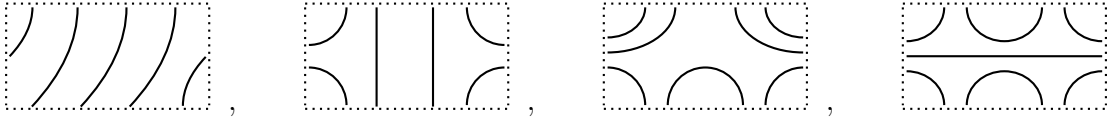

\medskip

We also note that the  diagrams
in this algebra allow winding of through-lines around the annulus any
integer number of times, and different windings result in independent
algebra elements. Moreover, in the ideal of zero through-lines, any number of
non-contractible loops (like in the fourth diagram in Fig.~\ref{fig:aff-diag}) is allowed. The algebra $\ATL{N}(m)$ is thus infinite-dimensional. For $N=1$, it is just the polynomial algebra $\oC[u,u^{-1}]$.

%

\newcommand{\aBr}{\widehat{\mathbb{B}}}

\subsection{The affine TL and the affine braid group}

Let $\oC\mathbb{B}_{N}$ be  the group algebra of the Artin's braid group. As an associative algebra, it is generated by $\g_i^{\pm1}$, with $1\leq i\leq N-1$, subject to $g_i g_j = g_j g_i$ for $|i-j|>1$ and to the standard braid relations:
\begin{equation}\label{br-rel}
\displaystyle g_i g_{i\pm 1} g_i = g_{i\pm 1} g_i g_{i\pm 1}
\end{equation}
or with the graphical notation
\begin{center}
\begin{tikzpicture}
\node at (0,-0.5) {$g_i\quad =$};
\braid[braid colour=black,strands=2,braid start={(0.2,0)}]%
{ \g_1 }
\node at (4.5,-0.5) {$g_i^{-1}\quad =$};
\braid[strands=2,braid start={(5,0)}] 
{\g_1^{-1}}
\end{tikzpicture}
\end{center}
the relations~\eqref{br-rel} can be graphically depicted as
%
%
%
%
%
%
%
\begin{equation}\label{eq:br-rel-diag}
\begin{tikzpicture}
\braid[braid colour=black,strands=3,braid width=24,braid start={(0,0)}]%
{\g_1 \g_2 \g_1}
\node[font=\large] at (4.5,-1.5) {\(=\)};
\braid[strands=3,braid width=24,braid start={(5,0)}]
{\g_2 \g_1 \g_2}
\end{tikzpicture}
\end{equation}

 It is well-known that the finite TL algebra $\TL{N}(\q+\q^{-1})$ is a finite-dimensional quotient of $\oC\mathbb{B}_{N}$ where
 we set
\begin{equation}\label{hecke-gen}
\displaystyle g_i^{\pm 1 } = \pm\rmi (\q^{\pm\half}\one - \q^{\mp \half } e_i)
\end{equation}
and imply the TL relations~\eqref{TL-rel}.

 Let now $\aBr_N$ be the affine braid group -- the group of braids on the surface of a cylinder -- it is generated by the translation $u$ (like above) and $g_i^{\pm1}$, with $0\leq i\leq N-1$, subject to  $u \g^{\pm1}_i u^{-1} = \g^{\pm1}_{i+1}$ and to the braid relations~\eqref{br-rel} where the index $i$ is now interpreted modulo~$N$. 
We recall~\cite{GL2} that the affine TL algebra $\ATL{N}(\q+\q^{-1})$ can be  defined as a quotient of $\oC\aBr_N$ where we again set~\eqref{hecke-gen} and imply the relations~\eqref{TL} and~\eqref{TLpdef-u2}.

Using this connection of $\ATL{N}$ with the braid groups, we will sometimes use below  the braid generators as a replacement for the TL generators. In the diagrams like~\eqref{eq:br-rel-diag}
we emphasize that each under/above crossing of lines should be interpreted as the replacement for the linear combination in~\eqref{hecke-gen}.

\subsection{A tower of affine TL algebras}
\label{sec:aTL-emb}

Our approach to the fusion of affine  TL modules in the next section relies on the induction functor which associates with any pair of  modules over the algebras $\ATL{N_1}(m)$ and $\ATL{N_2}(m)$ a module over the bigger algebra $\ATL{N_1+N_2}(m)$. This functor uses an explicit embedding of the two ``small'' algebras on $N_1$ and $N_2$ sites into the ``big'' one on $N_1+N_2$ sites. 

We start with defining a ``one-step'' embedding  $\ATL{N}(m)\to\ATL{N+1}(m)$:
\begin{align}
u^{(1)}&\,\mapsto \, u \, \g^{-1}_{N}\ , \nonumber\\
e^{(1)}_i &\,\mapsto\, e_i,\qquad 1\leq i \leq N-1\ , \label{one-step-map}\\
e^{(1)}_0 & \, \mapsto\,  \g_{N} \, e_0 \, \g^{-1}_{N}\ ,\nonumber
\end{align}
where we label the generators in $\ATL{N}$ with the superscript $(1)$ and $g_N$ stands for the combination in~\eqref{hecke-gen}. It is straightforward to check that this map is an algebra map. The kernel of this map is trivial: 
we have a basis in the image given by placing an extra ``vertical'' string between the $N$th and $1$st sites of the cylindrical or affine diagram for a basis element in $\ATL{N}$, while each crossing is replaced by the corresponding under-crossing: this gives obviously a bijection between the two bases (explicit diagrams will be given below).

We can use the map~\eqref{one-step-map} recursively and define the embedding $\ATL{N}(m)\to\ATL{N+k}(m)$. Similarly, we can embed the product of two affine TL algebras, $\ATL{N_1}$ and $\ATL{N_2}$, into $\ATL{N}$  with $N=N_1+N_2$.
Let us denote the generators in the $i$th algebra as $u^{(i)}$ and $e^{(i)}_j$, with $i=1,2$, and use standard notations for the generators in the ``big'' algebra $\ATL{N}$. We first define the map on the TL generators $e^{(i)}_j$, where $j\ne0$, in the standard way
\begin{equation}
e^{(1)}_j\mapsto e_j,\qquad e^{(2)}_k\mapsto e_{N_1+k}, \qquad 1\leq j\leq N_1-1,\quad 1\leq k\leq N_2-1.\label{TLhomomorph}
\end{equation}
The translation generators $u^{(1)}$ and $u^{(2)}$  are mapped as (recall, we set $N=N_1+N_2$)
\begin{equation}\label{two-transl}
u^{(1)} \mapsto u\,g_{N-1}^{-1}\ldots g_{N_1}^{-1},\qquad u^{(2)}\mapsto g_{N_1}\ldots g_1 u.
\end{equation}
In terms of diagrams, these two translation generators  are presented as simply as

%

\begin{equation}\label{u1-def}
 \begin{tikzpicture}
 \node[font=\large]  at (-1.5,-2.2) {\mbox{} $u^{(1)}\;\;\mapsto$ \mbox{}\qquad};
\drawu 
\braid[braid colour=black,strands=5,braid width=\brw,braid start={(-0.3,-0.88)}]  {\g_4^{-1}\g_3^{-1}}
\node[font=\large]  at (4.0,-2.2) {\mbox{}\qquad \(=\)};
\end{tikzpicture}
\qquad
 \begin{tikzpicture}
 	\draw[thick, dotted] (-0.05,0.5) arc (0:10:0 and -7.5);
 	\draw[thick, dotted] (-0.05,0.55) -- (3.25,0.55);
 	\draw[thick, dotted] (3.25,0.5) arc (0:10:0 and -7.5);
	\draw[thick, dotted] (-0.05,-0.85) -- (3.25,-0.85);

	\draw[thick] (0.3,0.5) arc (0:10:20 and -3.75);
	\draw[thick] (0.9,0.5) arc (0:10:40 and -7.6);
	\draw[thick] (1.5,0.5) arc (0:10:40 and -7.6);
	

	\draw[thick] (1.6,-0.81) .. controls (1.7,-0.3)   .. (2.1,-0.25);
		\draw[thick] (2.1,-0.25) -- (3,-0.25);
	\draw[thick] (2.1,0.5) -- (2.1,-0.15);
	\draw[thick] (2.1,-0.35) -- (2.1,-0.8);
	\draw[thick] (2.7,0.5) -- (2.7,-0.15);
	\draw[thick] (2.7,-0.35) -- (2.7,-0.8);
		\draw[thick] (3,-0.25) .. controls (3.13,-0.2)   .. (3.21,-0.1);
\end{tikzpicture}
\end{equation}
where we assumed that $N_1=3$ and $N_2=2$, and for the second translation $u^{(2)}$ we have the diagram
\begin{equation}\label{u2-def}
 \begin{tikzpicture}
 \node[font=\large]  at (-1.5,-0.2) {\mbox{} $u^{(2)}\;\;\mapsto$ \mbox{}\qquad};
 
\drawu 

\braid[braid colour=black,strands=5,braid width=\brw,braid start={(-0.28,3.6)}]  { \g_3 \g_2 \g_1}

\node[font=\large]  at (4.0,-0.2) {\mbox{}\qquad \(=\)};
\end{tikzpicture}
\qquad
 \begin{tikzpicture}
 	\draw[thick, dotted] (-0.05,0.5) arc (0:10:0 and -7.5);
 	\draw[thick, dotted] (-0.05,0.55) -- (3.25,0.55);
 	\draw[thick, dotted] (3.25,0.5) arc (0:10:0 and -7.5);
	\draw[thick, dotted] (-0.05,-0.85) -- (3.25,-0.85);
	\draw[thick] (0.0,-0.18) .. controls (0.15,0.08)   .. (0.43,0.1);
	\draw[thick] (0.58,0.1) -- (1.02,0.1);
	\draw[thick] (1.18,0.1) -- (1.62,0.1);
	\draw[thick] (1.77,0.1) .. controls (2.1,0.2)   .. (2.3,0.5);

	\draw[thick] (2.9,-0.81) arc (0:10:-20 and 3.75);

	\draw[thick] (0.5,0.5) -- (0.5,-0.8);
	\draw[thick] (1.1,0.5) --  (1.1,-0.8);
	\draw[thick] (1.7,0.5) -- (1.7,-0.8);
	\draw[thick] (2.75,0.5) arc (0:10:40 and -7.6);

\end{tikzpicture}
\end{equation}
or in words the rightmost string of $u^{(1)}$ (the one that starts at position $N_1$) passes {\sl above} the $N_2$ through-lines on the right from it and ends at the position $1$, and similarly for $u^{(2)}$ -- the leftmost string passes {\sl under} the $N_1$ through-lines on the left from it.
  It is then an easy (in terms of diagrams) calculation using the braid relations to check
\begin{equation}\label{comm-transl}
u^{(1)} u^{(2)} = u^{(2)} u^{(1)}.
\end{equation}
Due to the normalization of $g_i$'s as in~\eqref{hecke-gen}, we have the relations 
\begin{equation}
g_i g_{i+1} e_i =  e_{i+1} e_i,\qquad g_i^{-1} g^{-1}_{i+1} e_i =  e_{i+1} e_i
\end{equation}
and many others similar to these.
In terms of diagrams, these relations tell us  that a  TL arc (``half'' of the diagram for $e_i$) can be pulled out under or above any string at the price of the factor~$1$. We can thus simplify calculations using diagrams with braids and TL arcs only. It is only the twisting that produces a non-trivial factor $\rmi \q^{\frac{3}{2}}$:
\begin{equation}
e_i g_{i+1}e_i = \rmi \q^{\frac{3}{2}} e_i
\end{equation}
but these relations will not appear in calculations below.

Using the remarks above, we then immediately check the affine TL relations 
\begin{equation}
\bigl(u^{(1)}\bigr)^2 e_{N_1-1} = e_{1} \dots e_{N_1-1},\qquad
\bigl(u^{(2)}\bigr)^2 e_{N-1} = e_{N_1+1} \dots e_{N-1} \ .
\end{equation}

We then define the map on the periodic TL generators $e^{(i)}_{0}$ as
\begin{align}
e_0^{(1)} \, &\mapsto\, g_{N_1}\ldots g_{N-1}\,e_0\,g^{-1}_{N-1}\ldots g^{-1}_{N_1},\label{two-e0-1}\\
e_0^{(2)} \, &\mapsto\,  g^{-1}_{0}\ldots g^{-1}_{N_1-1}\,e_{N_1}\,g_{N_1-1}\ldots g_{0}.\label{two-e0-2}
\end{align}
Note that in terms of diagrams, this result is  natural, as illustrated below (again for $N_1=3$ and $N_2=2$)

\begin{equation}\label{u2-def}
 \begin{tikzpicture}
        \node[font=\large]  at (-3,0.4) {\mbox{} $e_0^{(1)}$ \mbox{}\qquad};
\node[font=\large]  at (-2.5,0.2) {\mbox{}\qquad \(=\)};
         \draw[thick] (0,-0.8) -- (0,-0.18);
	\draw[thick] (0.0,-0.18) .. controls (0.15,0.08)   .. (0.43,0.1);
	\draw[thick] (0.4,0.1) -- (1.02,0.1);
	\draw[thick] (1.,0.1) -- (1.5,0.1);
		\draw[thick] (0.5,0.45) -- (0.5,0.2);
	\draw[thick] (0.5,0.0) -- (0.5,-0.8);
	\draw[thick] (1.,0.45) --  (1.,0.20);
	\draw[thick] (1.,0.0) --  (1.,-0.8);

\draw[thick] (0.4,0.55)--(1.5,0.55);
\draw[thick] (0.5,0.65) -- (0.5,1.5);
\draw[thick] (1.,0.65) -- (1.,1.5);

\draw[thick] (0.0,0.8) .. controls (0.15,0.53)   .. (0.43,0.55);
\draw[thick] (0.,0.8) -- (0.,1.5);

\draw[thick] (-0.5,-0.8) -- (-0.5,1.5);

\draw[thick] (-1.,0.8) -- (-1.,1.5);
\draw[thick] (-1.,-0.8) -- (-1.,-0.1);

\draw[thick] (-1.,-0.1) .. controls (-1.1,0.09)   .. (-1.3,0.11);

\draw[thick] (-1.,0.8) .. controls (-1.1,0.5)   .. (-1.3,0.5);
\end{tikzpicture}
\end{equation}

\begin{equation}\label{u2-def}
 \begin{tikzpicture}
        \node[font=\large]  at (-3,0.4) {\mbox{} $e_0^{(2)}$ \mbox{}\qquad};
\node[font=\large]  at (-2.5,0.2) {\mbox{}\qquad \(=\)};
     
     \draw[thick] (-0.5,-0.8) -- (-0.5,1.5);
 \draw[thick] (0.,-0.8) -- (0.,1.5);
  \draw[thick] (0.5,-0.8) -- (0.5,1.5);
        \draw[thick] (-1,0.1) -- (-0.55,0.1);
  \draw[thick] (-1,0.5) -- (-0.55,0.5);
      \draw[thick] (-0.45,0.1) -- (-0.05,0.1);
  \draw[thick] (-0.45,0.5) -- (-0.05,0.5);
  \draw[thick] (0.1,0.1) -- (0.45,0.1);
  \draw[thick] (0.1,0.5) -- (0.45,0.5);
  
  \draw[thick] (1.,-0.1) .. controls (0.9,0.08)   .. (0.7,0.1);
   \draw[thick] (1.,-0.8) -- (1.,-0.1);
     \draw[thick] (0.55,0.1) -- (0.7,0.1);
     
      \draw[thick] (0.55,0.5) -- (0.7,0.5);
      \draw[thick] (1.,0.8) -- (1.,1.5);
      
       \draw[thick] (0.7,0.5) .. controls (0.85,0.55)   .. (1.,0.8);
      
        \draw[thick] (1.5,-0.8) -- (1.5,-0.1);
          \draw[thick] (1.5,0.8) -- (1.5,1.5);
       \draw[thick] (1.5,-0.1) .. controls (1.6,0.08)   .. (1.8,0.1);
         \draw[thick] (1.5,0.8) .. controls (1.6,0.55)   .. (1.8,0.55);

\end{tikzpicture}
\end{equation}
where $e_0^{(i)}$ are considered now as the images of the maps~\eqref{two-e0-1} and~\eqref{two-e0-2} (we will often use the same notation for the images of elements in the subalgebras.)

Using again the diagrammatical calculation, it is straightforward to check that
\begin{equation}
\Bigl(e_0^{(i)}\Bigr)^2 = (\q+\q^{-1})e_0^{(i)},\qquad i=1,2.
\end{equation}
Further, 
we also check all the other affine TL relations
\begin{equation}
e_0^{(i)} = u^{(i)}e^{(i)}_{N_i-1}\bigl(u^{(i)}\bigr)^{-1} = (u^{(i)}\bigr)^{-1}e^{(i)}_{1}u^{(i)},\qquad i=1,2,
\end{equation}
 and
\begin{gather}
e_0^{(i)}e_1^{(i)}e_0^{(i)} = e_0^{(i)},\qquad e_1^{(i)}e_0^{(i)}e_1^{(i)}=e_1^{(i)},\\
e_0^{(i)}e_{N_i-1}^{(i)}e_0^{(i)} = e_0^{(i)},\qquad e_{N_i-1}^{(i)}e_0^{(i)}e_{N_i-1}^{(i)} = e_{N_i-1}^{(i)},
\end{gather}
where we recall that $e^{(1)}_k=e_{k}$ and $e^{(2)}_k=e_{N_2+k}$. We also see using diagrammatic computation that both the subalgebras $\ATL{N_1}$ and $\ATL{N_2}$ indeed commute
\begin{equation}
e_0^{(1)}e_0^{(2)} = e_0^{(2)}e_0^{(1)},
\end{equation}
in addition to~\eqref{comm-transl}.

\newcommand{\afemb}[1]{{\varepsilon}_{#1}}
\newcommand{\afembm}[1]{{\varepsilon}^-_{#1}}

So, we have thus constructed a homomorphism of algebras
\begin{equation}\label{aTL-embed}
\afemb{N_1,N_2}:\qquad \ATL{N_1}\otimes\ATL{N_2} \longrightarrow \ATL{N},
\end{equation}
with the image of the generators given in~\eqref{TLhomomorph}, \eqref{two-transl} and~\eqref{two-e0-1}, and~\eqref{two-e0-2}. This homomorphism has trivial kernel (by recursively using the one-step embedding~\eqref{one-step-map} that has zero kernel), so we have actually an embedding of algebras.

\bigskip

\section{Fusion of affine TL modules}\label{FusAffTL}

In this section, we introduce  \textit{fusion} for modules over the affine TL algebras using the embeddings defined in the previous section. We will use this fusion construction in the next section to define a $\oN$-graded tensor product in the affine TL representation category.

\begin{Dfn}\label{afus-def}
Let $M_1$ and $M_2$ be two modules over $\ATL{N_1}(m)$ and $\ATL{N_2}(m)$ respectively.
Then, the tensor product $M_1\tensor M_2$ is a module over
 the product $\ATL{N_1}(m)\tensor\ATL{N_2}(m)$ of the two algebras.
Using the embedding~\eqref{aTL-embed}, we consider this
 product of algebras as a subalgebra in $\ATL{N}(m)$, for $N=N_1+N_2$.
\textit{The (affine) fusion functor}~$\afus$ on two modules $M_1$ and $M_2$  is then defined as the  module induced from this subalgebra, {\it i.e.}
\begin{equation}\label{fusfunc-def}
M_1\afus M_2 = \ATL{N}\tensor_{\bigl(\ATL{N_1}\tensor\ATL{N_2}\bigr)} M_1\tensor M_2,
\end{equation}
where we used the balanced tensor product over $\ATL{N_1}\tensor\ATL{N_2}$ and we abuse the notation by writing $\ATL{N}$ instead of $\ATL{N}(m)$.
\end{Dfn}

Below we give explicit examples of the affine TL fusion calculation. 
Before doing this, let us recall the basic $\ATL{N}$-modules called the standard modules.

\subsection{Standard $\ATL{N}$ modules}\label{sec:Wz}
We introduce here the standard modules $\StJTL{j}{z}[N]$  over $\ATL{N}(m)$,  which are generically irreducible, and give then several examples of  explicit calculations of the fusion.
The standard modules   are parametrized by pairs $(j,z)$, with a half-integer $j$ and a non-zero complex number $z$. In terms of diagrams, the first is the number of
 through-lines, which we denote by $2j$, $0\leq j\leq N/2$, connecting
 the inner boundary of the annulus with $2j$ sites and the outer boundary with $N$ sites.
 For example, the diagrams
  \begin{tikzpicture}
 	\draw[thick, dotted] (-0.05,0.1) arc (0:10:0 and -3.5);
 	\draw[thick, dotted] (-0.05,0.1) -- (1.05,0.1);
 	\draw[thick, dotted] (1.05,0.1) arc (0:10:0 and -3.5);
	\draw[thick, dotted] (-0.05,-0.55) -- (1.05,-0.55);
	\draw[thick] (0,-0.2) arc (-90:0:0.25 and 0.25);
	\draw[thick] (0.4,0.05) arc (0:10:0 and -2.6);
	\draw[thick] (0.6,0.05) arc (0:10:0 and -2.6);
	\draw[thick] (1.0,-0.2) arc (-90:0:-0.25 and 0.25);
	\end{tikzpicture}
  \, and \, 
  \begin{tikzpicture}
 	\draw[thick, dotted] (-0.095,0.1) arc (0:10:0 and -3.5);
 	\draw[thick, dotted] (-0.095,0.1) -- (1.025,0.1);
 	\draw[thick, dotted] (1.0,0.1) arc (0:10:0 and -3.5);
	\draw[thick, dotted] (-0.095,-0.55) -- (1.025,-0.55);
  \draw[thick] (0,0.05) arc (0:10:0 and -2.6);
	\draw[thick] (0.2,0.05) arc (-180:0:0.25 and 0.25);
	\draw[thick] (0.9,0.05) arc (0:10:0 and -2.6);
	\end{tikzpicture}
 correspond to $N=4$ and $j=1$, where as usual we identify the left and right sides of the framing rectangles, so the diagrams live on the annulus. We call such diagrams \textit{affine}. The action of an element $a\in\ATL{N}(m)$ on $v\in\StJTL{j}{z}$  is then defined
 by stacking the diagrams: joining the inner boundary of $a$ to the outer boundary 
 of the diagram for $v$,
  and removing the
 interior sites. As usual, a closed contractible loop is replaced by
the factor  $m=\q+\q^{-1}$ (we will often use this parametrisation by a complex number $\q$). Whenever the affine diagram thus obtained has a number of
 through lines less than $2j$, the action is zero. For a given
 non-zero value of $j$, it is possible in this action to earn a winding number of the through-lines.
In this case, we imply the relation~\cite{GL} 
\begin{equation*}
\mu=\mu'\circ u_j^n \equiv z^{n}\mu',
\end{equation*} 
where $\mu$ is an affine diagram with $2j$ through lines,  $\mu'$ is a so-called standard diagram which has no through
lines winding the annulus and $u_j$ is the
translational operator acting on the $2j$ sites of the inner boundary of $\mu'$.
 Said differently, whenever $2j$ through-lines wind counterclockwise around
 the annulus $l$ times, we unwind them at the price of a factor
 $z^{2jl}$;
 similarly, for clockwise winding, the phase 
is
$z^{-2jl}$~\cite{MartinSaleur,MartinSaleur1}. This is for $j>0$.
If $j=0$, by the concatenating the diagrams we can produce a non-contractible loop and it has to be replaced by the factor $z+z^{-1}$. Such action gives rise to a generically
irreducible $\ATL{N}(m)$ module, which we denote
 by
$\StJTL{j}{z}[N]$.

 The dimensions of these modules $\StJTL{j}{z}$ are then given by 
 \begin{equation}\label{dim-dj}
 \hat{d}_{j}[N]\equiv \dim \StJTL{j}{z}[N] =
 \binom{N}{{N\over 2}+j},\qquad j\geq 0.
 \end{equation}
Note that these numbers do not depend on $z$ (but modules with
different $z$ are not isomorphic).

\subsection{Examples of the fusion}\label{sec:ex-fus}
We consider here  examples of the fusion defined in Def.~\ref{afus-def} for several pairs of  standard modules. We also assume that $\q$ is generic, i.e. not a root of unity.

\subsubsection{Fusion on $1+1$ sites}\label{sec:ex-fus-1}
We begin with a simple example of the fusion of the pair of standard modules $\StJTL{\half}{z}$ on $1+1$ sites:  \begin{equation}\label{eq:example-fusion-1}
\StJTL{\half}{z_1}[1]\afus\StJTL{\half}{z_2}[1] =\mathrm{Ind}_{\ATL{1}\tensor\ATL{1}}^{\ATL{2}}
\StJTL{\half}{z_1}[1]\tensor\StJTL{\half}{z_2}[1],
\end{equation}
where the affine TL on 1 site is the commutative algebra generated by the translation generator, i.e. $\ATL{1}=\oC[u^{\pm1}]$. We say that the
 left $\ATL{1}$ (the left component of the tensor product $\ATL{1}\tensor\ATL{1}$) is generated by $u^{(1)}$ and the right one is generated by $u^{(2)}$. Then, the module $\StJTL{\half}{z_1}\tensor\StJTL{\half}{z_2}$ is one-dimensional and has the basis element $\vv$ with the action
 \begin{equation}\label{vv-act}
 u^{(1)} \vv = z_1 \vv,\qquad u^{(2)} \vv = z_2 \vv.
 \end{equation}
We have to write now all the relations in the module $\StJTL{\half}{z_1}\tensor\StJTL{\half}{z_2}$ using expressions~\eqref{two-transl} for the generators of the ``small'' algebra $\ATL{1}\tensor\ATL{1}$ in terms of elements of  the ``big'' algebra $\ATL{2}$. So, using~\eqref{vv-act} and~\eqref{two-transl} we have
\begin{align}
 z_1 \vv &= u g_1^{-1} \vv = -i \bigl(\q^{-\half} u  - \q^{\half} u e_1\bigr)\vv,\label{ex1-rel1}\\
 \ffrac{1}{z_1} \vv &=  g_1 u^{-1} \vv = i \bigl(\q^{\half} u^{-1}  - \q^{-\half} e_1 u^{-1}\bigr)\vv,\label{ex1-rel2}\\
 z_2 \vv &= g_1 u \vv = i \bigl(\q^{\half} u  - \q^{-\half}  e_1 u \bigr)\vv,\label{ex1-rel3}\\
 \ffrac{1}{z_2}\vv &=  u^{-1}  g_1^{-1} \vv = - i \bigl(\q^{-\half} u^{-1}  - \q^{\half}  u^{-1}e_1\bigr)\vv\label{ex1-rel4}
\end{align}
and we note here that $u e_1= u^{-1}e_1$ and $e_1 u^{-1}=e_1 u$ because of the relation~\eqref{TLpdef-u2} that takes the form $u^2 e_1 = e_1$. 

Therefore, we have two equations: taking the difference between first and fourth equations we get
\begin{equation}\label{u-uinv}
(u-u^{-1})\vv = i \q^{\half} (z_1 - z_2^{-1})\vv
\end{equation}
and second minus third gives
\begin{equation}
(u-u^{-1})\vv = i \q^{-\half} (z_1^{-1} - z_2)\vv
\end{equation}
and finally the relation between $z_1$ and $z_2$ is
\begin{equation}
\q (z_1 - z_2^{-1}) = (z_1^{-1} - z_2)
\end{equation}
that has only two solutions
\begin{equation}\label{cond-z}
z_2 = z_1^{-1} \qquad \text{or} \qquad z_2=-\q z_{1}.
\end{equation}
It tells us that the fusion or the induced module in~\eqref{eq:example-fusion-1} is zero when the condition~\eqref{cond-z} is not satisfied. 

We then construct a basis for the fusion  in the two different cases: (i) $z_2=-\q z_{1}$ and  (ii) $z_2 = z_1^{-1}$. It turns out that in the case (i) the fusion is a one-dimensional $\ATL{2}$-module while it is two-dimensional in the case (ii). Indeed, assume that $z_2=-\q z_{1}$ then we have the relation (using~\eqref{u-uinv})
\begin{equation}\label{ex1:uu}
u^2\vv = i \q^{\half} (z_1 + \q^{-1} z_1^{-1})u\vv + \vv. 
\end{equation}
The relations~\eqref{ex1-rel1}-\eqref{ex1-rel4} are not the only independent relations in $\StJTL{\half}{z_1}\tensor\StJTL{\half}{z_2}$. We have four more
\begin{align}
z_1 z_2 \vv &= u^{(1)}u^{(2)} \vv = u^2 \vv,\label{ex1-rel5}\\
\ffrac{1}{z_1 z_2} \vv &= \bigl(u^{(1)}u^{(2)}\bigr)^{-1} \vv = u^{-2} \vv,\\
\ffrac{z_1}{z_2} \vv &=  u^{(1)}\bigl(u^{(2)}\bigr)^{-1} \vv =(-\q^{-1} + e_1 + e_0 -\q e_0e_1)\vv,\label{ex1-rel7}\\
\ffrac{z_2}{z_1} \vv &=  u^{(2)}\bigl(u^{(1)}\bigr)^{-1} \vv =(-\q + e_1 + e_0 -\q^{-1} e_1e_0)\vv.\label{ex1-rel8}
\end{align}
Using the first one together with~\eqref{ex1:uu} we get
\begin{equation}\label{ex1-rel9}
u\vv = i\q^{\half} z_1 \vv
\end{equation}
and using then~\eqref{ex1-rel2} we get (recall that $e_0=u e_1 u^{-1}$)
\begin{equation}\label{ex1-rel10}
e_1\vv = 0,\qquad e_0\vv = 0.
\end{equation}
Note that the relations~\eqref{ex1-rel7} and~\eqref{ex1-rel8} are then trivially satisfied.
The  equation~\eqref{ex1-rel10} actually could be immediately deduced directly from the equation~\eqref{ex1-rel5} that tells that $u^2$ acts on $\vv$ by $-\q z_1^2$ and for generic $z_1$ it is not possible for the $\ATL{2}$-module $\StJTL{0}{z}$, where $u^2$ acts as identity. Therefore $\vv$ has to belong to $\StJTL{1}{z}$ where $e_1$ and $e_0$ act by zero. The value of $z$ here  is $ i\q^{\half} z_1$ as follows from~\eqref{ex1-rel9}. So, we conclude that because all the generators of $\ATL{2}$ but $u^{\pm1}$ act on $\vv$ as zero, the induced module~\eqref{eq:example-fusion-1} is one-dimensional and isomorphic to $\StJTL{1}{ i\q^{\half} z_1}$.

Now we turn to the case (ii) when we fix $z_2=z_1^{-1}$. In this case, both~\eqref{u-uinv} and~\eqref{ex1-rel5} give the same relation
\begin{equation}\label{ex1-rel11}
u^2\vv =\vv
\end{equation}
Using our basic relations~\eqref{ex1-rel1}-\eqref{ex1-rel2} we obtain the relations
\begin{align}
ue_1\vv &= \q^{-1}u \vv - i\q^{-\half} z_1 \vv,\\
e_1u\vv &= \q u \vv + i\q^{\half} z_1^{-1} \vv
\end{align}
and so applying $u^{-1}$ on the both sides of the equations and using~\eqref{ex1-rel11} we get
\begin{align}
e_1\vv &= \q^{-1}\vv - i\q^{-\half} z_1 u \vv,\\
e_0\vv &= \q  \vv + i\q^{\half} z_1^{-1} u\vv
\end{align}
and similar formulas for $ue_0\vv$ and $e_0 u\vv$.
Therefore, the action of $e_0 e_1$, $e_1e_0$, etc., on $\vv$ is a linear combination of $\vv$ and $u\vv$. Therefore the induced module~\eqref{eq:example-fusion-1} for $z_2=z_1^{-1}$ is two-dimensional and irreducible (the irreducibility is easy to check for generic $z_1$) and thus isomorphic to $\StJTL{0}{z}$. The basis in this module can be chosen as 
$\{\vv_1=e_1\vv,\, \vv_0 = -i\q^{-3/2}z_1 e_0\vv\}$ which is the standard affine diagrams basis:
$\vv_1=  \begin{tikzpicture}
 	\draw[thick, dotted] (-0.095,0.1) arc (0:10:0 and -3.5);
 	\draw[thick, dotted] (-0.095,0.1) -- (1.025,0.1);
 	\draw[thick, dotted] (1.0,0.1) arc (0:10:0 and -3.5);
	\draw[thick, dotted] (-0.095,-0.55) -- (1.025,-0.55);
	\draw[thick] (0.2,0.05) arc (-180:0:0.25 and 0.25);
	\end{tikzpicture}$
  \, and \,
$\vv_0=  \begin{tikzpicture}
 	\draw[thick, dotted] (-0.05,0.1) arc (0:10:0 and -3.5);
 	\draw[thick, dotted] (-0.05,0.1) -- (1.05,0.1);
 	\draw[thick, dotted] (1.05,0.1) arc (0:10:0 and -3.5);
	\draw[thick, dotted] (-0.05,-0.55) -- (1.05,-0.55);
	\draw[thick] (0,-0.2) arc (-90:0:0.25 and 0.25);
	\draw[thick] (1.0,-0.2) arc (-90:0:-0.25 and 0.25);
	\end{tikzpicture}$\,.
The weight of the non-contractible loops $z+z^{-1}$ is then computed as
\begin{equation}
e_0 \vv_1 = (z+z^{-1}) \vv_0\qquad \text{and}\qquad e_1 \vv_0 = (z+z^{-1}) \vv_1
\end{equation}
and a simple calculation gives $z=-i\q^{-\half}z_1$. Note that in this case, there is in fact an ambiguity for the sign of $z$, since the pair $(\vv_1,z)$ is defined only up to a sign: in fact, the two modules $\StJTL{0}{\pm z}[2]$ are isomorphic and our choice of the sign in $z$ is just a convention.

Finally, after simple but long calculations we conclude the fusion formula
\begin{equation}\label{afus-ex1}
\StJTL{\half}{z_1}[1]\afus\StJTL{\half}{z_2}[1] =
\begin{cases}
\StJTL{1}{ i\q^{\half} z_1}[2]\qquad & \qquad \text{when}\; z_2=-\q z_1,\\
\StJTL{0}{- i\q^{-\half}z_1}[2]\qquad & \qquad \text{when}\; z_2=z_1^{-1},\\
0\qquad & \qquad \text{otherwise.}
\end{cases}
\end{equation}

Note that we have the same conditions on non-zero fusion  $\StJTL{\half}{z_1}[1]\afus\StJTL{\half}{z_2}[3]$ on $1+3$ sites, as the basic relations have the same form as here. It  just takes more calculations to find a basis in the two non-zero cases. It will be proven below (using our construction of direct systems of categories) that the formula~\eqref{afus-ex1} holds in this case as well.

\subsubsection{Fusion on $1+2$ sites}
By a direct calculation on $1+2$ sites similar to the previous calculation, we have found the following fusion
\begin{equation}\label{afus-ex2}
\StJTL{\half}{z_1}[1]\afus\StJTL{1}{z_2}[2] =
\begin{cases}
\StJTL{\frac{3}{2}}{-\q z_1}[3] \qquad & \qquad \text{when}\; z_2=-i\q^{3/2} z_1,\\
\StJTL{\half}{-\q/ z_1}[3]  \qquad & \qquad \text{when}\; z_2=i\q^{1/2}z_1^{-1},\\
0\qquad & \qquad \text{otherwise.}
\end{cases}
\end{equation}

 A general formula for the fusion on $N_1+N_2$ sites for any pair of the standard modules $\StJTL{j_1}{z_1}[N_1]$ and $\StJTL{j_2}{z_2}[N_2]$ is derived in~\cite{GJS}, where we also discuss physical implications of the results.

\medskip

An alternative to the diagrammatical calculation is presented in App.~\ref{app:fusion-example-3} where  we give a few more examples. We use there the relation to the affine Hecke algebra discussed below.

\subsection{Remark on the affine Hecke algebra}

The affine Temperley Lieb algebra $\ATL{N}(m)$ is known to be deeply related to the affine Hecke algebra $\AH{N}(\q)$ with $m=\q+\q^{-1}$. It is useful here to recall some basic facts about the latter~\cite{ChariPressley}. 
We then show how our definition of the affine TL tower is related to the more standard tower of affine Hecke algebras.

\subsubsection{Definition I:}
The algebra $\AH{N}(\q)$ is usually defined as follows: it is an associative algebra over $\oC$ generated by 
  $\sigma_i$, with $1\leq i\leq N-1$,  and $y^{\pm1}_j$, with  $1\leq j\leq N$, subject to the relations\footnote{We use here the standard conventions to facilitate comparison with the literature. A variant - more suitable to the $\ATL{N}$ quotient is discussed in the appendix.}
\begin{eqnarray}
    (\sigma_{i}+1)(\sigma_{i}-\q^{-2})&=&0\nonumber\\
    \left[\sigma_{i},\sigma_{j}\right]&=&0, \qquad |i-j|\geq 2\nonumber\\
    \sigma_{i}\sigma_{i+1}\sigma_{i}&=&\sigma_{i+1}\sigma_{i}\sigma_{i+1}
\end{eqnarray}
together with 
\begin{eqnarray}
      \left[y_{i},y_{j}\right]&=&0\nonumber\\
    \left[y_{j},\sigma_{i}\right]&=&0, \qquad  j\neq i,i+1\nonumber\\
    \sigma_{i}y_{i}\sigma_{i}&=&\q^{-2}y_{i+1}\ .\label{theyrelations}
\end{eqnarray}
Note that we have thus the  braid generators $\sigma_i$ (subject to the standard Hecke relations) and the family of commutative generators $y_j$ that have the commutation relations 
$\sigma_i y_i= \q^{-2} y_{i+1}\sigma_i^{-1} = y_{i+1}\sigma_i+(1-\q^{-2})y_{i+1}$. The algebra $\AH{N}$ can be thus considered as a twisted tensor product $\fH{N}\otimes \oC[y^{\pm1}_1, \dots, y^{\pm1}_N]$, where $\fH{N}$ is the finite Hecke algebra. It is well-known indeed that $\AH{N}$ is isomorphic as a vector space to the tensor product of the finite Hecke algebra and the algebra of Laurent polynomials in $y_j$.

\subsubsection{Zelevinsky's tensor product}
The definition of $\AH{N}$ given above in terms of $\sigma_i$ and $y_j$ leads us to a (well known) homomorphism of algebras 
 \begin{equation}\label{H-emb}
 \AH{N_1}(\q)\otimes \AH{N_2}(\q) \mapsto \AH{N_1+N_2}(\q)
 \end{equation}
 where, in the same notations as in (\ref{TLhomomorph}), we have 
 \begin{equation}
 \sigma^{(1)}_j\mapsto \sigma_j,\qquad \sigma^{(2)}_k\mapsto \sigma_{N_1+k}, \qquad 1\leq j\leq N_1-1,\quad 1\leq k\leq N_2-1.\label{AHhomomorph}
\end{equation}
and 
 \begin{equation}
y^{(1)}_j\mapsto y_j,\qquad y^{(2)}_k\mapsto y_{N_1+k}, \qquad 1\leq j\leq N_1,\quad 1\leq k\leq N_2.\label{AHhomomorph1}
\end{equation}
It is thus an embedding of the algebras.  

Having the embedding in~\eqref{H-emb}, we can now define the affine Hecke fusion $\ahfus$ as the induced module (see e.g.~\cite{ChariPressley})
\begin{equation}\label{aH-fusfunc-def}
M_1\ahfus M_2 = \AH{N}(\q)\tensor_{\bigl(\AH{N_1}\tensor\AH{N_2}\bigr)} M_1\tensor M_2,
\end{equation}
where $M_1$ and $M_2$ are modules over $\AH{N_1}(\q)$ and $\AH{N_2}(\q)$, respectively.
This fusion was originally introduced by Zelevinsky and since then is usually called \textit{the Zelevinsky's tensor product} of affine Hecke algebra modules.

\subsubsection{Definition II:}
Meanwhile, the algebra $\AH{N}(\q)$  admits another definition
involving generators $\sigma_i$, $i=1,\ldots,N$ (recall $i$ was running only up to $N-1$ in the previous definition) and a translation generator $\tau$ such that 
\begin{eqnarray}
    (\sigma_{i}+1)(\sigma_{i}-\q^{-2})&=&0\nonumber\\
    \left[\sigma_{i},\sigma_{j}\right]&=&0, \qquad |i-j|\geq 2\nonumber\\
    \sigma_{i}\sigma_{i+1}\sigma_{i}&=&\sigma_{i+1}\sigma_{i}\sigma_{i+1}\nonumber\\
    \tau \sigma_{i}\tau^{-1}&=&\sigma_{i+1}
\end{eqnarray}
and the indices have to be interpreted modulo $N$. The equivalence of the two definitions follows from  the identification
\begin{equation}
\tau=y_1 \sigma_1\ldots\sigma_{N-1}\label{identif}\ ,
\end{equation}
see a complete proof in~\cite{GL2}.
Note that the relations are invariant under rescaling of $\tau$ in the second   definition, and $y_1$ in the first definition.

\bigskip

Let us now see what happens to the homomorphism~\eqref{AHhomomorph}-\eqref{AHhomomorph1} with the second definition of $\AH{N}$. We have now
\begin{equation}
\tau^{(1)}\mapsto y_1\sigma_1\ldots \sigma_{N_1-1},\qquad \tau^{(2)}\mapsto y_{N_1+1}\sigma_{N_1+1}\ldots \sigma_{N_1+N_2-1}
\end{equation}
Meanwhile, we have also for the algebra on $N_1+N_2$ sites
\begin{equation}
\tau=y_1 \sigma_1\ldots\sigma_{N_1+N_2-1}
\end{equation}
An easy calculation then leads to 
\begin{eqnarray}
\tau^{(1)}=\tau \sigma_{N_1+N_2-1}^{-1}\ldots \sigma_{N_1}^{-1},\qquad \tau^{(2)}=\q^{2N_1}\sigma_{N_1}\ldots\sigma_1\tau
\end{eqnarray}
These formulas  become identical\footnote{Note that in the affine Hecke algebra, the overall normalization of the $y_i$ generators (in the first definition) or the $\tau$ generator (in the second definition) is not fixed, in contrast with the normalization of $u$ in the definition of affine TL  (see eq.~(\ref{TLpdef-u2})).} with those we used earlier in (\ref{two-transl}) after taking  the proper quotient of $\AH{N}(\q)$ which we describe now.

We now recall that the affine Temperley--Lieb algebra $\ATL{N} \equiv \ATL{N}(\q+\q^{-1})$ can be obtained from the affine Hecke in two steps~\cite{GL2}: first, we demand the relations
\begin{equation}\label{Ei-def}
E_i \equiv 1+\sigma_i+\sigma_{i+1}+\sigma_i\sigma_{i+1}+\sigma_{i+1}\sigma_i+\sigma_i\sigma_{i+1}\sigma_i=0,\qquad i=1,\ldots,N,
\end{equation}
or equivalently  take a quotient of the affine Hecke algebra $\AH{N}$ by the two sided ideal generated by $E_1$ (note that all $E_i$'s are in the ideal).
This quotient  -- denote it $\widehat{\mathsf{TL}}_{N}$ -- is not in itself the affine Temperley-Lieb algebra $\ATL{N}$ because of  the second relation in (\ref{TLpdef-u2}) which does not follow from~\eqref{Ei-def}.
(This relation follows instead  from considering the realization as an algebra of diagrams.)
The second step in obtaining $\ATL{N}$ is thus, after identifying $\tau$ with $(\rmi \q^{-1/2})^{N-1}u$,
and $\sigma_i$ with $\rmi\q^{-1/2}g_i$, 
  to take a quotient by the ideal generated by the element
\begin{equation}
\ue \equiv u^2e_{N-1} - e_1\ldots e_{N-1}.
\end{equation} 
In total, we  have
\begin{equation}\label{TH}
\ATL{N} = \AH{N}/\I,\qquad \text{where} \quad \I=\langle E_1, \omega\rangle
\end{equation}
where we introduce the two-sided ideal $\I$ generated by $E_1$ and $\omega$.

\subsection{The affine TL fusion from Zelevinsky's tensor product}

We will use the following lemma that relates induced modules over an algebra and over its quotient.
\begin{lemma}\label{lem:ABCI}
Let $A$ be an associative algebra,  $I$ a two-sided ideal in $A$ and $B=A/I$ the quotient algebra. Let also $C$ be a subalgebra in $A$ and $I_C=C\cap I$ ($I_C$ is thus an ideal in $C$). Then, $C/I_C$ is a subalgebra in $B$  and we  have an isomorphism of induced modules
\begin{equation}\label{ABCI}
\Ind_{C/I_C}^B M  \cong \Ind_C^A\, M/(I \cdot \Ind_C^A\, M) \ ,
\end{equation}
where $M$ is a left $C$-module with trivial action of $I_C$, i.e., $M$ is also a module over the quotient algebra $C/I_C$.
\end{lemma}
\begin{proof}
We first show that $I_C=C\cap I$ is an ideal in $C$: let $r\in I_C$ then on one side $a\cdot r\cdot b\in C$ for any $a,b\in C$ because $r\in C$ and on the other side $a\cdot r\cdot b$ is also in $I$ because $r\in I$; therefore, $a\cdot r\cdot b$ is in $I_C$. We consider then the quotient $C/I_C$ which is obviously a subalgebra in $B=A/I$. Therefore, the left hand side of~\eqref{ABCI} is well-defined. We rewrite then the right hand side of~\eqref{ABCI}   as
$$
 A \otimes_C M/(I \cdot A\otimes_C M) \cong (A/I) \otimes_{C/I_C} (M/I_C\cdot M) = B\otimes_{C/I_C} M,
$$
where the first isomorphism follows simply from the definition of the balanced tensor product while the second is by definition of $B$ and our assumption on $M$, which has the trivial $I_C$ action. This finishes the proof of the lemma.
\end{proof}
\medskip

In our context, $A=\AH{N}$ with the ideal $I=\I$ defined in~\eqref{TH} and the quotient algebra $B=\ATL{N}$, the subalgebra $C$ is the product $\AH{N_1}\otimes\AH{N_2}$, with $N_1+N_2=N$, and we denote $\I_C = \I\cap C$. Using Lemma~\ref{lem:ABCI}, we have for  $\ATL{N_i}$-modules $M_i$ an isomorphism (reading~\eqref{ABCI} from right to the left)
\begin{equation}\label{H-T-fusion}
M_1\ahfus M_2/(\I\cdot M_1\ahfus M_2) \cong \ATL{N}\otimes_{\AH{N_1}\otimes\AH{N_2}/\I_C} \Bigl(M_1\tensor M_2/(\I_C\cdot M_1\otimes M_2)\Bigr)\ ,
\end{equation}
where $\ahfus$ is the affine Hecke fusion introduced in~\eqref{aH-fusfunc-def} and the quotient $\AH{N_1}\otimes\AH{N_2}/\I_C$ is considered as a subalgebra in $\ATL{N}$. 
Note that in App.~\ref{app:examples}  we actually compute the left-hand side of~\eqref{H-T-fusion} and it agrees with the affine TL fusion computed here and in~\cite{GJS}. Let us formulate the following conjecture:
 
\begin{conj}\label{conj:H-T}
We have an isomorphism $\AH{N_1}\otimes\AH{N_2}/\I_C \cong \ATL{N_1}\otimes\ATL{N_2}$ of algebras, where the ideal $\I_C$ is defined as the intersection of the ideal $\I=\langle E_1, \omega\rangle$ with the subalgebra $\AH{N_1}\otimes\AH{N_2}$ in $\AH{N_1+N_2}$.
\end{conj}

Under Conj.~\ref{conj:H-T}, we obtain  the affine TL fusion $M_1 \afus M_2$ on the right hand side of~\eqref{H-T-fusion}.
We demonstrate in several examples in App.~\ref{app:examples} that the fusion $\afus$ for $\ATL{N}$ modules indeed can be computed  as the quotient $M_1\ahfus M_2/(\I\cdot M_1\ahfus M_2)$ of the fusion $\ahfus$ of the same modules but considered as $\AH{N}$-modules, i.e. of the pull-back of the affine TL modules.
However, while arbitrary modules of $\AH{N_1}$ and $\AH{N_2}$ can be fused to give a non trivial result, in general,  the quotient turns out to be empty, except when some specific ``resonance'' conditions are satisfied. This is discussed more in the appendix.

\bigskip

\section{Limits of TL-mod categories }\label{limTLcat}
We now go back to the finite TL representation categories and prepare the machinery that we are going to use in the case (more interesting to us) of  affine TL  to show that the affine TL fusion does not depends on the choice of the pair $(N_1,N_2)$
and is equipped with an associator. The affine case will be treated in the next section.

\subsection{Arc-tower of TL algebras}\label{sec:arc-tower}
Recall that in Section~\ref{sec:TL-towers} we introduced a standard tower of TL algebras by using the standard embeddings. These embeddings were used to define the fusion functor for each pair of numbers $(N_1,N_2)$. Our task here is to connect  the $\TL{N}$-mod categories at different $N$, i.e. to construct an inductive system of $\TL{N}$-mod and to show that the fusion functors induce a monoidal structure in the inductive limit category. For such a connection between different $N$, we use 
another tower of the finite Temperley--Lieb algebras given by what we call \textit{(right) arc-embeddings}:
$$
\TL{N} \xrightarrow{\quad \psi \quad} \TL{N+2}
$$
  defined in terms of TL diagrams by enlarging the $\TL{N}$ diagrams with  arcs (instead of the vertical strings) at sites $(N+1,N+2)$ in the top and bottom of the diagram, or in other words
 $$
 \psi(e_j) = \e e_j \e \ ,
 $$
  where we introduce the idempotent
 \begin{equation}\label{e-def}
 \e=\ffrac{1}{m} \,e_{N+1} \ .
 \end{equation}
 It is straightforward to check that $\psi$ defines a homomorphism of algebras with trivial kernel. Such a tower will be called \textit{arc-tower}.
 
 Note also that we have  an isomorphism $\TL{N} \cong \e \TL{N+2} \e$ and can thus consider $\TL{N}$ as an idempotent subalgebra in $\TL{N+2}$. This allows us to define two functors between the categories of TL modules as follows. Recall that $\catTL_N$ denotes the category of finite-dimensional $\TL{N}$-modules. We have the localization functor
 \begin{equation}\label{loc-fun}
 \funLoc_N: \; \catTL_{N+2} \to \catTL_{N}\qquad \text{such that}\qquad M \mapsto \e M \ ,
 \end{equation} 
 with an obvious  map on morphisms, and its right inverse, the so called globalization functor
 \begin{equation}\label{glob-fun}
 \funGl_N: \; \catTL_{N} \to \catTL_{N+2}\qquad \text{such that}\qquad M \mapsto \TL{N+2}\e \otimes_{\TL{N}} M,
 \end{equation} 
 where $\TL{N+2}\e$ is considered as a left module over $\TL{N+2}$ (by the left multiplication) and a right module over the idempotent subalgebra  $\e \TL{N+2} \e$ (by the right multiplication), the balanced tensor product is also taken over the idempotent subalgebra  $\TL{N} = \e \TL{N+2} \e$. On morphisms, we have $\funGl(f) = \mathrm{id} \otimes_{\TL{N}} f$.
 It is a simple exercise to check that the composition $\funLoc_N\circ \funGl_N$ is naturally isomorphic to the identity functor on $\catTL_N$. The reverse composition is not the identity, as the two categories are obviously not equivalent. 
 Instead, we have the following statement.
 
\begin{Prop}\label{prop:Ie}
The composition $\funGl_N\circ \funLoc_N$ maps a $\TL{N+2}$-module $M$ to $I_{\es}\cdot M$, where $I_{\es}$ is the two-sided ideal generated by $\e$ in $\TL{N+2}$.
\end{Prop}
\begin{proof}
 We  compute the composition $\funGl_N\circ \funLoc_N$ as
\begin{equation}\label{GL-M}
\funGl_N\circ \funLoc_N:\qquad M\; \mapsto \; \TL{N+2}\e \otimes_{\TL{N}} \bigl(\e \TL{N+2} \otimes_{\TL{N+2}} M\bigr)\ ,
\end{equation}
where we rewrote $M$ as $\TL{N+2} \otimes_{\TL{N+2}} M$. Then, we use the associativity of the tensor product over algebras and rewrite the expression in~\eqref{GL-M} as
\begin{equation}\label{GL-M-2}
\funGl_N\circ \funLoc_N(M) = \bigl(\TL{N+2}\e \otimes_{\TL{N}} \e \TL{N+2}\bigr) \otimes_{\TL{N+2}} M\ ,
\end{equation}
 where recall that the tensor product is over $\TL{N}=\e\TL{N+2}\e$. We also note the isomorphism
$$
\TL{N+2}\e \otimes_{\TL{N}} \e \TL{N+2} \xrightarrow{\quad \cong \quad} \TL{N+2} \cdot \e \cdot \TL{N+2} = I_\es
$$ 
of the $\TL{N+2}$ bimodules (by the left and right multiplication) given by
\begin{equation}\label{iso:aeb}
a \e \otimes \e b \; \mapsto\; a \e b\ , \quad \text{for}\quad a,b\in\TL{N+2}\ .
\end{equation}
The inverse to this map can be constructed as follows: any element in $I_\es$ can be presented as $a\e b$ (though not in a unique way, one can rewrite $a\e b = a'\e d$ with $a'=a\e c$ if $b=c\e d$); take any of these representatives and map $a\e b \mapsto a\e\otimes \e b$. This map is well-defined, i.e., does not depend on the representative because the tensor product is over $\e\TL{N+2}\e$, and obviously inverse to the map~\eqref{iso:aeb}.
Finally, using~\eqref{GL-M-2} together with the isomorphism in~\eqref{iso:aeb} we obtain an isomorphism of vector spaces
$$
\funGl_N\circ \funLoc_N(M)  \xrightarrow{\quad \cong \quad} I_{\es}  \otimes_{\TL{N+2}} M = I_\es \cdot M.
$$
Note also that the $\TL{N+2}$ actions are equal on both of sides -- they are simply given by the multiplication, so the modules are actually equal and not just isomorphic.
\end{proof}

\begin{rem}\label{rem:Ie-span}
Note that the ideal $I_{\es}$  generated by $\e$ in $\TL{N+2}$ is spanned by all TL diagrams except the unit $\one$. This is easy to see in terms of the generators of the subalgebra~$I_\es$: all $e_i$'s are in $I_{\es}$.
\end{rem}

\subsection{The direct limit $\catTLinf$}

We use the globalisation functors $\funGl_N$, for $N\geq 1$, in constructing certain direct systems of the TL representation categories $\catTL_N$ and eventually their  direct limits. 
By a \textit{direct (or inductive) system} of categories we mean a pair $\{C_i,F_{ij}\}$ of a family of categories $C_i$ indexed by an ordered set~$I$ and a family of functors $F_{ij}: C_i\rightarrow C_j$ for all $i\leq j$ satisfying the following properties: (i) $F_{ii}$ is the identity functor on $C_i$, and (ii) $F_{ik}$ is naturally isomorphic to $F_{jk} \circ F_{ij}$ for all $i\leq j\leq k$. The \textit{direct limit} 
\begin{equation}\label{eq:dir-lim}
C_{\infty}\equiv\varinjlim C_i
\end{equation}
 of the direct system $\{C_i, F_{ij}\}$ is defined as the disjoint union $\coprod_i C_i/\sim$\, modulo an equivalence relation:
two objects $V_i\in C_i$ and $V_j\in C_j$ in the disjoint union are equivalent if and only if
there is $k\in I$ such that $F_{ik}(V_i)=F_{jk}(V_j)$; and similarly for morphisms: two morphisms $f_i: V_i\to W_i$ and $f_j: V_j\to W_j$ are equivalent if the equality $F_{ik}(f_i)=F_{jk}(f_j)$ holds in $\mathrm{Hom}(F_{ik}(V_i),F_{ik}(W_i))$. We obtain from this definition canonical functors $F_{i\infty}: C_i\rightarrow C_{\infty}$ mapping each object to its equivalence class.
We will use these functors to define additional structures on $C_{\infty}$, such as tensor product etc. 
 
 \medskip

Recall that we introduced  the ``enveloping'' TL category
\begin{equation}
\catTL = \bigoplus_{N\geq 1} \catTL_{N}
\end{equation}
which means that $\catTL$ contains $\catTL_{N}$ as a full subcategory and there are no morphisms between the full subcategories for different $N$ (it is also known as the disjoint union of $\catTL_N$'s). The category $\catTL$ is thus graded by $\oN$. We will label  an object $M$ from $\catTL_N$ as $M[N]$ to emphasize its grade.

Inside $\catTL$, we consider two direct systems (recall the definition above, where $C_i=\catTL_i$ and $F_{ij}=\funGl_{j-2}\circ\ldots \funGl_{i+2}\circ\funGl_{i}$, with $i\leq j$)
$$
\catTL_1 \xrightarrow{\;\; \funGl_1 \;\;} \catTL_3  \xrightarrow{\;\; \funGl_3 \;\;} \ldots
$$
and 
$$
\catTL_2 \xrightarrow{\;\; \funGl_2 \;\;} \catTL_4  \xrightarrow{\;\; \funGl_4 \;\;} \ldots\ .
$$
We denote the corresponding direct limits as 
\begin{equation}\label{catTLinf-0}
\catTLodd = \varinjlim \catTL_{\mathrm{odd}}
\qquad
\text{ and } 
\qquad
 \catTLev = \varinjlim \catTL_{\mathrm{even}}\ .
 \end{equation}
Then, we define the category
\begin{equation}\label{catTLinf}
\catTLinf = \catTLev \oplus \catTLodd.
\end{equation}
Note that by the construction the category $\catTLinf$ is an abelian $\oC$-linear category for any non-zero value of $\q$, i.e., including the roots of unity cases (except $\q=\pm\rmi$ where our construction is not defined).

\newcommand{\el}{\,\tilde{\mathsf{e}}_{(N)}}
\newcommand{\psil}{\tilde{\psi}}
\begin{rem}\label{rem:funGll}
We note that \textit{left arc-embeddings} can be introduced similarly enlarging diagrams on the left with the pair of the arcs:   introduce a new idempotent $\el$ in $\TL{N+2}$ as
 \begin{equation}\label{el-def}
 \el=\ffrac{1}{m} \,e_{1} \ \in \ \TL{N+2}
 \end{equation}
and define  \textit{left arc-embeddings}:
$$
\TL{N} \xrightarrow{\quad \psil \quad} \TL{N+2}
$$
with 
 $$
 \psil(e_j) = \el e_{j+2} \el \ .
 $$
 
 We have also  new localization $\funLocl_N$ and globalization  $\funGll_N$ functors corresponding to the new idempotent subalgebra $\TL{N}\cong \el\TL{N+2}\el$: they are defined similarly to $\funLoc_N$ and $\funGl_N$ as in~\eqref{loc-fun} and~\eqref{glob-fun}, respectively. It is clear that $\funLocl_N$ and $\funLoc_N$ are naturally isomorphic, as well as their adjoint functors $\funGll_N$ and $\funGl_N$. (The isomorphism can be explicitly stated in terms of the TL diagrams.)
\end{rem}

The interesting property of the localization functors $\funLoc_N$ is that
they map the standard (resp., costandard) module of a weight $j$ on $N+2$ sites to the standard (resp., costandard) module of the same weight $j$ but on $N$ sites. We state this as follows.

\begin{prop}\label{prop:eW}
For any non-zero $\q$, the  $\TL{N}$-module $\e \StTL{j}[N+2]$ is equal to the standard module $\StTL{j}[N]$ for $j\leq N/2$ and 0 otherwise.
\end{prop}
\begin{proof}
We recall that $\TL{N+2}$ is a quasi-hereditary algebra and the idempotent $\e$ belongs to a heredity chain of $\TL{N+2}$, therefore $\TL{N} = \e \TL{N+2}\e$ is also a quasi-hereditary algebra. The proof of the proposition is then standard (in the theory of quasi-hereditary algebras) and follows for example from \cite[Sec. 3.3 and Prop. 3.4]{MW}.

It is also possible to prove the result by elementary means. First, by adding an arc to the right of any link diagram in $\StTL{j}[N]$, it is clear that we obtain a link diagram in $\e \StTL{j}[N+2]$. Let us call this map by  $\rho: \StTL{j}[N]\to\e \StTL{j}[N+2]$, its kernel is obviously trivial. Conversely, we start from a link diagram in $\StTL{j}[N+2]$. If the last two points are both occupied by through-lines, the action of the idempotent $\e$ is zero. If none of the last two points are occupied by through-lines, the action of $\e$ simply produces, apart from the arc between the $N+1^{th}$ and $N+2^{th}$ points, a new connection between two points among the first $N$ ones, resulting in  a link diagram in $\StTL{j}[N]$ after this arc is removed. If only one point is occupied by a through-line, this point is necessarily the $N+2^{th}$ one. The action of $\e$ then moves this through-line to a new position -- the point to which the $N+1^{th}$ point was connected by an arc, resulting again, apart from the arc between the $N+1^{th}$ and $N+2^{th}$ points, into a  link diagram  with the same number of through-lines, that is a link diagram in $\StTL{j}[N]$. 
We have thus shown by the first map $\rho$ that $\StTL{j}[N]\subset \e \StTL{j}[N+2]$ and the inclusion $\e \StTL{j}[N+2]\subset \StTL{j}[N]$ by the second map. Hence, we have  a bijection. Also, both the maps are obviously  module maps and we have thus an isomorphism of the $\TL{N}$ modules.    Moreover, the first map $\rho$  is  the identity map
 -- adding the arc on the right is only a convention in terms of the diagrammatical bases in the two spaces -- and its inverse is the identity as well. We have thus shown that $\funLoc_N$ sends the standard module $\StTL{j}[N+2]$ to~$\StTL{j}[N]$.
\end{proof}

As a consequence, we obtain the important property of the globalisation functors:

\begin{prop}\label{prop:funGl-W}
For $x\leq j\leq \frac{N}{2}$ and  $x = \half(N\,\mathrm{mod}\, 2)$, we have  
\begin{equation}\label{eq:funGl-St}
\funGl_N: \qquad \StTL{j}[N]\; \mapsto \; \StTL{j}[N+2]\ .
\end{equation}
\end{prop}
\begin{proof}
We first compute the composition $\funGl_N\circ \funLoc_N$ on the module $\StTL{j}[N+2]$ using Prop.~\ref{prop:Ie} and obtain
\begin{equation}
\funGl_N\circ \funLoc_N:\quad \StTL{j}[N+2]\; \mapsto\;  I_{\es} \cdot \StTL{j}[N+2]\ ,
\end{equation}
which is $\StTL{j}[N+2]$ for $x\leq j\leq \frac{N}{2}$ and zero otherwise, see Rem.~\ref{rem:Ie-span}.
Then, recall that  the composition $\funLoc_N\circ \funGl_N$ is naturally isomorphic to the identity functor on $\catTL_N$.
Together with Prop.~\ref{prop:eW}, we then obtain~\eqref{eq:funGl-St}.
\end{proof}


The idea is then to study the direct limit of the standard modules (their subquotient structure)  and their fusion rules in the direct limit category. 
So, our next step is to use the bilinear $\oN$-graded tensor product $\fus$ on $\catTL$, recall Prop.~\ref{prop:N-graded}, and to show that it defines a monoidal structure on $\catTLinf$, with braiding. We start by exploring an associator for $\fus$ on $\catTL$ and prove  the statement about associativity in Prop.~\ref{prop:N-graded}.

\subsection{Associativity of TL fusion}\label{sec:ass-TL}
The important property of the fusion tensor-product  $\fus$ introduced in~\eqref{fusfunc-TL-def} is the associativity. 

\begin{prop}\label{prop:ass}
Let $M_1$, $M_2$ and $M_3$ be three modules over $\TL{N_1}(m)$, $\TL{N_2}(m)$ and $\TL{N_3}(m)$, respectively.
The tensor product~$\fus$ is equipped with an associator, {\it i.e.,} we have a family $\ass_{M_1,M_2,M_3}$ of natural isomorphisms of $\TL{N_1+N_2+N_3}(m)$ modules
\begin{equation}\label{fusfunc-TL-ass}
\ass_{M_1,M_2,M_3}:\quad \bigl(M_1\fus M_2\bigr) \fus M_3 \xrightarrow{\quad \cong\quad}
M_1\fus \bigl( M_2 \fus M_3\bigr) 
\end{equation}
given explicitly as, for any triple of vectors $m_i\in M_i$, with $i=1,2,3$, 
\begin{equation}\label{TL-ass-map}
\ass_{M_1,M_2,M_3}:\quad a\otimes(b\otimes m_1\otimes m_2)\otimes m_3 \mapsto a\cdot b\otimes \bigl(m_1\otimes (\one_{2+3}\otimes m_2\otimes m_3)\bigr)\ ,
\end{equation}
where $a\in\TL{N_1+N_2+N_3}(m)$, $b\in\TL{N_1+N_2}(m)$, and $\one_{2+3}$ is the identity in $\TL{N_2+N_3}$, and $a\cdot b$ stands for the product of
the element  $a$ with the image of $b$ under the standard embedding of $\TL{N_1+N_2}$ into $\TL{N_1+N_2+N_3}$.
\end{prop}
\begin{proof}
To prove that the map $\ass_{M_1,M_2,M_3}$  in~\eqref{fusfunc-TL-ass} is an isomorphism, we show that it is a composition of two isomorphisms: $\liso$ and the inverse to $\riso$, where
$\liso$ is defined as the composition of the isomorphisms
\begin{multline}\label{ass:3lines-l}
\liso:\quad \bigl(M_1\fus M_2\bigr) \fus M_3  = \TL{1+2+3}\otimes_{\TL{1+2}\otimes\TL3} \left[\Bigl(\TL{1+2}\otimes_{\TL1\otimes\TL2} M_1\otimes M_2\Bigr)\otimes M_3\right]\\
\xrightarrow{\quad \cong\quad} \TL{1+2+3}\otimes_{\TL{1+2}\otimes\TL3} \left[\Bigl(\TL{1+2}\otimes \TL3\otimes_{\TL1\otimes\TL2\otimes\TL3} (M_1\otimes M_2\otimes M_3)\Bigr)\right]\\
\xrightarrow{\quad \cong\quad} \TL{1+2+3}\otimes_{\TL1\otimes\TL2\otimes\TL3}  (M_1\otimes M_2\otimes M_3),
\end{multline}
where we use the short-hand notations $\TL{i}\equiv\TL{N_i}(m)$, $\TL{i+j}\equiv\TL{N_i+N_j}(m)$, etc.; the first line is by definition, in the second line we used that $M_3 \cong \TL3\otimes_{\TL3}M_3$ and an obvious rearrangement of the tensor factors, to establish an isomorphism in the third line we used the associativity of the tensor product over rings --- the final result is obvious then.
In the third line, we consider the tensor product $M_1\tensor M_2\tensor M_3$ as the module over
 the algebra $\TL{1}\tensor\TL{2}\tensor\TL{3}$ which is considered as a subalgebra in $\TL{1+2+3}$, by the embedding~\eqref{aTL-embed}. 
 Explicitly, the isomorphism $\liso$ is given by
 \begin{equation}\label{eq:liso-def}
 \liso:\quad a\otimes(b\otimes m_1\otimes m_2)\otimes m_3 \mapsto a\cdot b\otimes (m_1 \otimes m_2\otimes m_3)\ .
 \end{equation}
 Similarly, we introduce the isomorphism $\riso$
of the right-hand side of~\eqref{fusfunc-TL-ass} 
 to the third line in~\eqref{ass:3lines-l}:
  \begin{equation}\label{eq:riso-def}
 \riso:\quad a\otimes\bigl(m_1\otimes (c \otimes  m_2 \otimes m_3)\bigr) \mapsto a\cdot c\otimes (m_1 \otimes m_2\otimes m_3)\ ,
 \end{equation} 
where $c\in\TL{N_2+N_3}(m)$. We have thus the inverse to $\riso$, and the map $\ass_{M_1,M_2,M_3}$ is defined as the composition $\riso^{-1}\circ \liso$ with the final result in~\eqref{TL-ass-map}.
 
 The naturality of $\ass_{M_1,M_2,M_3}$ is obvious.
 This finishes our proof of the proposition.
\end{proof}

Note that $\fus$ defines the tensor product also for morphisms: for two morphisms in $\catTL_N$ $f:M\to M'$ and $g:K\to K'$,  we say $f\fus g: M\fus K\to M'\fus K'$ for $\id_{\TL{N}}\otimes f\otimes g$ where $\otimes$ is in the category of vector spaces.

\newcommand{\fuss}{\times}
\begin{prop}\label{prop:ass-pentagon}
The family $\ass_{M_1,M_2,M_3}$ of natural isomorphisms of $\TL{N_1+N_2+N_3}$-modules from \eqref{fusfunc-TL-ass} satisfies the pentagon identity
\begin{multline}\label{eq:pentagon}
   \ass_{M_1, M_2, M_3\fus M_4}  \circ \ass_{M_1\fus M_2, M_3, M_4}\\
   =  (\id_{M_1} \fus \ass_{M_2,M_3,M_4})  \circ \ass_{M_1,M_2\fus M_3, M_4} \circ  (\ass_{M_1,M_2,M_3} \fus \id_{M_4})
\end{multline}
 or, equivalently, the ``pentagon'' diagram
 $$
\xygraph{ !{0;/r4.5pc/:;/u4.5pc/::}[]*+{((M_1\fus M_2)\fus M_3)\fus M_4}(
  :[u(1.1)r(1.7)]*+{(M_1\fus M_2)\fus(M_3\fus M_4)} ^{\mbox{}\!\!\!\!\!\!\!\!\!\!\!\!\!\!\!\!\!\ass_{M_1\fuss M_2,M_3, M_4}}
  :[d(1.1)r(1.7)]*+{M_1\fus(M_2\fus(M_3\fus M_4))} ="r" ^{\;\ass_{M_1, M_2, M_3 \fuss M_4}}
  ,
  :[r(.6)d(1.5)]*+!R(.7){(M_1\fus(M_2\fus M_3))\fus M_4} _{\ass_{M_1,M_2,M_3}\fus \id_{M_4}}
  :[r(2.2)]*+!L(0.7){M_1\fus((M_2\fus M_3)\fus M_4)} ^{\mbox{}\quad\ass_{M_1,M_2\fuss M_3,M_4}\quad}
  : "r" _{\id_{M_1}\fus\, \ass_{M_2,M_3,M_4}}
)
}
$$
 commutes (we use here the abbreviation $\fuss$ instead of $\fus$ in the indices of $\ass$). 
\end{prop}
\begin{proof}
We denote an element from $((M_1\fus M_2)\fus M_3)\fus M_4$ as 
\begin{equation}\label{abc-vect}
a\otimes \bigl(b\otimes (c\otimes m_1\otimes m_2)\otimes m_3\bigr)\otimes m_4,
\end{equation}
for $m_i\in M_i$, $i=1,2,3,4$, and $a\in\TL{1+2+3+4}$, $b\in\TL{1+2+3}$, and $c\in\TL{1+2}$.
(It is clear which tensor products in~\eqref{abc-vect} are the balanced tensor products over the  TL subalgebras, so we do not indicate them explicitly.)
Using~\eqref{TL-ass-map}, we begin with calculating the left-hand side of~\eqref{eq:pentagon} applied to such a vector: 
\begin{multline}\label{pentagon-TL-left}
a\otimes \bigl(b\otimes (c\otimes m_1\otimes m_2)\otimes m_3\bigr)\otimes m_4\\
\mapsto
a\cdot b\otimes \bigl((c\otimes m_1\otimes m_2)\otimes (\one_{3+4}\otimes m_3\otimes m_4)\bigr)\\
\mapsto
a\cdot b\cdot c \otimes  m_1\otimes \bigl( \one_{2+3+4} \otimes m_2\otimes (\one_{3+4}\otimes m_3\otimes m_4)\bigr)\ .
\end{multline}
On the other hand, the right side of~\eqref{eq:pentagon} gives
\begin{multline}
a\otimes \bigl(b\otimes (c\otimes m_1\otimes m_2)\otimes m_3\bigr)\otimes m_4\\
\mapsto
a\otimes \bigl((b\cdot c) \otimes m_1\otimes (\one_{2+3}\otimes m_2\otimes m_3) \bigr) \otimes m_4\\
\mbox{}\qquad\qquad\qquad \mapsto
a\cdot b\cdot c \otimes  m_1\otimes \bigl( \one_{2+3+4} \otimes (\one_{2+3} \otimes m_2 \otimes m_3) \otimes m_4\bigr)\\
\mapsto
a\cdot b\cdot c \otimes  m_1\otimes \bigl( \one_{2+3+4} \otimes m_2 \otimes (\one_{3+4} \otimes m_3 \otimes m_4)\bigr)\ ,
\end{multline}
which equals the third line of~\eqref{pentagon-TL-left}, and so the pentagon diagram indeed commutes.
\end{proof}

We have thus proven that $\catTL$ is a semi-group category (note that we have no tensor unit because we do not have the grade zero $N=0$ subcategory.) Our next step is the introduction of braiding isomorphisms in $\catTL$.

 \newcommand{\br}{c}

\subsection{Braiding for TL fusion}\label{sec:braiding-TL}
Motivated by a construction in~\cite{ReadSaleur07-1},
we introduce the braiding in $\catTL$ as follows. Let $\g_{N_1,N_2}$ defines the following element in $\TL{N_1 + N_2}$
\begin{equation}
\g_{N_1, N_2} = \left(g_{N_2}^{-1}\ldots g_2^{-1}g_1^{-1}\right) \left(g_{N_2+1}^{-1}\ldots g_2^{-1}\right)\ldots \left(g_{N_2+N_1-1}^{-1}\ldots g_{N_1}^{-1}\right)
\end{equation}
which passes strings from the left over those from the right 
(here, $N_1=3$, $N_2=2$):

\begin{equation}\label{sigma-def}
 \begin{tikzpicture}
 \node[font=\large]  at (-1.5,-4) {\mbox{} $\g_{N_1,N_2}\;\;\equiv$ \mbox{}\qquad};
\braid[braid colour=black,strands=5,braid width=\brw,braid start={(-0.3,-0.88)}]  {\g_2^{-1}\g_1^{-1}\g_3^{-1}\g_2^{-1}\g_4^{-1}\g_3^{-1}}

\node[font=\large]  at (4.0,-4) {\mbox{}\qquad \(=\)};
\draw[thick] (5.5,-2.5) -- (6.2,-3.9);
\draw[thick] (5.75,-2.5) -- (6.3,-3.6);

\draw[thick](6.65,-4.6)--(7.1,-5.5);
\draw[thick] (6.80,-4.4) -- (7.35,-5.5);

\draw[thick] (5.5,-5.5) -- (7,-2.5);
	\draw[thick] (5.75,-5.5) -- (7.25,-2.5);
         \draw[thick] (6,-5.5) -- (7.5,-2.5);
\end{tikzpicture}
\end{equation}
 and each braid-crossing (or equivalently $g_i^{\pm1}$) stands for the linear combination~\eqref{hecke-gen}.
Note that the element $\g_{N_1, N_2}$ defines an automorphism on $\TL{N_1+N_2}$ by the conjugation $a\mapsto \g_{N_1, N_2}\cdot a \cdot \g_{N_1, N_2}^{-1}$ which maps the subalgebra $\TL{N_1}\otimes\TL{N_2}$ (under the standard embedding) to the subalgebra $\TL{N_2}\otimes\TL{N_1}$ as 
\begin{equation}\label{g-TL-TL}
a\otimes b \; \mapsto \;  \g_{N_1, N_2} \cdot (a\cdot b)\cdot \g_{N_1, N_2}^{-1} = b \otimes a\ ,\qquad a\in\TL{N_1},\; b\in\TL{N_2}\ ,
\end{equation}
where $a\cdot b$ stands for the multiplication of $a$ and $b$ which are considered as the elements  in $\TL{N_1+N_2}$ under the standard embedding, and the last equality is most easily computed in terms of the diagrams.
Of course, the conjugation on $\TL{N_1+N_2}$ is non-trivial, which is easily seen for the generator $e_{N_1}$.

 Because of this flip of the two subalgebras,  the action by $\g_{N_1, N_2}$ relates  the two inductions from modules over these two subalgebras.
This allows us to  define the family of braiding isomorphisms on the $(N_1,N_2)$ graded components of~$\catTL$ as
\begin{equation}\label{br-TL}
\br_{M_1,M_2}: \quad M_1[N_1]\fus M_2[N_2] \xrightarrow{\quad \cong \quad} M_2[N_2] \fus M_1[N_1]
\end{equation}
by the conjugation with $\g_{N_1,N_2}$:
\begin{equation}\label{br-TL-1}
a\otimes m_1\otimes m_2 \; \mapsto \; \g^{}_{N_1,N_2}\cdot a \cdot \g_{N_1,N_2}^{-1} \otimes m_2\otimes m_1\ ,
\end{equation}
where $a\in\TL{N_1+N_2}$,  and $m_1\in M_1[N_1]$,  $m_2\in M_2[N_2]$. 
Here, we write $a\otimes m_1\otimes m_2$ for a representative in the corresponding class in $M_1[N_1]\fus M_2[N_2]$.
We first check that the map~\eqref{br-TL-1} is well-defined, i.e., does not depend on a representative in the class. Indeed, assume that $m_1=b \cdot m'_1$ and $m_2 = c\cdot m'_2$ for some $b\in\TL{N_1}$, $c\in\TL{N_2}$ and  some $m'_1\in M_1[N_1]$,  $m'_2\in M_2[N_2]$, and let us compute~\eqref{br-TL-1} for the other representative (setting here $\g\equiv \g_{N_1,N_2}$)
\begin{multline}\label{br-TL-2}
a\cdot (b\cdot c)\otimes m'_1\otimes m'_2 \; \mapsto \; \g \cdot a b c \cdot \g^{-1} \otimes m'_2\otimes m'_1 
= \g \cdot a \cdot g^{-1} \cdot (\g \cdot bc \cdot g^{-1}) \otimes m'_2\otimes m'_1 \\
=   \g \cdot a \cdot g^{-1}\otimes (c\otimes b) \otimes m'_2\otimes m'_1
= \g \cdot a \cdot \g^{-1} \otimes m_2\otimes m_1  \ ,
\end{multline}
where we used~\eqref{g-TL-TL} for the second equality, and the first and third $\otimes$'s in the second line are over the subalgebra $\TL{N_2}\otimes\TL{N_1}$.
The final result in~\eqref{br-TL-2} agrees with~\eqref{br-TL-1}.

 We emphasize  that the $\TL{N_1+N_2}$ action on the left-hand side of~\eqref{br-TL-1} is given by the multiplication while an element $a'\in\TL{N_1+N_2}$ acts by multiplication with $\g^{}_{N_1,N_2}\cdot a' \cdot \g_{N_1,N_2}^{-1}$ on the right-hand side of~\eqref{br-TL-1}, by the definition of its module structure which is given by applying the automorphism on the algebra. 
This shows the intertwining property of the map $\br_{M_1,M_2}$. This map is obviously bijective. We have thus proven that $\br_{M_1,M_2}$ is an isomorphism in $\catTL$.

Note finally that the family $\br_{M_1,M_2}$ satisfies the coherence (hexagon) conditions required for the braiding (we use conventions from  Kassel's book~\cite{Kassel}, see the hexagon conditions in eqs. (1.3)-(1.4) in Chapter 13.1.) 

\begin{prop}\label{prop:C-hexagon}
The family $\br_{M_1,M_2}$ of isomorphisms defined in~\eqref{br-TL}-\eqref{br-TL-1} satisfies the hexagon conditions:
\begin{equation}\label{eq:hexagon-1}
\ass_{M_2,M_3,M_1}\circ\br_{M_1,M_2\fus M_3}\circ\ass_{M_1,M_2,M_3} = (\id_{M_2}\fus \br_{M_1,M_3})\circ \ass_{M_2,M_1,M_3} \circ(\br_{M_1,M_2}\fus \id_{M_3})
\end{equation}
on the maps from $\bigl(M_1\fus M_2\bigr) \fus M_3$ to $M_2\fus\bigl(M_3\fus M_1\bigr)$ and
\begin{equation}\label{eq:hexagon-2}
\ass^{-1}_{M_3,M_1,M_2}\circ\br_{M_1\fus M_2, M_3}\circ\ass^{-1}_{M_1,M_2,M_3} =(\br_{M_1,M_3}\fus \id_{M_2}) \circ \ass^{-1}_{M_1,M_3,M_2} \circ (\id_{M_1}\fus \br_{M_2,M_3}) 
\end{equation}
on the maps from $M_1\fus \bigl(M_2 \fus M_3\bigr)$ to $\bigl(M_3\fus M_1\bigr)\fus M_2$.
\end{prop}
\begin{proof}
To show the equalities, we first note that the isomorphism $\ass_{M_1,M_2,M_3}$ from~\eqref{TL-ass-map} maps a representative to another representative in the same equivalence class corresponding to $a\cdot b\otimes m_1\otimes m_2\otimes m_3$. Therefore, the map~\eqref{fusfunc-TL-ass}
 on the set of the equivalence classes, which is the set $\bigl(M_1\fus M_2\bigr) \fus M_3$, is actually the identity map.
 The equality~\eqref{eq:hexagon-1} then follows from the identity $\g_{N_1,N_2+N_3} =(\one_{N_2}\otimes \g_{N_1,N_3})\cdot (\g_{N_1,N_2}\otimes \one_{N_3})$, where $\one_{N}$ is the unit on $N$ sites, while the equality~\eqref{eq:hexagon-2} holds because of the identity $\g_{N_1+N_2,N_3} =(\g_{N_1,N_3}\otimes \one_{N_2})\cdot(\one_{N_1}\otimes \g_{N_2,N_3})$. 
\end{proof}

We have thus proven the following theorem.
\begin{Thm}
The category $\catTL=\oplus_{N>0}\catTL_N$ is a braided semi-group category: it has the $\oN$-graded tensor product $\fus$ equipped with the associator $\ass_{A,B,C}$ defined in~\eqref{fusfunc-TL-ass}-\eqref{TL-ass-map} and satisfying the pentagon condition and equipped with the braiding isomorphisms  $\br_{A,B}$ defined  in~\eqref{br-TL}-\eqref{br-TL-1} that satisfy the hexagon conditions.
\end{Thm}

\subsection{Braided monoidal category $\catTLinf$}

We now recall the category $\catTLinf$ obtained in Sec.~\ref{sec:arc-tower} as the direct limit, see~\eqref{catTLinf} with~\eqref{catTLinf-0}, of a direct sequence of the categories $\catTL_N$ inside $\catTL$. The direct limit $\catTLinf$ is an abelian $\oC$-linear category by construction (for any $\q$ which is not $0$ or $\pm\rmi$). 
Our first objective is to study  subquotient structure of objects, e.g. the projective objects in $\catTLinf$. It turns out that there is an interesting correspondence between these projective objects and the so-called staggered representations of the Virasoro algebra (this will be discussed below in Sec.~\ref{sec:outlook}). Our second objective is to study different structures on $\catTLinf$, as tensor product, braiding, dual objects, etc.
The idea here is to use  the tensor product $\fus$ in $\catTL$, its associator $\ass_{M_1,M_2,M_3}$ and the braiding $\br_{M_1,M_2}$ and to show that these structures can be lifted to the limit category $\catTLinf$.

\subsubsection{Standard and projective objects in $\catTLinf$}
\label{sec:st-proj-Cinf}
For the first objective, we  recall the definition of the standard TL modules given above~\eqref{dj} and  
Prop.~\ref{prop:funGl-W} -- it suggests the following numeration of standard objects in the direct limit $\catTLinf$:
\begin{equation}\label{catTLinf-Wj}
\StTL{j} \; \equiv \; \{\, \StTL{j}[2j],\, \StTL{j}[2j+2],\, \ldots,\, \StTL{j}[N], \, \ldots\, \}\ , \qquad j\in\half\oN\ ,
\end{equation}
where the set in the figure brackets is the equivalence class of the objects in $\catTL$ connected by the globalization functors $\funGl_N$ for different values of $N$. 
In general, for any non-zero $\TL{N}$-mo\-du\-le $M[N]$ such that its localisation is zero, we define
\begin{equation}\label{catTLinf-M}
M \; \equiv \; \{\, M[N],\, M[N+2],\, \ldots\, \},\quad \text{with}\quad M[N+2]\equiv \funGl_N (M[N]),\quad \text{etc.,}
\end{equation}
as the corresponding equivalence class in the direct limit $\catTLinf$.

To study  the subquotient structure of $\StTL{j}$ and their projective covers, we recall first the definition of the space of morphisms in the direct limit category, see definitions below~\eqref{eq:dir-lim}: let $M_1, M_2\in \catTLinf$ then the vector space of morphisms is the set of equivalence classes
\begin{equation}\label{Hom-lim-def}
\Hom_{\catTLinf} (M_1, M_2)  = \coprod_{N\in\oN} \Hom_{\catTL_N} (M_1[N],M_2[N]) / \sim\ , 
\end{equation}
where the equivalence relations are defined as $f_1\sim f_2$ iff $f_2$ is the image of a sequence of the functors $\funGl_N$ applied to $f_1$. Therefore, to specify a morphism in $\Hom_{\catTLinf}(M_1,M_2)$ it is enough to choose a non-zero representative in  $\Hom_{\catTL_N} (M_1[N],M_2[N])$.

\medskip
Recall then the description of the abelian categories $\catTL_{N}$ in Sec.~\ref{sec:CN}.
It is clear that when $\q$ is not a root of unity the direct limit $\catTLinf$ is a semi-simple category and isomorphism classes of simple objects are exhausted by $\StTL{j}$ from~\eqref{catTLinf-Wj}.

Let $\q=e^{\rmi\pi/p}$  with integer $p \geq 3$ and recall the notation 
$s\equiv s(j)=(2j+1)\;\mathrm{mod}\; p$. By the definition of the $\Hom$ spaces in $\catTLinf$ we have a morphism from $\StTL{k}$ to $\StTL{j}$ iff $k=j$ or $k=j+p-s$ and the properties of the morphism for  $k=j+p-s$ are similar to the finite $N$ case: its kernel is the socle and its image is the socle as well.
Then, the subquotient structure of the standard objects $\StTL{j}$ is the following: 
if $s(j)=0$ the objects are simple while for non-zero $s(j)$
we have the subquotient structure
\begin{equation}\label{StTL-inf}
\StTL{j} \; = \; \IrTL{j} \longrightarrow \IrTL{j+p-s}
\end{equation}
where we introduce the notation $\IrTL{j}$ for the irreducible quotient of $\StTL{j}$ -- it is the equivalence class of the simple TL modules $\IrTL{j}[N]$ in $\catTL$ (note that  now there are no zero conditions on $\IrTL{k}$ as we had for finite $N$ cases).

We note further that the projective cover of $\IrTL{j}[N+2]$ is $\PrTL{j}[N+2] = \funGl_N (\PrTL{j}[N])$.
We have thus the projective objects
 in the direct limit $\catTLinf$:
\begin{equation}\label{catTLinf-Pj}
\PrTL{j} \; \equiv \; \{\, \PrTL{j}[2j],\, \PrTL{j}[2j+2],\, \ldots,\, \PrTL{j}[N], \, \ldots\, \}\ , \qquad j\in\half\oN\ ,
\end{equation}
as the equivalence classes of the objects in $\catTL$ connected by the globalisation functors $\funGl_N$ for different values of $N$.
Following the description of the projective covers in $\catTL_N$ around~\eqref{prTL-pic}, the projective objects $\PrTL{j}$ are then simple if $s(j)=0$, they are equal to $\StTL{j}$ for $0\leq j\leq \half(p-2)$ and otherwise have the following structure
\begin{align}\label{prTL-pic-inf}
   \xymatrix@C=5pt@R=25pt@M=2pt{%
    &&\\
    &\PrTL{j}\quad = &\\
    &&
 }
&  \xymatrix@C=20pt@R=20pt{%
    &{\IrTL{j}}\ar[dl]\ar[dr]&\\
    \IrTL{j-s}\ar[dr]&&\IrTL{j+p-s}\ar[dl]\\
    &\IrTL{j}&
 } 
\end{align}
That they are  projective covers of the simple objects $\IrTL{j}$ is easy to show by the definition of the direct limit (they are projective indecomposable and obviously cover $\IrTL{j}$).
Note that now in $\catTLinf$ we do not have the zero conditions as on the nodes for~\eqref{prTL-pic} and the subquotient structure for $\PrTL{j}$ has always four non-zero simple subquotients (if $s(j)\ne 0$ and $2j\geq  p$).

We have thus the reciprocity relation 
\begin{equation}
[\PrTL{j}:\StTL{j'}] = [\StTL{j'}:\IrTL{j}]
\end{equation}
in $\catTLinf$ as well (as in $\catTL_N$) and it is a highest-weight category with the standard objects described by~\eqref{catTLinf-Wj} and~\eqref{StTL-inf}.

\medskip

We will come back to the  $\StTL{j}$ and their projective covers $\PrTL{j}$ for $\q$ a root of unity at the last section~\ref{sec:outlook} where we discuss the connection with the Virasoro algebra representation theory.

\subsubsection{Tensor product in $\catTLinf$}
We now turn to our second objective in this section and study the monoidal structure on~$\catTLinf$.
Let us define the tensor product $\fuslim$ on $\catTLinf$ as
\begin{equation}\label{fuslim}
\fuslim:\qquad (M_1,M_2)\; \mapsto\; M_1\fuslim M_2 = \lim\Bigl[M_1[N_1]\fus M_2[N_2]\Bigr] \ ,
\end{equation}
where $\lim[\ldots]$ stands for taking the direct limit of the object (in this case, from $\catTL_{N_1+N_2}$) or the corresponding equivalence class (with respect to the functors $\funGl_N$). We show now that~\eqref{fuslim} is well-defined and does not depend on representatives in the classes $M_1$ and~$M_2$.

We begin by establishing a simple lemma.
\begin{lemma}\label{lem:TeTe}
Let $T$ be an associative algebra over a field $k$ with unit $\one$, and $I_e = T\cdot e$ is its left ideal generated by an idempotent $e\in T$. We then have
$$
I_e \cdot I_e = I_e\ .
$$
\end{lemma}
\begin{proof}
It is obvious that $I_e \cdot I_e\subset I_e$ because $I_e$ is the left ideal. On the other hand,  we also have  $I_e \subset I_e \cdot I_e$. Indeed, any element of the form $a\cdot e\in I_e$ can be rewritten in the form $(a \cdot e) \cdot (\one \cdot e) \in I_e\cdot I_e$. This proves the statement in the lemma.
\end{proof}

And now we use this lemma to prove yet another one.

\begin{lemma}\label{lem:Gl-class}
The TL modules $M_1[N_1]\fus M_2[N_2]$ and $M_1[N_1]\fus M_2[N_2+2]$ are in the same equivalence class in $\catTLinf$ or, equivalently, we have
\begin{equation}\label{eq:Gl-class}
\funGl_{N_1+N_2}:\qquad M_1[N_1]\fus M_2[N_2] \; \mapsto M_1[N_1]\fus M_2[N_2+2]\ . 
\end{equation}
\end{lemma}
\begin{proof}
By definition of $\fus$ and the functor $\funGl_N$, we set $N=N_1+N_2$ (also $\TL{1}\equiv\TL{N_1}$, etc.) and we have 
\begin{multline}\label{eq:Gl-M1M2}
 M_1[N_1]\fus M_2[N_2]  = \TL{N}\otimes_{(\TL1\otimes\TL2)}  M_1[N_1]\otimes M_2[N_2] \\
 \xrightarrow{\quad \funGl_{N}\quad } 
 \TL{N+2}\e \otimes_{\TL{N}}\bigl( \TL{N}\otimes_{(\TL1\otimes\TL2)} M_1[N_1]\otimes M_2[N_2] \bigr)
 \end{multline}
 and then we establish the following sequence of isomorphisms of the right-hand side of~\eqref{eq:Gl-M1M2}
 \begin{multline}\label{eq:Gl-M1M2-2}
 \text{RHS of~\eqref{eq:Gl-M1M2}}\; \xrightarrow{\; \text{associativity} \;}  
  \bigl( \TL{N+2}\e \otimes_{\TL{N}}\TL{N}\bigr) \otimes_{(\TL1\otimes\TL2)} M_1[N_1]\otimes M_2[N_2] \\
 \xrightarrow{\; \TL{N} = \e \TL{N+2}\e \;}  
   \TL{N+2}\e\TL{N +2}\e \otimes_{(\TL1\otimes\TL2)} M_1[N_1]\otimes M_2[N_2] \\
    \xrightarrow{\; \text{Lem.~\ref{lem:TeTe}} \;}   \TL{N+2}\e \otimes_{(\TL1\otimes\TL2)} M_1[N_1]\otimes M_2[N_2]\ ,
 \end{multline}
 where the second isomorphism is due to the first balanced tensor product in the first line which is over the subalgebra $\TL{N} =  \e \TL{N+2}\e$, while we used Lem.~\ref{lem:TeTe} for the third isomorphism. Note that $\TL{2}$ in the third line stands for $\e \TL{N_2+2}\e$ and $\e$ is considered as the corresponding idempotent in the subalgebra $\TL{N_2+2}\subset\TL{N+2}$, also the module $M_2[N_2]$ has to be considered as the corresponding  module over $\e \TL{N_2+2}\e$. Then, we rewrite  $\TL{N+2}$ as $\TL{N+2}\otimes_{(\TL{N_1}\otimes\TL{N_2+2})}\TL{N_1}\otimes\TL{N_2+2}$ and establish further isomorphisms
  \begin{multline}\label{eq:Gl-M1M2-3}
 \!\! \!\! \!\!  \text{RHS of~\eqref{eq:Gl-M1M2-2}}
    \xrightarrow{ \;\cong \; }   \TL{N+2}\otimes_{(\TL{N_1}\otimes\TL{N_2+2})} \Bigl( \bigl(\TL{N_1}\otimes\TL{N_2+2}\e\bigr) \otimes_{(\TL{1}\otimes\TL{2})} M_1[N_1]\otimes M_2[N_2] \Bigr)\\
     \xrightarrow{\quad \cong \quad }  \TL{N+2}\otimes_{(\TL{N_1}\otimes\TL{N_2+2})}   M_1[N_1]\otimes \bigl( \TL{N_2+2}\e\otimes_{\TL{2}} M_2[N_2] \bigr)   \ ,
\end{multline}
where we used simple rearrangement of the tensor factors and that $\TL{N_1}\otimes_{\TL1}M_1[N_1] = M_1[N_1]$. The right-hand side of~\eqref{eq:Gl-M1M2-3} obviously equals to $M_1[N_1]\fus \bigl(\funGl_{N_2}(M_2[N_2])\bigr)$ which is $M_1[N_1]\fus M_2[N_2+2]$ by the definition of  $M_2[N_2+2]$. We have thus established an isomorphism between the right-hand sides of~\eqref{eq:Gl-M1M2} and~\eqref{eq:Gl-M1M2-3}. Finally, note that the $\TL{N+2}$ actions on these two spaces are equal and not just isomorphic and the composition of the maps is the identity map. 
 
Indeed, the isomorphism from the right-hand side of~\eqref{eq:Gl-M1M2} (or $\funGl_N(M_1[N_1]\fus M_2[N_2])$) to the right-hand side of~\eqref{eq:Gl-M1M2-3} is given explicitly by the map
\begin{equation}\label{eq:kappa}
\varkappa: \quad a \e \otimes b \otimes m_1\otimes m_2\; \mapsto \; a \cdot b \otimes m_1\otimes (\one_{N_2+2}\e \otimes m_2)\ ,
\end{equation}
where $a\in\TL{N+2}$ and $b\in\TL{N}$, and $m_i\in M_i[N_i]$, and we label the generators $e_j$ of $\TL{N_2+2}$ from $j=N_1+1$ to $N+1$, as usual for the ``right'' subalgebra, so $\e$ is in the subalgebra $\TL{N_2+2}$.
Both the elements (in LHS and RHS of~\eqref{eq:kappa}) are just different representatives of the same equivalence class, which is the element of the balanced tensor product in the third line of~\eqref{eq:Gl-M1M2-2}. The map $\varkappa$ is therefore the identity map on the set of the equivalence classes.
This finishes our proof.
\end{proof}

 Equivalently, we have the following lemma for the adjoint functors. 
\begin{lemma} \label{lem:loc-fus}
Let $N=N_1+N_2$ and  $M_1\in\catTL_{N_1}$ and $M_2\in\catTL_{N_2+2}$. The localisation functor has the following property:
$$
\funLoc_{N_1+N_2}: \quad M_1[N_1]\fus M_2[N_2+2] \;  \mapsto \; M_1[N_1]\fus \e M_2[N_2+2] \ .
$$
\end{lemma}
\begin{proof}
We have by definition
$$
\funLoc_{N_1+N_2}\bigl(M_1[N_1]\fus M_2[N_2+2]\bigr) =   \e \TL{N+2}\tensor_{\bigl(\TL{N_1}\tensor\TL{N_2+2}\bigr)} M_1[N_1]\tensor M_2[N_2+2].
$$
Note then that the latter expression is isomorphic to
\begin{equation}\label{eq:loc-fus}
 \e \TL{N+2}\e \tensor_{\bigl(\TL{N_1}\tensor \e \TL{N_2+2}\e\bigr)} \e \Bigl( M_1[N_1]\tensor M_2[N_2+2]\Bigr)
\end{equation}
because we have (by Lem.~\ref{lem:TeTe})
$$
\e\TL{N}\e \TL{N} \cong \e \TL{N}.
$$
The expression in~\eqref{eq:loc-fus} obviously coincides with $M_1[N_1]\fus \e M_2[N_2+2]$ and thus finishes the proof.
\end{proof}

Applying repeatedly  Lem.~\ref{lem:Gl-class}, we see that the TL modules $M_1[N_1]\fus M_2[N_2]$ and $M_1[N_1]\fus M_2[N_2+2n]$, for integer $n$, are in the same equivalence class, as an object  in~$\catTLinf$.  Now, we would like to vary  the index $N_1$. For this, recall the definition of the left-arc embeddings using the idempotents $\el$ and the corresponding globalisation functors $\funGll_N$ introduced in Rem.~\ref{rem:funGll}. Repeating arguments in the proof of Lem.~\ref{lem:Gl-class} we obtain
\begin{equation}\label{eq:Gl-class-l}
\funGll_{N_1+N_2}:\qquad M_1[N_1]\fus M_2[N_2] \; \mapsto  \funGll_{N_1}(M_1[N_1])\fus M_2[N_2]\ . 
\end{equation}
 Finally, note that the module $\funGll_{N_1}(M_1[N_1])$ is identical to $\funGl_{N_1}(M_1[N_1])$, as the $\TL{N_1+2}$ actions are equal
 (the use of $\el$ instead of $\e$ is a matter of  convention).
We thus obtain that the TL modules $M_1[N_1]\fus M_2[N_2]$ and $M_1[N_1+2n]\fus M_2[N_2]$, for integer $n$, are in the same equivalence class in $\catTLinf$.  Altogether, we 
conclude with the following corollary.

\begin{cor}\label{cor:M1M2}
The TL modules $M_1[N_1]\fus M_2[N_2]$ and $M_1[N_1+2n]\fus M_2[N_2+2m]$, for integer $n$ and $m$, are in the same equivalence class in the direct limit $\catTLinf$ and thus just different representatives of the same object in $\catTLinf$. 
 If $m=-n$, the two $\TL{N_1+N_2}$ modules are even identical in $\catTL_{N_1+N_2}$.
Their equivalence class is by definition the fusion $M_1\fuslim M_2$ of two classes $M_1$ and $M_2$ in $\catTLinf$. 
\end{cor}

We next define the associator and the braiding for the tensor product $\fuslim$ in $\catTLinf$. Recall that to specify a morphism in $\Hom_{\catTLinf}(M_1,M_2)$ it is enough to choose a representative in  $\Hom_{\catTL_N} (M_1[N],M_2[N])$. 

\begin{dfn}\label{def:asslim}
For a triple of objects $M_1$, $M_2$, $M_3$ in $\catTLinf$ we define an isomorphism
\begin{equation}\label{asslim}
\asslim_{M_1,M_2,M_3}: \quad \bigl(M_1\fuslim M_2\bigr) \fuslim M_3 \xrightarrow{\quad \cong\quad}
M_1\fuslim \bigl( M_2 \fuslim M_3\bigr) 
\end{equation}
as follows: take any triple of positive integers $(N_1,N_2,N_3)$ such that $M_i[N_i]$ are non-zero, for $i=1,2,3$, and take then the corresponding associator
$\ass_{M_1[N_1],M_2[N_2],M_3[N_3]}$ defined in Prop.~\ref{prop:ass}, then $\asslim_{M_1,M_2,M_3}$ is its equivalence class, i.e. the corresponding element in the quotient~\eqref{Hom-lim-def}.
\end{dfn}

Of course, we have to show that our $\asslim$ is well-defined and does not depend on the choice of representatives in the $\Hom$ spaces. For this, we first note by Cor.~\ref{cor:M1M2} that $\bigl(M_1[N_1]\fus M_2[N_2]\bigr) \fus M_3[N_3]$ is the same object in $\catTL_{N}$ for any choice of $(N_1,N_2,N_3)$ such that $N=N_1+N_2+N_3$, and similarly for the other bracketing. The corresponding isomorphisms $\ass_{M_1[N_1],M_2[N_2],M_3[N_3]}$ for the different choices of $N_i$ are also identical. We then only need to show that the associators for   $(N_1,N_2,N_3)$ and  $(N_1,N_2,N_3+2)$ are in the same equivalence class, i.e. the second is the image of the first through the functor $\funGl_N$. Recall that $\funGl_N(f) = \id_{\TL{N+2}}\fus f$.
To show this, we first calculate  the isomorphism $\funGl_N(\ass_{M_1[N_1],M_2[N_2],M_3[N_3]})$ on a general element from $\funGl_N\bigl((M_1[N_1]\fus M_2[N_2]) \fus M_3[N_3]\bigr)$ as
\begin{equation}\label{eq:id-ass}
\id\otimes \ass: \quad c\e \otimes \bigl(a\otimes (b\otimes m_1\otimes m_2)\otimes m_3\bigr) \; \mapsto \;  c\e \otimes a \cdot b\otimes \bigl(m_1\otimes (\one_{N_2+N_3}\otimes m_2\otimes m_3) \bigr)
\end{equation}
and it can be further rewritten as (by applying the identity map $(\id\otimes\varkappa)\circ\varkappa$, see the definition in~\eqref{eq:kappa})
$$
c \cdot a \cdot b \otimes m_1 \otimes \bigl(\one_{N_2+N_3+2}\otimes m_2 \otimes (\one_{N_3+2}\e\otimes m_3)\bigr) \ . 
$$
We then note that the last expression coincides with the image of $\ass_{M_1[N_1],M_2[N_2],M_3[N_3+2]}\circ\varkappa$ on the left-hand side of~\eqref{eq:id-ass}.
We have thus shown explicitly (on representatives in the balanced tensor products) the equality 
$$
\funGl_N(\ass_{M_1[N_1],M_2[N_2],M_3[N_3]}) = \ass_{M_1[N_1],M_2[N_2],M_3[N_3+2]}
$$
and this finishes our proof that Def.~\ref{def:asslim} is well-defined. It is obvious by the construction that the family of isomorphisms in~\eqref{asslim} satisfies the pentagon identities because each representative does, recall Prop.~\ref{prop:ass-pentagon}.

\begin{dfn}\label{def:br-lim}
In $\catTLinf$, we similarly define  the braiding by the family of isomorphisms
\begin{equation}\label{br-TLlim}
\brlim_{M_1,M_2}: \quad M_1\fuslim M_2 \xrightarrow{\quad \cong \quad} M_2 \fuslim M_1
\end{equation}
as the  equivalence  class corresponding  to $\br_{M_1,M_2}$ from $\Hom_{\catTL_N} (M_1[N_1]\fus M_2[N_2], M_2[N_2]\fus M_1[N_1])$, which is defined in~\eqref{br-TL}, for a choice of $N_1$ and $N_2$ such that both $M_1[N_1]$ and $M_2[N_2]$ are non-zero. 
\end{dfn}

Like in  the discussion below the definition of the associator in $\catTLinf$, we show that~\eqref{br-TLlim} does not depend on the choice of $(N_1,N_2)$, or  Def.~\ref{def:br-lim} is well-defined. 
For this, it is enough to show that the two braidings (denote them for brevity as) $c_1$ and $c_2$, one for $M_1[N_1]\fus M_2[N_2]$ and the second for $M_1[N_1]\fus M_2[N_2+2]$, are in the same equivalence class, or $\funGl_{N_1+N_2}(c_1)=c_2$. We check the equality by a direct calculation on the representatives in the tensor product $\funGl_{N_1+N_2}\bigl(M_1[N_1]\fus M_2[N_2]\bigr)$.

Finally, the braiding isomorphisms~\eqref{br-TLlim} satisfy the hexagon identities because each representative does, recall Prop.~\ref{prop:C-hexagon}.
We have thus proven  the following theorem.

\begin{Thm}\label{thm:catTLinf}
The direct-limit category $\catTLinf$ defined in~\eqref{catTLinf} with~\eqref{catTLinf-0} is a braided monoidal category with  the  tensor product $\fuslim$ given in~\eqref{fuslim}, with the tensor unit $\StTL0$, with the associator $\asslim_{A,B,C}$ introduced in Def.~\ref{def:asslim}
and with the braiding isomorphisms  $\brlim_{A,B}$ defined in Def.~\ref{def:br-lim}.
\end{Thm}

We only need to comment on the tensor unit $\StTL{0}$ -- it is the equivalence class of the standard TL modules with zero number of through-lines. The tensor unit properties of $\StTL{0}$ follows from the finite TL fusion
(with even $N_1$ of course)
$$
\StTL{0}[N_1]\fus M[N_2] \cong  M[N_2]\fus \StTL{0}[N_1] \cong M[N_1+N_2]\ .
$$
It is easy to prove for $N_1=2$ that $M[N_2]\fus \StTL{0}[2]$ actually equals to $\funGl_{N_2}(M[N_2])=M[N_2+2]$ and then one proceeds by the induction in $N_1$ using $\StTL{0}[2]\fus\StTL{0}[2] = \StTL{0}[4]$ and the associativity of the fusion.

\subsection{On rigidity and non-rigidity of $\catTLinf$}

We can also introduce (right) duals for the standard objects $\StTL{j}$ in $\catTLinf$ (though not in $\catTL$ because $\catTL$ does not have the tensor unit) as follows. We introduce first contragredient objects:  for an object $M\in\catTLinf$ let us pick up its representative $M[N]$, we then define the contragredient object $M^*$ as the equivalence class corresponding to the space of linear forms $\bigl(M[N]\bigr)^*$ where the $\TL{N}$ action is given with the help of the anti-involution reflecting a TL diagram along the horizontal line in the middle of the diagram. (Note that the definition does not depend on the choice of $N$.) We then introduce right dual for each standard objects $\StTL{j}$, with $j\in\half\oN$, as the  contragredient object $\StTL{j}^*$ -- this duality is equipped with an evaluation map $\mathrm{ev}:\; \StTL{j}^*\fuslim\StTL{j}\to \StTL{0}$ and a coevaluation map  $\mathrm{coev}:\; \StTL{0} \to \StTL{j}\fuslim\StTL{j}^*$  (they can be explicitly fixed on representatives at finite $N$ using the diagrammatical formulation of the fusion) satisfying the zig-zag rules. One can similarly introduce right duals  
for any other objects which are filtered by the standard objects, e.g. for the projective covers. For generic $\q$, this gives thus duals for all objects in $\catTLinf$ and the category is actually rigid.

The problem appears at $\q$ a root of unity:
 the  contragredient objects $\IrTL{j}^*$ do not give right duals to $\IrTL{j}$ for  $0\leq j\leq \half(p-2)$. Moreover for these values of $j$, the simple objects $\IrTL{j}$ can not have duals (for our choice of the tensor unit) because the $\IrTL{j}$'s form a tensor ideal that does not contain our tensor unit $\StTL{0}$, e.g. for $p=3$ we compute directly $\IrTL{0}^*[N]\fus\IrTL{0}[N]= \IrTL{0}[2N]$ which is the simple (one-dimensional) quotient of $\StTL{0}[2N]$.  We thus conclude that $\catTLinf$ is not rigid (for $\q=e^{\rmi\pi/p}$ and  $p=3,4,5,\dots$). 

We note that the presence of such a tensor ideal that spoils the rigidity is an interesting property which is  in common with representation theory of the Virasoro algebra at critical central charges, like $c=0$, to be discussed in Sec.~\ref{sec:outlook}.

\medskip

We are now going to use the approach elaborated in this section in the case of affine TL algebra representations. Our proofs of statements in this section were designed in such a way that  the generalization to the affine case is straightforward.

\section{A semi-group affine TL category}\label{semiGaffTL}

Recall that in the finite TL case we used the idempotent subalgebras to construct direct sequences of TL representation categories and their limit $\catTLinf$. 
 For the affine TL case, we obtain surprisingly  analogous statement by using the same idempotent.
 
\begin{Prop}
Let $m$ be non zero. Introduce the idempotent $\e=\frac{1}{m}e_{N+1}$. Then, there is an isomorphism
$$
\psi: \quad \ATL{N}\xrightarrow{\quad \cong \quad} \e\ATL{N+2}\e,
$$
such that the generators of $\ATL{N}$ are mapped as
\begin{align}
u^{\pm1}&\mapsto m \e u^{\pm1}  \e,\label{psi-image-1}\\
e_j&\mapsto \e e_j\ , \qquad 1\leq j\leq N-1,\label{psi-image-2}\\
e_N &\mapsto m \e e_N e_{N+2} \e.\label{psi-image-3}
\end{align}
\end{Prop}
\begin{proof}
To prove that $\psi$ is a homomorphism of algebras is a straightforward use of the relations in $\ATL{N+2}$. For instance, under the mapping we have
\begin{eqnarray}
u^2 e_{N-1}&\mapsto& m^{-1} e_{N+1} u e_{N+1}ue_{N+1}e_{N-1}\nonumber\\
&=&m^{-1} e_{N+1} e_{N+2}u^2e_{N+1}e_{N-1}\nonumber\\
&=& m^{-1} e_1e_2\ldots e_{N-1} e_{N+1}\label{checku2}
\end{eqnarray}
where the first relation  in (\ref{TLpdef-u2}) (but for  $\ATL{N+2}$) was used to go from the first to the second line, and similarly the second relation in (\ref{TLpdef-u2})  was used to go from the second  to the third line. Meanwhile, it is easy to check that
\begin{equation}
e_1\ldots e_{N-1}\mapsto  m^{-1} e_1e_2\ldots e_{N-1} e_{N+1}
\end{equation}
as well, hence checking (\ref{TLpdef-u2}) in the image of $\ATL{N+2}$.

 To prove that the kernel of $\psi$ is zero, we use the graphical representation of the images in~\eqref{psi-image-1}-\eqref{psi-image-3}, some of which are shown below for $N=4$: 
 
 \begin{equation}\label{u1-def}
\begin{tikzpicture}

 \node[font=\large]  at (-1.5,-.7) {\mbox{} $u=$ \mbox{}\qquad};

\node[font=\large]  at (3.0,-.7) {\mbox{}\qquad \(\mapsto~~{1\over m}\)};
	\draw[thick] (0.3,0.) arc (0:10:20 and -3.75);
	\draw[thick] (2.1,-1.31) arc (0:10.6:-30 and 3.75);

	\draw[thick] (0.9,0.) arc (0:10:40 and -7.6);
	\draw[thick] (1.5,0.) arc (0:10:40 and -7.6);
	\draw[thick] (2.1,0.) arc (0:10:40 and -7.6);

\draw[thick] (5.3,1) arc (0:10:20 and -3.75);

	\draw[thick] (5.9,0.5) arc (0:10:40 and -7.6);
	\draw[thick] (6.5,0.5) arc (0:10:40 and -7.6);
	\draw[thick] (7.1,0.5) arc (0:10:40 and -7.6);
\draw[thick] (7.4,-0.1) arc (0:10:20 and -3.75);

\draw[thick] (8.,-.1) arc (0:180:0.3 and 0.45);

\draw[thick] (8.3,-0.65) arc (0:10.6:-30 and 3.75);

\draw[thick] (8.2,1) arc (0:180:0.3 and -0.45);

\draw[thick] (8.,0.) arc (0:10:20 and -3.75);

\draw[thick] (5.9,0.5) -- (5.9,1);

\draw[thick] (6.5,0.5) -- (6.5,1);
\draw[thick] (7.1,0.5) -- (7.1,1);

\draw[thick] (5.3,-0.8) -- (5.3,-2.5);
\draw[thick] (5.9,-0.8) -- (5.9,-2.5);
\draw[thick] (6.5,-0.8) -- (6.5,-2.5);

\draw[thick] (7.1,-0.75) -- (7.1,-2.5);

\draw[thick] (8.3,-1) arc (0:180:0.3 and -0.45);
\draw[thick] (8.3,-1) -- (8.3,-0.6);

\draw[thick] (7.7,-1) -- (7.7,-0.65);

\draw[thick] (8.2,-2.2) arc (0:180:0.3 and 0.45);
\draw[thick] (8.2,-2.2) -- (8.2,-2.5);
\draw[thick] (7.6,-2.2) -- (7.6,-2.5);

\end{tikzpicture}
\end{equation}

\begin{equation}\label{u2-def}
 \begin{tikzpicture}
        \node[font=\large]  at (-10,1) {$e_N~=~$};
     
   \draw[thick] (-8.2,0.5) .. controls (-8.3,0.68)   .. (-8.5,0.7);
   \draw[thick] (-8.2,-0.2) -- (-8.2,0.5);
  \draw[thick] (-8.2,1.4) -- (-8.2,2.1);
       \draw[thick] (-8.5,1.1) .. controls (-8.35,1.15)   .. (-8.2,1.4);
      
        \draw[thick] (-6.5,-0.2) -- (-6.5,0.5);
          \draw[thick] (-6.5,1.4) -- (-6.5,2.1);
       \draw[thick] (-6.5,0.5) .. controls (-6.4,0.68)   .. (-6.2,0.7);
         \draw[thick] (-6.5,1.4) .. controls (-6.4,1.15)   .. (-6.2,1.15);

\draw[thick] (-7.6,-0.2) -- (-7.6,2.1);
\draw[thick] (-7.,-0.2) -- (-7.,2.1);
  \node[font=\large]  at (-5,1) {$\mapsto~~{1\over m}$};
  
    \draw[thick] (-3.2,-0.1) .. controls (-3.3,0.08)   .. (-3.5,0.1);
   \draw[thick] (-3.2,-0.1) -- (-3.2,-2.2);
  \draw[thick] (-3.2,0.8) -- (-3.2,4.2);
       \draw[thick] (-3.5,0.5) .. controls (-3.35,0.55)   .. (-3.2,0.8);

\draw[thick] (-2.6,-2.2) -- (-2.6,4.2);
\draw[thick] (-2.,-2.2) -- (-2.,4.2);

  \draw[thick] (-0.3,-0.8) -- (-0.3,-0.1);
          \draw[thick] (-0.3,0.8) -- (-0.3,2.8);
      \draw[thick] (-0.3,-0.1) .. controls (-0.2,0.08)   .. (0,0.1);
         \draw[thick] (-0.3,0.8) .. controls (-0.2,0.55)   .. (0,0.55);

\draw[thick] (-1.4,-2.2) -- (-1.4,1.5);
\draw[thick] (-1.4,2.7) -- (-1.4,4.2);
\draw[thick] (-0.8,-0.8) -- (-0.8,1.5);

\draw[thick] (-0.3,-0.8) arc (0:180:0.25 and -0.5);

\draw[thick] (-0.3,-2.2) arc (0:180:0.25 and 0.5);


\draw[thick] (-0.8,1.5) arc (0:180:0.3 and 0.5);

\draw[thick] (-0.8,2.8) arc (0:180:0.3 and -0.5);

\draw[thick] (-0.3,2.8) arc (0:180:0.25 and 0.5);

\draw[thick] (-0.3,4.2) arc (0:180:0.25 and -0.5);

\end{tikzpicture}
\end{equation}

\end{proof}

We can thus consider $\ATL{N}$ as the idempotent subalgebra in $\ATL{N+2}$. This allows us similarly to Sec.~\ref{sec:arc-tower} to define two functors between the categories of affine TL modules. Let $\catATL_N$ denotes the category of finite-dimensional $\ATL{N}$-modules. We introduce the localisation functor
 \begin{equation}\label{loc-fun-a}
 \afunLoc_N: \quad \catATL_{N+2} \to \catATL_{N}\qquad \text{such that}\qquad M \mapsto \e M \ ,
 \end{equation} 
 with an obvious  map on morphisms, and its right inverse, so called globalisation functor
 \begin{equation}\label{glob-fun}
 \afunGl_N: \quad \catATL_{N} \to \catATL_{N+2}\qquad \text{such that}\qquad M \mapsto \ATL{N+2}\e \otimes_{\ATL{N}} M\ .
 \end{equation} 

The composition $\afunLoc_N\circ \afunGl_N$ is naturally isomorphic to the identity functor on $\catATL_N$. 
Similarly to Prop.~\ref{prop:Ie},  for the reverse composition we establish the analogous result for $\ATL{N}$ (the proof is just a repetition of the proof of Prop.~\ref{prop:Ie} replacing $\TL{N}$ by $\ATL{N}$, etc.)

\begin{Prop}\label{prop:Ie-affine}
The composition $\afunGl_N\circ \afunLoc_N$ maps a $\ATL{N+2}$-module $M$ to $I_{\es}\cdot M$, where $I_{\es}$ is the two-sided ideal generated by $\e$ in $\ATL{N+2}$.
\end{Prop}

\begin{rem}\label{rem:Ie-affine-span}
Note that the ideal $I_{\es}$  generated by $\e$ in $\ATL{N+2}$ is spanned by all affine TL diagrams except the powers of the translation generator $u^n$, for $n\in\oZ$. This is easy to see in terms of the generators of the subalgebra $I_\es$: all $e_i$, with $i\in\oZ_{N+2}$, are in $I_{\es}$.
\end{rem}

We study then properties of the two functors $\afunLoc_N$  and $\afunGl_N$  with respect to the standard modules introduced in Sec.~\ref{sec:Wz}.
As in the finite TL case, the localisation functors $\afunLoc_N$  send the standard (resp., costandard) modules to the standard (resp., costandard) modules of the same weight $(j,z)$.

\begin{Prop}\label{prop:eWz}
For any non-zero $\q$, the  $\ATL{N}$-module $\e \StJTL{j}{z}[N+2]$ is equal to $\StJTL{j}{z}[N]$ for $j\leq N/2$ and 0 otherwise.
\end{Prop}
\begin{proof}
We begin with determining $\e \StJTL{j}{z}[N+2]$ as a module over the  TL subalgebra $\TL{N}\subset\ATL{N}$. For this, we use results of~\cite[Sec.~2.12]{GL} and decompose (for a generic value of $\q$)
$$
\StJTL{j}{z}[N+2] = \bigoplus_{k=j}^{(N+2)/2} \StTL{k}[N+2]
$$ 
as a module over the  $\TL{N+2}$ subalgebra. For $\q$ a root of unity, the direct sum is replaced by the filtration by $\StTL{k}$ modules such that $\StTL{j}$ is a submodule in $\StJTL{j}{z}$ (it is the span of affine TL diagrams of  rank-$0$),  $\StTL{j+1}$ is a submodule in the quotient $\StJTL{j}{z}/\StTL{j}$ (it is the span  of affine TL diagrams of rank-$1$), etc.   Note then that $\e$ is in the subalgebra  $\TL{N+2}$. So, the problem reduces to the finite TL problem: we have to compute $\oplus_{k=j}^{(N+2)/2} \e \StTL{k}[N+2]$ (or, equivalently, the action of $\e$ on each section in the filtration by $\StTL{k}$'s if $\q$ is a root of unity). For this, we  use the finite TL result from Prop.~\ref{prop:eW}: $\e \StTL{k}[N+2] = \StTL{k}[N]$ and obtain 
$$
\e \StJTL{j}{z}[N+2] = \bigoplus_{k=j}^{N/2} \StTL{k}[N]
$$ 
as the $\TL{N}$ module (or the corresponding filtration by $\StTL{k}$'s for $\q$ a root of unity). We have thus $\e \StJTL{j}{z}[N+2] \cong \StJTL{j}{z'}[N]$ for some complex number $z'$.

The fact that the parameter $z'=z$  can be proven by considering powers of the translation generator. We known by definition that in $\StJTL{j}{z}[N+2]$ there is the relation $u^{N+2}=z^{2j}\one$. Consider now the image of $u^N$ in $\StJTL{j}{z}[N]$. We have
\begin{eqnarray}
u^N&\mapsto& m^{-1} (e_{N+1} u)^Ne_{N+1}\nonumber\\
&=& m^{-1} ue_N (e_{N+1}u)^{N-1}e_{N+1}\nonumber\\
&=& m^{-1} u^2 e_{N-1}e_N (e_{N+1}u)^{N-2} e_{N+1}\nonumber\\
&=&\ldots\nonumber\\
&=& m^{-1} u^N e_1e_2\ldots e_N e_{N+1}
\end{eqnarray}
where the relation $e_j u=ue_{j-1}$ in $\ATL{N+2}$ was repeatedly used. We can now replace the product of Temperley--Lieb elements on the right using the second relation in first relation  in (\ref{TLpdef-u2}) for $\ATL{N+2}$, leading to
\begin{equation}
u^N\mapsto u^{N+2} \e=z^{2j}\e
\end{equation}
as required. 
This finishes our proof of the proposition.
\end{proof}

Then, we prove the following property of the globalisation functor in the affine case, the proof  repeats the one for Prop.~\ref{prop:funGl-W} with the use of Prop.~\ref{prop:eWz}.

\begin{prop}\label{prop:afunGl-W}
For $x\leq j\leq \frac{N}{2}$ and  $x = \half(N\,\mathrm{mod}\, 2)$, we have  
\begin{equation}\label{eq:afunGl-St}
\afunGl_N: \qquad \StATL{j,z}[N]\; \mapsto \; \StATL{j,z}[N+2]\ .
\end{equation}
\end{prop}

\subsection{Associativity of the affine TL fusion}
Similarly to the finite TL case, we introduce  the ``enveloping'' affine TL category
\begin{equation}
\catATL = \bigoplus_{N\geq 1} \catATL_{N}
\end{equation}
where $\catATL_{N}$ is a full subcategory and there are no morphisms between the full subcategories for different $N$. The category $\catATL$ is thus graded by $\oN$. We will label  an object $M$ from $\catATL_N$ as $M[N]$ to emphasize its grade.

As in the previous section, our next step is to introduce a bilinear $\oN$-graded tensor product (bi-functor) on $\catATL$ and to show that it defines an associative tensor product  in  the  direct limit category  $\catATLinf$ introduced below. 
We will use then Prop.~\ref{prop:eWz} for studying the affine TL fusion rules in the  direct limit.

Recall that in~\eqref{fusfunc-def} we have defined the affine TL fusion bi-functor
\begin{equation}\label{fusfunc-ATL-def-0}
\afus: \quad \catATL_{N_1}\times \catATL_{N_2}\to \catATL_{N_1+N_2}
\end{equation}
on two modules $M_1$ and $M_2$ as the   induced module. It obviously respects the $\oN$ grading.
The important property of the fusion $\afus$ is the associativity.

\begin{prop}\label{prop:ass-ATL}
Let $M_1$, $M_2$ and $M_3$ be three modules over $\ATL{N_1}(m)$, $\ATL{N_2}(m)$ and $\ATL{N_3}(m)$, respectively.
The tensor product~$\afus$ is equipped with an associator, {\it i.e.,} we have a family $\aass_{M_1,M_2,M_3}$ of natural isomorphisms of $\ATL{N_1+N_2+N_3}(m)$ modules
\begin{equation}\label{fusfunc-ass}
\aass_{M_1,M_2,M_3}:\quad \bigl(M_1\afus M_2\bigr) \afus M_3 \xrightarrow{\quad \cong\quad}
M_1\afus \bigl( M_2 \afus M_3\bigr) 
\end{equation}
given explicitly as, for any triple of vectors $m_i\in M_i$, with $i=1,2,3$, 
\begin{equation}\label{ATL-ass-map}
\aass_{M_1,M_2,M_3}:\quad a\otimes(b\otimes m_1\otimes m_2)\otimes m_3 \mapsto a\cdot b\otimes \bigl(m_1\otimes (\one_{2+3}\otimes m_2\otimes m_3)\bigr)\ ,
\end{equation}
where $a\in\ATL{N_1+N_2+N_3}(m)$, $b\in\ATL{N_1+N_2}(m)$, and $\one_{2+3}$ is the identity in $\ATL{N_2+N_3}$, and $a\cdot b$ stands for the product of
the element  $a$ with the image of $b$ under the $\varepsilon_{N_1+N_2,N_3}$ embedding of $\ATL{N_1+N_2}$ into $\ATL{N_1+N_2+N_3}$ defined in~\eqref{aTL-embed}.

The isomorphisms $\aass_{M_1,M_2,M_3}$ satisfy the pentagon identity.
\end{prop}
\begin{proof}
The proof essentially repeats the poof of Prop.~\ref{prop:ass}: we replace $\TL{N}$ by $\ATL{N}$ and $\fus$ by $\afus$, and our manipulations in~\eqref{ass:3lines-l} with the balanced tensor products are valid for the infinite-dimensional algebras, and define the maps $\liso$ and $\riso$ as in~\eqref{eq:liso-def} and~\eqref{eq:riso-def}. The map $\liso$ gives an isomorphism from $\bigl(M_1\afus M_2\bigr) \afus M_3$ to $\ATL{1+2+3}\otimes_{\ATL1\otimes\ATL2\otimes\ATL3}  (M_1\otimes M_2\otimes M_3)$ where $\ATL1\otimes\ATL2\otimes\ATL3\equiv \ATL{N_1}\otimes\ATL{N_2}\otimes\ATL{N_3}$ is considered as the subalgebra in $\ATL{N_1+N_2+N_3}$ under the composition $\varepsilon_{N_1+N_2,N_3}\circ(\varepsilon_{N_1,N_2}\tensor\id)$, recall the definition~\eqref{aTL-embed}.
Similarly the map $\riso$ gives  an isomorphism from the other bracketing $M_1\afus \bigl(M_2 \afus M_3\bigr)$ to $\ATL{1+2+3}\otimes_{\ATL1\otimes\ATL2\otimes\ATL3}  (M_1\otimes M_2\otimes M_3)$ where now $\ATL{N_1}\otimes\ATL{N_2}\otimes\ATL{N_3}$ is considered as the subalgebra in $\ATL{N_1+N_2+N_3}$ under a different composition $\varepsilon_{N_1,N_2+N_3}\circ(\id\tensor\varepsilon_{N_2,N_3})$. We finally note that both the compositions have identical images (this is trivial in the finite TL case and a non-trivial but simple check for the affine algebras). Therefore we can define the composition $\riso^{-1}\circ \liso$ and it gives the associativity isomorphism $\aass_{M_1,M_2,M_3}$.
The pentagon identity is proven along the same lines as in Prop.~\ref{prop:ass-pentagon}.
\end{proof}

We have thus shown the following.
\begin{prop}
Let $\afus$  denote the $\oN$-graded bilinear tensor product on $\catATL$ as defined for each pair $(N_1,N_2)\in\oN\times\oN$ in~\eqref{fusfunc-ATL-def-0} and~\eqref{fusfunc-def}. It is equipped with the associator $\alpha$ from Prop.~\ref{prop:ass-ATL} and it satisfies  the pentagon identity. The category $\catATL$ is thus an  $\oN$-graded semi-group category.
\end{prop}

\subsection{The direct-limit category $\catATLinf$}
Inside the enveloping category $\catATL$, we consider two direct systems (recall the definition~\eqref{eq:dir-lim} with $C_i=\catATL_i$ and $F_{ij}=\afunGl_{j-2}\circ\ldots \afunGl_{i+2}\circ\afunGl_{i}$,  $i\leq j$)
$$
\catATL_1 \xrightarrow{\;\; \afunGl_1 \;\;} \catATL_3  \xrightarrow{\;\; \afunGl_3 \;\;} \ldots
$$
and 
$$
\catATL_2 \xrightarrow{\;\; \afunGl_2 \;\;} \catATL_4  \xrightarrow{\;\; \afunGl_4 \;\;} \ldots\ .
$$
We denote the corresponding direct limits as 
\begin{equation}\label{catATLinf-0}
\catATLodd = \varinjlim \catATL_{\mathrm{odd}}
\qquad
\text{ and } 
\qquad
 \catATLev = \varinjlim \catATL_{\mathrm{even}}\ .
 \end{equation}
Then, we define the category
\begin{equation}\label{catATLinf}
\catATLinf = \catATLev \oplus \catATLodd.
\end{equation}
Note that by the construction the category $\catATLinf$ is an abelian $\oC$-linear category for any non-zero value of $\q$ (except $\q=\pm\rmi$ where our construction is not defined).

\subsection{Objects in $\catATLinf$}
\label{sec:st-proj-hatCinf}

Prop.~\ref{prop:afunGl-W} suggests the following numeration of standard objects in the direct limit $\catATLinf$:
\begin{equation}\label{catATLinf-Wj}
\StATL{j,z} \; \equiv \; \{\, \StATL{j,z}[2j],\, \StATL{j,z}[2j+2],\, \ldots,\, \StATL{j,z}[N], \, \ldots\, \}\ , \qquad j\in\half\oN, \; z\in\oC^\times\ ,
\end{equation}
where the set in the figure brackets is the equivalence class of the objects in $\catATL$ mapped by the globalisation functors $\funGl_N$ for different values of $N$. 
In general, for any non-zero $\ATL{N}$-mo\-du\-le $M[N]$ such that its localisation is zero, we define
\begin{equation}\label{catTLinf-M}
M \; \equiv \; \{\, M[N],\, M[N+2],\, \ldots\, \},\quad \text{with}\quad M[N+2]\equiv \afunGl_N (M[N]),\quad \text{etc.,}
\end{equation}
as the corresponding equivalence class in the direct limit $\catATLinf$.

\begin{Lemma}\label{lem:aGl-class}
Let $M_1[N_1]$ and  $M_2[N_2]$ be affine TL modules. 
Then, the affine TL modules $M_1[N_1]\afus M_2[N_2]$ and $M_1[N_1]\afus M_2[N_2+2]$ are in the same equivalence class in $\catATLinf$ or, equivalently, we have
\begin{equation}\label{eq:aGl-class}
\afunGl_{N_1+N_2}:\qquad M_1[N_1]\afus M_2[N_2] \; \mapsto M_1[N_1]\afus M_2[N_2+2]\ . 
\end{equation}
\end{Lemma}

The proof of this lemma  repeats the proof of the analogous Lem.~\ref{lem:Gl-class} where  our manipulations with the balanced tensor products are also valid for infinite-dimensional algebras. One has only to replace $\TL{N}$   by $\ATL{N}$ and the product $a\cdot b$ in the map~\eqref{eq:kappa} stands now for the product of $a$ and the image of $b$ under our affine TL embedding of $\ATL{N}$ into $\ATL{N+2}$, recall Sec.~\ref{sec:aTL-emb}.

Using Lem.~\ref{lem:aGl-class} together with Prop.~\ref{prop:afunGl-W}, we have an immediate application to the calculation of affine TL fusion 
rules:\footnote{To vary  the index $N_1$, we introduce the globalisation functors $\afunGll_N$  corresponding to the idempotents $\el$   as in Rem.~\ref{rem:funGll} and obtain
$\afunGll_{N_1+N_2}: M_1[N_1]\afus M_2[N_2] \; \mapsto  \afunGll_{N_1}(M_1[N_1])\afus M_2[N_2]$.
 And then note that the $\ATL{N_1+2}$ module $\afunGll_{N_1}(M_1[N_1])$ is identical to $\afunGl_{N_1}(M_1[N_1])$.
 }
\begin{multline}
\label{afus-ex1-gen}
\StJTL{\half}{z_1}[N_1]\afus\StJTL{\half}{z_2}[N_2] =
\delta_{z_2,-\q z_1}
\StJTL{1}{ i\q^{\half} z_1}[N_1+N_2] \\
\oplus
\delta_{z_2,z_1^{-1}}
\StJTL{0}{- i\q^{-\half}z_1}[N_1+N_2] \ ,
\end{multline}
for odd $N_1$ and $N_2$, and
\begin{multline}
\label{afus-ex2-gen}
\StJTL{\half}{z_1}[N_1]\afus\StJTL{1}{z_2}[N_2] =
\delta_{z_2,-i\q^{3/2} z_1}
\StJTL{\frac{3}{2}}{-\q z_1}[N_1+N_2] \\
\oplus
\delta_{z_2,i\q^{1/2}z_1^{-1}}
\StJTL{\half}{-\q/ z_1}[N_1+N_2]\ ,
\end{multline}
for odd $N_1$ and even $N_2$, where we also used the result of the calculation on $1+1$ and $1+2$ sites in~\eqref{afus-ex1} and~\eqref{afus-ex2}, respectively.
We thus  see that the fusion rules are stable with the index $N$.

\medskip
 We also  get the following result, similarly to the finite TL case. 
\begin{prop}\label{cor:aM1M2}
The modules $M_1[N_1]\afus M_2[N_2]$ and $M_1[N_1+2n]\afus M_2[N_2+2m]$ over the corresponding  affine TL algebras, for integer $n$ and $m$, are in the same equivalence class in the direct limit $\catATLinf$. 
 If $m=-n$, the two $\ATL{N_1+N_2}$ modules are even identical in $\catATL_{N_1+N_2}$.
Their equivalence class is by definition the fusion $M_1\afuslim M_2$ of two classes $M_1$ and $M_2$ in $\catATLinf$. 
\end{prop}

By this proposition we can introduce the definition of the tensor product in $\catATLinf$.

\begin{Dfn}
We define the tensor product $\afuslim$ in $\catATLinf$ as
\begin{equation}\label{afuslim}
\afuslim:\qquad (M_1,M_2)\; \mapsto\; M_1\afuslim M_2 = \lim\Bigl[M_1[N_1]\afus M_2[N_2]\Bigr] \ ,
\end{equation}
where $\lim[\ldots]$ stands for taking the direct limit of the object (in this case, from $\catATL_{N_1+N_2}$) or the corresponding equivalence class, for any choice of $(N_1,N_2)$ such that $M_1[N_1]$ and $M_2[N_2]$ are non-zero. (This definition does not depend on such a choice because of Prop.~\ref{cor:aM1M2}.)
\end{Dfn}

By the definition of $\afuslim$, note that the fusion obtained in~\eqref{afus-ex1-gen} and~\eqref{afus-ex2-gen} allows us to calculate or decompose the tensor products $\StJTL{\half}{z_1}\afuslim\StJTL{\half}{z_2}$ and $\StJTL{\half}{z_1}\afuslim\StJTL{1}{z_2}$ in $\catATLinf$: we have just to remove the square brackets in the formulas~\eqref{afus-ex1-gen} and~\eqref{afus-ex2-gen} replacing  $\afus$ by $\afuslim$.

\medskip

We next define the associator for the tensor product $\afuslim$ in $\catATLinf$. Recall that to specify a morphism in $\Hom_{\catATLinf}(M_1,M_2)$ it is enough to choose a representative in the space $\Hom_{\catATL_N} (M_1[N],M_2[N])$. 

\begin{Dfn}\label{def:aasslim}
For a triple of objects $M_1$, $M_2$, $M_3$ in $\catATLinf$ we define an isomorphism
\begin{equation}\label{aasslim}
\aasslim_{M_1,M_2,M_3}: \quad \bigl(M_1\afuslim M_2\bigr) \afuslim M_3 \xrightarrow{\quad \cong\quad}
M_1\afuslim \bigl( M_2 \afuslim M_3\bigr) 
\end{equation}
as follows: take any triple of positive integers $(N_1,N_2,N_3)$ such that $M_i[N_i]$ are non-zero, for $i=1,2,3$, and take then the corresponding associator
$\ass_{M_1[N_1],M_2[N_2],M_3[N_3]}$ defined in Prop.~\ref{prop:ass-ATL}, then $\aasslim_{M_1,M_2,M_3}$ is its equivalence class, i.e. the corresponding element in the quotient space as in~\eqref{Hom-lim-def}.
\end{Dfn}

The arguments showing  that   $\aasslim$ is well-defined and does not depend on the choice of representatives in the $\Hom$ spaces are similar to those after Def.~\ref{def:asslim}. The arguments that the family $\aasslim_{M_1,M_2,M_3}$ satisfies the pentagon identity~\eqref{eq:pentagon} are identical to those in the finite TL case.

\medskip
 
For the moment, we were not able to deduce a tensor unit in the limit category $\catATLinf$ of affine TL modules, so we have obtained at least a semi-group category.
We have thus proven the main theorem of this section.

\begin{Thm}\label{thm:catATLinf}
The category $\catATLinf$ defined in~\eqref{catATLinf} with~\eqref{catATLinf-0} is a semi-group category with  the  tensor product $\afuslim$ given in~\eqref{afuslim} with the associator $\aasslim_{A,B,C}$ given in~\eqref{aasslim}
satisfying the pentagon condition.
\end{Thm}

 The question on existence of the tensor unit in this category will be explored in our next paper~\cite{GJS}.

\medskip

It is now time to explore the braiding properties of the affine TL fusion.

\subsection{Non-commutativity of affine TL fusion}
\label{sec:br-ATL}

In contrast to the finite TL fusion, the affine one $\afus$ is trivial in most of the cases, recall the result in~\eqref{afus-ex1} as well as~\eqref{afus-ex1-gen} and~\eqref{afus-ex2-gen}, except certain resonance conditions on the $z$-parameters. It makes the fusion $\afus$ non-commutative. Indeed, we compute the fusion in~\eqref{afus-ex1} in the two orders: 
\begin{equation}\label{afus-ex1-rev}
\StJTL{\half}{z}[1]\afus\StJTL{\half}{-\q z}[1] = \StJTL{1}{ i\q^{\half} z}[2] \qquad \text{while} \qquad 
\StJTL{\half}{-\q z}[1] \afus\StJTL{\half}{z}[1] = 0
\end{equation} 
(if $\q$ is not $\pm1$) and
\begin{equation*}
\StJTL{\half}{z}[1]\afus\StJTL{\half}{z^{-1}}[1] =\StJTL{0}{-i\q^{-\half}z}[2] \quad \text{while} \quad 
\StJTL{\half}{z^{-1}}[1] \afus\StJTL{\half}{z}[1] = \StJTL{0}{-i\q^{-\half}z^{-1}}[2]\ .
\end{equation*}

We thus conclude that in contrast to the finite TL fusion $\fus$, which has the braiding, our affine TL fusion is non-commutative and there exists no braiding. 
However, we can introduce another affine TL fusion  by replacing $g_i$ by its inverse $g_i^{-1}$ in the definition~\eqref{aTL-embed} of the embedding  $\afemb{N_1,N_2}: \ATL{N_1}\otimes\ATL{N_2} \to \ATL{N_1+N_2}$ introduced in Sec.~\ref{sec:aTL-emb}, which diagrammatically corresponds to the interchange between under- and above-crossings. Let us denote such embedding as $\afembm{N_1,N_2}$. The definition~\ref{afus-def} using the embedding $\afembm{N_1,N_2}$ gives then a different tensor product that we denote as  $\afusm$. With this new embedding, the two translation generators correspond now, instead of (\ref{two-transl}),  to 
\begin{equation}\label{braiiden-2}
\tilde{u}^{(1)}=u g_{N-1}\ldots g_{N_1}, \qquad \tilde{u}^{(2)}=g_{N_1}^{-1}\ldots g_1^{-1}u\ .
\end{equation}
In terms of diagrams, we have for instance (with $N_1=3$ and $N_2=2$), instead of (\ref{u1-def}) and (\ref{u2-def}), the following

\begin{equation}\label{utilde1-def}
 \begin{tikzpicture}
 \node[font=\large]  at (-1.5,-2.2) {\mbox{} $\tilde{u}^{(1)}\;\;\mapsto$ \mbox{}\qquad};
\drawu 
\braid[braid colour=black,strands=5,braid width=\brw,braid start={(-0.3,-0.88)}]  {\g_4\g_3}
\node[font=\large]  at (4.0,-2.2) {\mbox{}\qquad \(=\)};
\end{tikzpicture}
\qquad
 \begin{tikzpicture}
 	\draw[thick, dotted] (-0.05,0.5) arc (0:10:0 and -7.5);
 	\draw[thick, dotted] (-0.05,0.55) -- (3.25,0.55);
 	\draw[thick, dotted] (3.25,0.5) arc (0:10:0 and -7.5);
	\draw[thick, dotted] (-0.05,-0.85) -- (3.25,-0.85);

	\draw[thick] (0.3,0.5) arc (0:10:20 and -3.75);
	\draw[thick] (0.9,0.5) arc (0:10:40 and -7.6);
	\draw[thick] (1.5,0.5) arc (0:10:40 and -7.6);
	

	\draw[thick] (1.6,-0.81) .. controls (1.7,-0.3)  .. (2.,-0.25 );
		\draw[thick] (2.2,-0.25) -- (2.6,-0.25);
	\draw[thick] (2.1,0.5) -- (2.1,-0.8);
	\draw[thick] (2.7,0.5) -- (2.7,-0.8);
		\draw[thick] (2.8,-0.25) .. controls (3.13,-0.2)   .. (3.21,-0.1);
\end{tikzpicture}
\end{equation}
and 
\begin{equation}\label{utilde2-def}
 \begin{tikzpicture}
 \node[font=\large]  at (-1.5,-0.2) {\mbox{} $\tilde{u}^{(2)}\;\;\mapsto$ \mbox{}\qquad};
 
\drawu 

\braid[braid colour=black,strands=5,braid width=\brw,braid start={(-0.28,3.6)}]  { \g_3^{-1} \g_2^{-1} \g_1^{-1}}

\node[font=\large]  at (4.0,-0.2) {\mbox{}\qquad \(=\)};
\end{tikzpicture}
\qquad
 \begin{tikzpicture}
 	\draw[thick, dotted] (-0.05,0.5) arc (0:10:0 and -7.5);
 	\draw[thick, dotted] (-0.05,0.55) -- (3.25,0.55);
 	\draw[thick, dotted] (3.25,0.5) arc (0:10:0 and -7.5);
	\draw[thick, dotted] (-0.05,-0.85) -- (3.25,-0.85);
	\draw[thick] (0.0,-0.18) .. controls (0.15,0.08)   .. (0.58,0.1);
	\draw[thick] (0.58,0.1) -- (1.02,0.1);
	\draw[thick] (1.02,0.1) -- (1.77,0.1);
	\draw[thick] (1.77,0.1) .. controls (2.1,0.2)   .. (2.3,0.5);

	\draw[thick] (2.9,-0.81) arc (0:10:-20 and 3.75);

	\draw[thick] (0.5,0.5) -- (0.5,0.17);
	\draw[thick] (0.5,0.) -- (0.5,-0.8);
	\draw[thick] (1.1,0.5) --  (1.1,0.2);
	\draw[thick] (1.1,0.) --  (1.1,-0.8);
	\draw[thick] (1.7,0.5) -- (1.7,0.2);
	\draw[thick] (1.7,0.) -- (1.7,-0.8);

	\draw[thick] (2.75,0.5) arc (0:10:40 and -7.6);

\end{tikzpicture}
\end{equation}

\medskip
It is interesting that there is a  braiding-type operation that relates the two affine TL fusions $\afus$ and $\afusm$. Indeed, recall that in Sec.~\ref{sec:braiding-TL} we have introduced the braiding $\br_{M_1,M_2}$ for the  TL fusion given by conjugation~\eqref{br-TL}-\eqref{br-TL-1} with the ``braid-like'' element $\g_{N_1,N_2}$.
It is easy to see graphically -- or by direct calculation using repeatedly that $ug_i=g_{i+1}u$ -- that the following identities hold:
\begin{eqnarray}
\g_{N_1,N_2}~ u^{(1)}_{N_1,N_2}=\tilde{u}^{(2)}_{N_2,N_1}~ \g_{N_1,N_2}\ ,\nonumber\\
\g_{N_1,N_2} ~u^{(2)}_{N_1,N_2}=\tilde{u}^{(1)}_{N_2,N_1}~\g_{N_1,N_2}\ ,
\label{braiiden}
\end{eqnarray}
where we temporarily used the notation  $u^{(1,2)}_{N_j,N_k}$ and $\tilde{u}^{(1,2)}_{N_j,N_k}$ for images of the translation generators $u^{(1,2)}$
under the homomorphisms $\varepsilon_{N_j,N_k}$ and  $\varepsilon^-_{N_j,N_k}$, respectively. 

Similarly to the finite TL case, we note that the element $\g_{N_1, N_2}$ defines an automorphism on $\ATL{N_1+N_2}$ by the conjugation $a\mapsto \g_{N_1, N_2}\cdot a \cdot \g_{N_1, N_2}^{-1}$ which maps the subalgebra $\afemb{N_1,N_2}\bigl(\ATL{N_1}\otimes\ATL{N_2}\bigr)$ (i.e., under the first type of the affine TL embedding) to the subalgebra $\afembm{N_2,N_1}\bigl(\ATL{N_2}\otimes\ATL{N_1}\bigr)$ (i.e., under the second embedding) as 
\begin{equation}\label{g-ATL-ATL}
\afemb{}(a\otimes b) \; \mapsto \;  \g_{N_1, N_2} \cdot \afemb{}(a\otimes b)\cdot \g_{N_1, N_2}^{-1} = \afembm{}(b \otimes a)\ ,\qquad a\in\ATL{N_1}\ ,\;\; b\in\ATL{N_2}\ .
\end{equation}

\newcommand{\abr}{\widehat{\br}}

Then, we can introduce a braiding-type  relation between $\afus$ and $\afusm$ given by the isomorphism 
\begin{equation}\label{br-ATL}
\abr_{M_1,M_2}: \quad  M_1[N_1]\afus M_2[N_2] \; \xrightarrow{\quad \cong \quad} \; M_2[N_2] \afusm M_1[N_1] \quad
\end{equation}
with 
\begin{equation}\label{br-ATL-1}
\abr_{M_1,M_2}: \quad 
a\otimes m_1\otimes m_2 \; \mapsto \; \g^{}_{N_1,N_2} \cdot a \cdot \g_{N_1,N_2}^{-1} \otimes m_2\otimes m_1\ ,
\end{equation}
where $a\in\ATL{N_1+N_2}$,  and $m_1\in M_1[N_1]$,  $m_2\in M_2[N_2]$. 
Recall that we write $a\otimes m_1\otimes m_2$ here for a representative in the corresponding class in $M_1[N_1]\afus M_2[N_2]$.
The only non-trivial thing to  check is that the map~\eqref{br-ATL-1} is well-defined, i.e., does not depend on a representative in the class. Indeed, assume that $m_1=b \cdot m'_1$ and $m_2 = c\cdot m'_2$ for some $b\in\ATL{N_1}$, $c\in\ATL{N_2}$ and  some $m'_i\in M_i[N_i]$, and let us compute~\eqref{br-ATL-1} for the other representative (setting here $\g\equiv \g_{N_1,N_2}$, $\afemb{}\equiv \afemb{N_1,N_2}$  and  $\afembm{}\equiv \afembm{N_2,N_1}$ for brevity):
\begin{multline}\label{br-ATL-2}
a\cdot \afemb{}(b\otimes c)\otimes m'_1\otimes m'_2 \; \mapsto \; \g \cdot a \cdot  \afemb{}(b\otimes c) \cdot \g^{-1} \otimes m'_2\otimes m'_1 
= \g \cdot a \cdot g^{-1} \cdot \bigl(\g \cdot  \afemb{}(b\otimes c) \cdot g^{-1}\bigr) \otimes m'_2\otimes m'_1 \\
=   \g \cdot a \cdot g^{-1}\cdot  \afembm{}(c\otimes b) \otimes m'_2\otimes m'_1
= \g \cdot a \cdot \g^{-1} \otimes m_2\otimes m_1  \ ,
\end{multline}
where we used~\eqref{g-ATL-ATL} for the second equality, and note that the $\otimes$ in front of $m'_2$ on the right-hand side from `$\mapsto$' is over the subalgebra $\afembm{N_2,N_1}\bigl(\ATL{N_2}\otimes\ATL{N_1}\bigr)$, as assumed in~\eqref{br-ATL}.
The final result in~\eqref{br-ATL-2} thus agrees with~\eqref{br-ATL-1}.
Note that if we would use the same tensor product ($\afus$ or $\afusm$) in~\eqref{br-ATL}-\eqref{br-ATL-1}, the map would not be well-defined.

The rest of the proof of the isomorphism property repeats the finite TL case discussed in Sec.~\ref{sec:braiding-TL}:
recall   that  an element $a'\in\ATL{N_1+N_2}$   acts on the left-hand side of~\eqref{br-ATL-1}  by the multiplication with $a'$ while  on the right-hand side of~\eqref{br-ATL-1} it acts  by the multiplication with $\g^{}_{N_1,N_2}\cdot a' \cdot \g_{N_1,N_2}^{-1}$.
The intertwining property of the map $\abr_{M_1,M_2}$ in~\eqref{br-ATL-1}, i.e., that it commutes with the two $\ATL{N_1+N_2}$ actions, is then straightforward to check. This map is obviously bijective. We have thus proven that $\abr_{M_1,M_2}$ is an isomorphism in the category~$\catATL$.

\medskip

We call the isomorphisms  $\abr_{M_1,M_2}$ as \textit{semi-braiding} associated with the two tensor products $\afus$ and $\afusm$.
We finally note that the family of isomorphisms $\abr_{M_1,M_2}$ defined in~\eqref{br-ATL}-\eqref{br-ATL-1} satisfies an analogue of the coherence (hexagon) conditions (required for the ordinary braiding in a tensor category) but involving the two tensor products $\afus$ and $\afusm$.
More properties of the relation between $\afus$ and $\afusm$ will be explored in our forthcoming paper~\cite{GJS}.

We finally note that the semi-braiding $\abr_{M_1,M_2}$ is lifted to the corresponding family of isomorphisms $\abr_{M_1,M_2}^{\;\catATLinf}$ in the direct-limit category $\catATLinf$, similarly to what we have in Def.~\ref{def:br-lim}. This extends Thm.~\ref{thm:catATLinf} by the semi-braiding structure on $\catATLinf$ with respect to the two tensor products -- the \textit{chiral} $\afuslim$ and the \textit{anti}-chiral $\afuslim^{-}$.

\newcommand{\Vir}{V}
\newcommand{\catVir}{\mathsf{Vir}}
\newcommand{\fusV}{\otimes_{\catVir}}

\section{Outlook: a relation to Virasoro algebra}\label{sec:outlook}
Let $\Vir_p$ be the Virasoro algebra of central charge 
\begin{equation}
c(p)=1-{6\over p(p-1)}\label{centch}, \qquad p \in (1,\infty]
\end{equation}
i.e., a Lie algebra generated by $L_n$, with $n\in\oZ$, and the central element $c$ with brackets
\begin{equation}
[L_n,L_m]=(n-m)L_{n+m}+{c\over 12} (n^3-n)\delta_{n+m,0}\ .
\end{equation}

Verma modules  $\Verma_{h}$ are highest-weight representations of $\Vir_p$ and completely characterized by the central charge~$c$ and the eigenvalue~$h$  (so called ``conformal weight'') of $L_0$ on their highest weight vector. These modules $\Verma_h$ admit singular vectors iff their conformal weights are of the form
\begin{equation}
h_{r,s}={[pr-(p-1)s]^2-1\over 4p(p-1)}\ ,
\end{equation}
with $r,s\in \oZ_+$.  The first singular vector appears at level $rs$, that is, with $L_0$ eigenvalue given by $h_{r,s}+rs=h_{r,-s}$. For generic central charge $c$, these modules admit a unique singular vector. The corresponding Kac module  is then defined as the quotient of the  Verma module with $h_{r,s}$ by its submodule of the weight $h_{r,s}+rs$ (which is also a Verma module):  $\VK_{r,s} \equiv \Verma_{h_{r,s}}/\Verma_{h_{r,-s}}$. This module is irreducible when the central charge is generic. More generally, the Kac module $\VK_{rs}$ is defined as the submodule of  the Feigin-Fuchs module~\cite{FF} $\VF_{rs}$ generated by the subsingular vectors of grade strictly less than~$rs$.

We define then a certain abelian $\oC$-linear category of  $\Vir_p$ representations which is a highest-weight category with the standard objects given by the Kac modules $\VK_{1,n}$, with $n\in\oN$, and where indecomposable projective objects admit non-diagonalizable action of $L_0$ and are filtered by the Kac modules and the length of the filtration is at most  two. Here, the projective covers (those which are reducible) are  the so-called \textit{staggered} modules~\cite{KytolaRidout, Rohsiepe}, which means an extension of two highest-weight modules such that the action of $L_0$ is non-diagonalizable. In our case, their subquotient structure has a diamond shape, for $j\modd p\ne \frac{kp-1}{2}$ with
$k=0,1$,
\begin{align}\label{stagg-pic-gen-dense-even}     
   \xymatrix@C=5pt@R=25pt@M=3pt{%
    &&\\
    &\VP_{1,2j+1}: &\\
    &&
 }      
&  \xymatrix@C=4pt@R=25pt@M=3pt@W=3pt{%
    &&{\stackrel{h_{1,2j+1}}{\bullet}}\ar[dl]\ar[dr]&\\
    &{\stackrel{h_{1,1+2(j-s)}}{\bullet}}\ar[dr]&&\stackrel{h_{1,1+2(j+p-s)}}{\bullet}\ar[dl]\\
    &&{\stackrel{h_{1,2j+1}}{\bullet}}&
 } \quad
&   \xymatrix@C=5pt@R=25pt@M=3pt{%
    &&\\
    &\text{for}\quad  j \geq \ffrac{p}{2},&\\
    &&
 }&      
\end{align}
where we set $s\equiv s(j)=(2j+1)\;\mathrm{mod}\; p$, 
and the nodes `$\bullet$' together with conformal weights $h_{1,j}$ denote irreducible $\Vir_p$ subquotients.
We note that the modules with this subquotient structure and the requirement that the action of $L_0$ is non-diagonalisable on them are unique up to an isomorphism for the central charges $c(p)$ and integer  $p\geq3$, see~\cite{KytolaRidout}.
For $j<p/2$, the projective covers $\VP_{1,2j+1}$ are $\VK_{1,2j+1}$.
We denote such an abelian category (generated by the staggered modules $\VP_{1,2j+1}$) as~$\catVir_p$. 

By a direct comparison of the projective objects in $\catVir_p$ and in $\catTLinf$ for $\q=e^{\rmi\pi/p}$ described in Sec.~\ref{sec:st-proj-Cinf} we establish an isomorphism between the $\Hom$ spaces in both the categories and eventually the following equivalence.

\begin{Prop}\label{prop:equiv-1}
There exists an equivalence of  abelian $\oC$-linear categories  $\catTLinf$ for $\q=e^{\rmi\pi/p}$ from Sec.~\ref{sec:st-proj-Cinf} and $\catVir_p$   such that the simple objects $\IrTL{j}$ are identified with irreducible representations of  conformal weight $h_{1,2j+1}$, the standard objects $\StTL{j}$ are identified with the Kac modules $\VK_{1,1+2j}$, and the projective covers $\PrTL{j}$ with the staggered modules $\VP_{1,2j+1}$.
\end{Prop}

This equivalence has an interesting interpretation from the physics point of view, as already mentioned in the introduction.

When $|m|\leq 2$, statistical mechanics models whose Boltzmann weights are built using representations of the Temperley--Lieb algebra are heuristically known to be critical, and have their continuum limit ``described'' by conformal field theories, with central charge~\eqref{centch} if we parametrize 
$m=\q+\q^{-1}$ with $\q=e^{\rmi\pi/p}$ and $p\in(1,\infty]$.
This statement can be made more precise as follows. The spectrum of eigenvalues of the Hamiltonian $H=-\sum_{i=1}^{N-1} e_i$ (which describes a statistical system with ``open boundary conditions'')   in modules over the TL algebra has been found to coincide, in the limit $N\to\infty$ and after proper rescaling, with the spectrum of the generator $L_0-{c\over 24}$ (note, there is a single Virasoro algebra here because we deal with open boundary conditions, corresponding to boundary conformal field theory) in some corresponding Virasoro modules for the central charge (\ref{centch}). When $\q$ is not a root of unity, one can choose without loss of generality to study the TL modules $\StTL{j}$. The corresponding Virasoro module is then found to be the Kac module  with lowest weight given by the conformal weight 
\begin{equation}
h_{1,1+2j}={(1-2j(p-1))^2-1\over 4p(p-1)}\ .
\end{equation}
Using the definition of the Kac module as the quotient $\VK_{1,1+2j} = \Verma_{h_{1,1+2j}}/\Verma_{h_{1,-1-2j}}$, the character of this module is  
\begin{equation}
\hbox{Tr\,}_{\VK_{1,1+2j}} q^{L_0-c/24}=q^{-c/24}{q^{h_{1,1+2j}}-q^{h_{1,-1-2j}}\over P(q)}
\end{equation}
where $P(q)=\prod_{n=1}^\infty (1-q^n)$, 
and $q$ here is a formal parameter. In the case when $\q$ is a root of unity -- that is  integer $p$ --  the representation theory of both the Temperley-Lieb algebra and the Virasoro algebra become more complicated. The correspondence between these two algebras however continues to hold, after we fix an appropriate category for $\Vir_p$ representations of course. In particular, the projective modules  $\PrTL{j}$ described in Sec.~\ref{sec:st-proj-Cinf}, see  the ``diamond shape''  diagram in~\eqref{prTL-pic-inf}, can now be put in correspondence with staggered modules $\VP_{1,2j+1}$, which are composed of two Kac modules, i.e., have a subquotient structure with a diamond shape again~\eqref{stagg-pic-gen-dense-even}.

\medskip

Further, we believe that the category $\catVir_p$ has the tensor product structure  (or fusion) that we denote by $\fusV$, and that it is endowed with the tensor unit $\VK_{1,1}$ (this part is clear from the Vertex-Operator Algebra realization of $\Vir_p$), with the associator and the braiding for any  $p$ from~\eqref{centch}: this is based on the rigorous VOA theory of tensor products in~\cite{[HLZ]}.
There are then indications from physics (based on computation of the fusion rules) that our result in Prop.~\ref{prop:equiv-1} extends further to the level of monoidal or tensor categories. We make the following conjecture:

\begin{Conj}\label{conj:equiv-2}
Let $\catTLinf$ be the direct limit of TL module categories at any $\q$ such that $\q=e^{i\pi/p}$, which is the braided tensor category from Thm.~\ref{thm:catTLinf}, and $\catVir_p$ be the braided tensor category of the Virasoro algebra representations at central charge $c(p)$. We have then an equivalence of the braided tensor categories 
\begin{equation}
\catTLinf \; \xrightarrow{\quad \sim \quad} \; \catVir_p
\end{equation}
such that it reduces to an equivalence from Prop.~\ref{prop:equiv-1} if the categories are considered just as abelian $\oC$-linear categories.
\end{Conj}

The indications are the following.
The category $\catVir_p$ or the representation theory of $\Vir_p$ within the category $\catVir_p$ is believed to serve as the fundamental description of a class of conformal field theories (CFTs).
A fundamental question about CFTs is to determine the operator product expansions (OPEs) of their quantum fields. Thanks to the conformal symmetry, the OPEs are essentially determined by the fusion rules for the corresponding modules over the Virasoro algebra.
Going back to our conjecture for generic $\q$, it is well known  indeed that the OPEs of primary fields of conformal field theory associated with the Kac modules obey the following Virasoro fusion rules~\cite{BPZ, FZ}
\begin{equation}\label{eqFusionKacGeneric}
\displaystyle \VK_{1,1+2j_1} \fusV \VK_{1,1+2j_2} = \bigoplus_{j=|j_1-j_2|}^{j_1+j_2} \VK_{1,1+2j},
\end{equation}
which corresponds exactly to the result of fusion in the Temperley-Lieb case,
 or strictly speaking, to the decomposition for $\fuslim$ in our direct-limit category $\catTLinf$:
\begin{equation}
\StTL{j_1}\fuslim \StTL{j_2}=\bigoplus_{|j_1-j_2|}^{j_1+j_2}\StTL{j}\label{resfusopen}\ .
\end{equation}
Moreover, in the case $\q$ a root of unity,  fusion of the Temperley-Lieb modules~\cite{ReadSaleur07-2,GV,Belletete} can again be compared with fusion in the corresponding (logarithmic) conformal field theory based on calculations of logarithmic OPEs in~\cite{GurarieLudwig1} and in~\cite[Sec.~5]{GV} (see also more references therein)
and in the works~\cite{Flohr,MathieuRidout,MathieuRidout1,GRW09} that  use the so-called
Nahm--Gaberdiel--Kausch algorithm~\cite{Nahm,KauschGaberdiel}.
Having the identification from Prop.~\ref{prop:equiv-1} between the modules from both sides, from $\catTLinf$ and $\catVir_p$, we have as well an identification of the fusion rules (or multiplicities of the modules in tensor products of two indecomposables) for all the cases explored so far on the Virasoro side. This agreement motivates  Conj.~\ref{conj:equiv-2}, to which we hope to get back in subsequent work.

\subsection{Non-chiral case and affine TL category $\catATLinf$}

The physics of critical statistical lattice models away from their boundaries (the so called ``bulk'' case)  is described by two copies of the Virasoro algebra, corresponding to the chiral and anti-chiral dependencies of the correlation functions. While in the case of rational conformal field theories, most properties in the bulk  can be inferred from those near the boundary~\cite{BPPZ}, no such relationship is known to exist in general (see \cite{RGW} for a discussion) in the case of logarithmic conformal field theories. Meanwhile, lattice models away from their boundaries are obtained  by choosing periodic boundary conditions, which corresponds to considering now the affine  instead of the finite Temperley-Lieb algebra. 

In a series of works on some simple cases, we have begun to explore \cite{GRS1,GRS2,GRS3,GRSV1} the relationship between modules of $\ATL{N}(m)$ and modules of products of two Virasoro algebras with central charge (\ref{centch}). This had led to a deeper understanding of the subquotient structures appearing in logarithmic CFT for values $\q=i,\q=e^{i\pi/3}$,  (corresponding to central charges $c=-2$ and $c=0$), and to the introduction of the promising concept of interchiral algebra.

In order to go further in our understanding of LCFTs by using lattice models, it is necessary to 
understand  fusion of non-chiral fields, i.e., the fusion of modules over the product of two Virasoro algebras.  Since in the chiral case we observed a full correspondence between fusion in TL and  Virasoro fusion rules, it is natural to expect we can learn something about fusion of non-chiral fields by studying  fusion of modules of the affine Temperley-Lieb algebra.

 It is not so simple  to make progress in this direction however. First, 
we saw that doing a direct calculation of fusion in the periodic case is technically much harder than in the open because of complicated relations between different words in $g_i$'s and $e_j$'s.
We thus need to find another and more constructive way to compute the fusion in periodic systems: this will be  studied in our next paper~\cite{GJS}. Second, note that  the fusion we have defined in the periodic case does not ``reduce'' to fusion in the open case once we restrict to  the finite Temperley-Lieb (sub)algebra. Indeed, from 
the $\TL{N}$ module
\begin{equation}
\StJTL{j}{z}[N] = \bigoplus_{k=j}^{N/2} \StTL{k}[N]
\end{equation}
we see the TL fusion of the right hand side of this equation  using $\fus$ and thus (\ref{resfusopen}) gives a direct sum over many TL modules, while the fusion of the left hand side using $\afus$ gives in most cases a trivial result. Moreover, it is easy to check that even when this fusion is non trivial, the result does not decompose over $\TL{N}$ according to the tensor product $\fus$. This is a priori different from what one would expect in conformal field theory, where fusion of non-chiral fields is expected to decompose in some simple way in terms of the fusion of the chiral components. %
 It would be also interesting to understand better our semi-braiding in $\catATLinf$ and its relation to non-chiral conformal field theory. We will  discuss what happens  in our next paper~\cite{GJS}. 

\bigskip

Meanwhile, we note that it would be  interesting to find a duality with a quantum algebra. Recall that in the finite TL case, there is a well known  duality with (a finite-dimensional quotient of) the quantum algebra $U_{\q} sl(2)$:  
this duality was actually used in~\cite{ReadSaleur07-2} (for $p=2,3$ and for projective objects) and then in~\cite{GV} to compute a large and almost exhaustive list of fusion rules for a pair of indecomposable TL modules at any root of unity. One could then expect that  in the periodic case there is a duality but now with (a quotient of) the affine quantum algebra $U_{\q} \widehat{sl}(2)$.
The idea would be then to use this duality and compute the affine TL fusion  using the coproduct in the affine quantum group \cite{ChariPressley}. 
We leave this interesting problem for a future work.

\appendix

\section{Affine TL fusion: examples}\label{app:fusion-Hecke} \label{app:examples}
In this section we consider the affine TL fusion from the perspective of the affine Hecke algebra calculations. 
More precisely, we compute the right-hand side of~\eqref{H-T-fusion} and
the results of these calculations agree with the diagrammatical calculation of the fusion in the main text of the paper
in Sec.~\ref{sec:ex-fus}. This also supports our Conj.~\ref{conj:H-T}.

\subsection{$j_1=j_2=1/2$}
We consider again the  fusion of  $\StJTL{\half}{z_1}[1]$ and $\StJTL{\half}{z_2}[1]$ 
discussed in Sec.~\ref{sec:ex-fus-1}
but now from the point of view of the affine Hecke algebra, i.e. we are going to analyze    $\StJTL{\half}{z_1}[1]\ahfus\StJTL{\half}{z_2}[1]$ where $\ahfus$ is introduced in~\eqref{aH-fusfunc-def}. With the normalizations adopted in this paper, we use $g_i,x_i$ instead of $\sigma_i,y_i$ and modify the relations (\ref{theyrelations}) into 
\begin{equation}\label{gxg}
g_ix_ig_i=x_{i+1}
\end{equation}
together with 
\begin{equation}
u=x_1g_1\ldots g_{N-1}
\end{equation}
instead of (\ref{identif}). 
Recall then~\eqref{vv-act}:
\begin{equation*}
 u^{(1)} \vv = z_1 \vv,\qquad u^{(2)} \vv = z_2 \vv\ ,
 \end{equation*}
we therefore have
\begin{equation}
 x_1 \vv = z_1 \vv,\qquad x_2 \vv = z_2 \vv\ .
 \end{equation}
We introduce $\ww=g_1\vv$ and note that the two vectors $\vv$ and $\ww$ form a basis in the induced module $\StJTL{\half}{z_1}[1]\ahfus\StJTL{\half}{z_2}[1]$. This is easy to see using the relations  
$x_i g_i = x_{i+1}g_i^{-1}$, recall~\eqref{gxg}.
We use then these affine Hecke relations to obtain the matrix representation of the generators in the space  $\oC\vv\oplus \oC\ww$:
\begin{equation}
x_1=\left(\begin{array}{cc}
z_1&-i\q^{-1/2}z_2(1-\q^2)\\
0&-\q z_2\end{array}\right),~~x_2=\left(\begin{array}{cc}
z_2&i\q^{-3/2}(\q^2-1)z_2\\
0&-\q^{-1}z_1\end{array}\right)\ .
\end{equation}
We have meanwhile
\begin{equation}
g_1=\left(\begin{array}{cc}
0&-\q^{-1}\\
1&-i\q^{1/2}(\q^{-2}-1)
\end{array}\right)
\end{equation}
and similarly for $g_2$ with a bit more complicated matrix. We check then that  $e_1$ and $e_2$ defined via $e_i=\q+\rmi\q^{1/2}g_i$ do satisfy the Temperley--Lieb relations. Moreover,  $u=x_1g_1$ obeys $u^2=z_1z_2 \one$ (recall that $u^N$ is central). Finally, we find  
\begin{eqnarray}
e_1e_2e_1=(z+z^{-1})^2 e_1\nonumber\\
e_2e_1e_2=(z+z^{-1})^2 e_2
\end{eqnarray}
with $z=\pm i\q^{-1/2}\sqrt{z_1z_2^{-1}}$. For  $N=2$, the second relation in (\ref{TLpdef-u2}) becomes simply $u^2 e_1=e_1$ and combining with $u^2=z_1z_2 \one$ this gives the condition $z_1 z_2=1$. So, the result of the fusion is therefore zero unless    $z_2=z_1^{-1}$, and thus  $z=- i\q^{-1/2}z_1$
(with the same sign convention as discussed in the main text in Sec.~\ref{sec:ex-fus-1}). 
Hence, we have found 
\begin{equation}\label{eq:app-ex-1}
\StJTL{\half}{z_1}[1]\afus\StJTL{\half}{z_2}[1] =
\StJTL{0}{- i\q^{-\half}z_1}[2]\ \qquad \text{when}\; z_2=z_1^{-1}.\\
\end{equation}

The result~\eqref{eq:app-ex-1} is  obtained assuming $\ww$ is linearly independent of $\vv$. Otherwise, one gets $g_1\vv=\lambda \vv$ where $\lambda$ is a constant easily determined -- using the relation between $g^2_i$ and $g_i$  -- to be $\lambda=i\q^{1/2}$. It follows that 
\begin{equation}\label{eq:app-ex-1-1}
\StJTL{\half}{z_1}[1]\afus\StJTL{\half}{z_2} [1]=
\StJTL{1}{ i\q^{\half}z_1}[2]\ \qquad \text{when}\; z_2=-\q z_1.\\
\end{equation}
In other words, the induced affine Hecke module admits an invariant subspace at $z_2=-\q z_1$ and the quotient by this submodule allows the action of $\ATL{2}(\q+\q^{-1})$. This result can be interpreted as follows: 
for generic values of $z_1$ and $z_2$ the induced module $\StJTL{\half}{z_1}[1]\ahfus\StJTL{\half}{z_2}[1]$ does not admit the action of $\ATL{2}$ because the ideal $\I$ from~\eqref{TH} generates the whole module, so the right hand side of~\eqref{H-T-fusion} is zero, while for $z_2=z_1^{-1}$ the ideal $\I$ acts as zero and~\eqref{H-T-fusion} gives~\eqref{eq:app-ex-1}, and for $z_2=-\q z_1$ it generates a one-dimensional invariant subspace and~\eqref{H-T-fusion} gives~\eqref{eq:app-ex-1-1}, and these are all possible cases.

Combining~\eqref{eq:app-ex-1} with~\eqref{eq:app-ex-1-1}, we see that our affine Hecke calculation of the affine TL fusion (under the result in~\eqref{H-T-fusion} and Conj.~\ref{conj:H-T}) is
in agreement with the previous diagrammatical calculation in the main text resulted in~\eqref{afus-ex1}.

\subsection{$j_1=0$ and $j_2=1/2$}

The definition of  fusion $\afus$ holds for all modules of course, not just the standard ones. As an example, we consider here the case of ${\StJTL{0}{\q}}[N]$, which is well known \cite{GL} to be reducible with a submodule isomorphic to $\StJTL{1}{1}[N]$, and admits a simple (for generic $\q$) quotient ${\StJTL{0}{\q}}/{\StJTL{1}{1}}\equiv \bAStTL{0}{\q}[N]$ of dimension $ \hat{d}_{0}[N] - \hat{d}_{1}[N]$, see~\eqref{dim-dj}. 

Restricting  to the simplest case $N=2$,  the module $\bAStTL{0}{\q}[2]$   is one-dimensional and has the basis vector $\uu$ with the action of the affine Hecke generators: $g_1\uu=-iq^{-3/2}\uu$, and $u^{(1)}\uu=\uu$, and $x_1\uu=-iq^{-3/2}\uu$. We consider then the fusion $\ahfus$ with a one dimensional module $\StJTL{\half}{z}[1]$. Within the induced module, the vectors $g_2\uu\equiv \vv$ and $g_1\vv\equiv \ww$ are linearly  independent and form a basis with $\uu$. We then find in the $\uu,\vv,\ww$ basis the matrix representation: 
\begin{equation}
g_1=\left(\begin{array}{ccc}
i\q^{-3/2}&0&0\\
0&0&-\q^{-1}\\
0&1&i\q^{1/2}(1-\q^{-2})
\end{array}\right),\qquad g_2=\left(\begin{array}{ccc}
0&-\q^{-1}&0\\
1&iq^{1/2}(1-\q^{-2})&0\\
0&0&-i\q^{-3/2}
\end{array}\right)\ .
\end{equation}
It is easy to check however that the  generators $e_i\equiv \q+\rmi\q^{1/2}g_i$ now do not satisfy the required relation $e_ie_{i+1}e_i=e_i$ for $i=1,2$. It is necessary to take a quotient implying the linear relation $\uu=-i \q^{-1/2}\vv+\q^{-1}\ww$. In the $\vv,\ww$ basis for instance, one finds now
\begin{equation}
g_1=\left(\begin{array}{cc}
0&-\q^{-1}\\
1& i\q^{1/2} (1-\q^{-2})
\end{array}\right)
\end{equation}
and 
\begin{equation}
g_2=\left(\begin{array}{cc}
i\q^{1/2}&0\\
-\q^{-2}&-i\q^{-3/2}\end{array}\right)
\end{equation}
together with 
\begin{equation}
x_1=\left(\begin{array}{cc}
i\q^{3/2}&0\\0&z\end{array}\right),\qquad x_2=\left(\begin{array}{cc}
-z\q^{-1}&iz\q^{-5/2}(\q^2-1)\\
iz\q^{-3/2}(\q^2-1)&-i\q^{1/2}-z\q^{-3}(\q^2-1)^2\end{array}\right)
\end{equation}
and a slightly more complicated expression for $x_3$. The point however is that, in this restricted space,  $[x_i,x_j]=0$ if and only if $z=i\q^{3/2}$. 
 We find then
\begin{equation}
u=\left(\begin{array}{cc}
i\q^{-3/2}&-\q^{-1}\\
1-\q^2-\q^{-2}&i(\q^{1/2}-\q^{-3/2})
\end{array}\right)
\end{equation}
and check that, in the quotient, we have
\begin{equation}
u^2 e_2=e_1e_2
\end{equation}
while $u^3=i\q^{3/2}\one$.  In other words, the ideal  $\I$ from~\eqref{TH} does not generate the whole module but only a proper invariant subspace only at $z=i\q^{3/2}$.
The resulting quotient-module being two dimensional is a quotient of the standard  (three dimensional) $\ATL{3}$-module ${\StJTL{\half}{z}}[3]$ when $z=i\q^{3/2}$: it is known that ${\StJTL{\half}{i\q^{3/2}}}$ is reducible and admits a two dimensional irreducible quotient-module $\bAStTL{1/2}{i\q^{3/2}}[3]$, and it is the only two-dimensional irreducible module for $\ATL{3}$ at generic $\q$. Therefore, we obtain the affine TL fusion
\begin{equation}\label{eq:app-ex-2}
\bAStTL{0}{\q}[2]\afus\StJTL{\half}{z}[1] =
\begin{cases}
\bAStTL{1/2}{z}[3] \qquad &z=i\q^{3/2}\ ,\\
0\qquad & \text{otherwise}\ .
\end{cases} 
\end{equation}
again using the result in~\eqref{H-T-fusion} and Conj.~\ref{conj:H-T}. We finally note that the result~\eqref{eq:app-ex-2} will be also confirmed in our next paper~\cite{GJS}.

\subsection{$j_1=1/2$ and $j_2=1$}\label{app:fusion-example-3}
Finally we consider the fusion $\StJTL{\half}{z_1}[1]\afus\StJTL{1}{z_2}[2]$ with the first module on one site the second on two sites. Starting with the only basis element $\uu$ in the ordinary tensor product  $\StJTL{\half}{z_1}[1]\otimes\StJTL{1}{z_2}[2]$  we again generate the two more basis elements $g_1\uu=\vv$ and $g_2\vv=\ww$ in the induced module. Using relations  in the module  $\StJTL{1}{z_2}$, 
we have $u^{(2)}\uu = z_2 \uu$ and on the other hand
  $u^{(2)}\uu=i\q^{1/2} x_2\uu$ (because $u^{(2)}=x_2g_2$) and therefore $x_2\uu=-i\q^{-1/2}  z_2\uu$, while $x_1\uu=z_1\uu$. The defining relations lead, in the $\uu,\vv,\ww$ basis, to 
\begin{equation}
x_1=\left(\begin{array}{ccc}
z_1&z_2(\q-\q^{-1})&i\q^{1/2} z_2(\q-\q^{-1})\\
0&i\q^{1/2} z_2&0\\
0&0&i\q^{1/2} z_2
\end{array}\right)
\end{equation}
and similar matrix expressions for the other generators.
 The generators $g_1$ and $g_2$ give rise to generators $e_1,e_2$ that satisfy the Temperley--Lieb relations:
\begin{equation}
e_1=\left(\begin{array}{ccc}
\q&-i\q^{-1/2}&0\\
i\q^{1/2}&\q^{-1}&0\\
0&0&0\end{array}\right),\qquad e_2=\left(\begin{array}{ccc}
0&0&0\\
0&\q&-i\q^{-1/2}\\
0&i\q^{1/2}&\q^{-1}\end{array}\right)
\end{equation}
One also finds $u^3=z_1z_2^2\one$. However, the $\ATL{3}$ relation $u^2 e_2=e_1e_2$ implies the constraint $z_1 z_2=i\q^{1/2}$, and so $u^3= i\q^{1/2} z_2\one$. This leads to the three dimensional module  $\StJTL{\half}{i\q^{1/2}z_2}[3]=\StJTL{\half}{-\q z_1^{-1}}[3]$, or  in other words, we have shown that  the ideal  $\I$ from~\eqref{TH} acts by zero  only at  $z_1=i\q^{1/2}z_2^{-1}$.

We can also   consider the case where the ideal $\I$ generates an invariant subspace. It happens indeed when $\vv$ is proportional to $\uu$, which implies that $\ww$ is proportional to $\uu$ as well. After taking the corresponding quotient, one finds then $e_1=e_2=0$ and $u=-\q z_1\one$, while $z_2=-i\q^{3/2}z_1$. It follows that the quotient is the (one-dimensional) standard module $\StJTL{3/2}{-\q z_1}=\StJTL{3/2}{-i\q^{-1/2} z_2}$. 

We thus conclude that the affine TL fusion is zero at all values of $z_1$ and $z_2$ except the following  cases:
\begin{align}
\StJTL{\half}{z_1}[1]\afus\StJTL{1}{z_2}[2]&=\StJTL{\half}{i\q^{1/2}z_2}[3],\qquad z_1=i\q^{1/2}z_2^{-1}\ ,\\
\StJTL{\half}{z_1}[1]\afus\StJTL{1}{z_2}[2]&=\StJTL{3/2}{-i\q^{-1/2} z_2}[3],\qquad z_1=i\q^{-3/2} z_2\ ,
\end{align}
which is  in agreement with~\eqref{afus-ex2}.

\end{document}